\documentclass[11pt]{amsart} 
\usepackage{color}
\usepackage[english]{babel} 

\usepackage{amsmath,amsfonts,amsthm} 

\usepackage{fancyhdr} 
\usepackage{verbatim} 
\usepackage{hyperref} 
\hypersetup{          
	colorlinks=true, breaklinks, linkcolor=[rgb]{0.15 0 0.8}, 
	filecolor=magenta, urlcolor=blue,
	citecolor=magenta, linktoc=all, }
\usepackage[nameinlink]{cleveref}

\makeatletter
\def\l@subsection{\@tocline{2}{0pt}{2.5pc}{5pc}{}}
\makeatother
\usepackage{bm} 
\usepackage{graphicx} 
\usepackage[font=footnotesize,labelfont=bf]{caption} 
\usepackage{tikz-cd} 

\usepackage{float} 
\graphicspath{{./figures/}}
\numberwithin{equation}{subsection} 
\numberwithin{figure}{subsection} 
\numberwithin{table}{subsection} 
\usepackage{epstopdf}

\title{Plane curve singularities via divides} 
\author{Norbert A'Campo and Pablo Portilla Cuadrado} 

\subjclass[2010]{}

\date{\normalsize\today} 

\usepackage{enumitem}

\usepackage{aliascnt}

\usepackage[all]{xy}

\usepackage{pifont}

\usepackage{setspace}

\newtheorem{theorem}[equation]{Theorem}

\newtheorem{lemma}[equation]{Lemma}

\newtheorem{corollary}[equation]{Corollary}

\theoremstyle{definition}

\newtheorem{example}[equation]{Example}

\newtheorem{definition}[equation]{Definition}

\newtheorem{notation}[equation]{Notation}

\newtheorem{remark}[equation]{Remark}

\usepackage[bindingoffset=0in,left=1.3in,right=1.3in,top=1.5in,bottom=1.5in,
footskip=.85in]{geometry}
\thanks{This publication has been partially funded by the grant (holded by the 
	second author)
	RYC2022-035158-I, funded by MCIN/AEI/10.13039/501100011033 and by
	the FSE+}
\begin{document}

\title{Plane curve singularities via divides}

\maketitle

\tableofcontents

\begin{abstract}	
	Generic relative  immersions of compact one-manifolds in the closed unit 
	disk, 
	i.e. divides, provide a powerful combinatorial framework, and allow a 
	topological construction of fibered classical links, for which the 
	monodromy 
	diffeomorphism is explicitly  given as a product of Dehn twists. Complex 
	isolated plane curve singularities
	provide a classical fibered link, the Milnor fibration, with
	its Milnor monodromy, monodromy group, and vanishing cycles.
	This surveys puts together much of the work done on divides and their role 
	in 
	the topology of isolated plane curve singularities.  
	We review two complementary approaches for constructing divides: one via 
	embedded 
	resolution techniques and controlled real deformations, and another via 
	Chebyshev polynomials, which yield explicit real morsifications.
	A combinatorial description of the Milnor fiber is developed, leading to an 
	explicit factorization of the geometric monodromy as a product of 
	right-handed Dehn twists. We further explore the structure of reduction 
	curves 
	that arise from the Nielsen description of quasi-finite mapping classes and 
	from iterated cabling operations on divides.
	The interplay between 
	the geometric and integral 
	homological monodromies is analyzed, with special attention to symmetries 
	induced by complex conjugation and strong invertibility phenomena. In 
	particular, the integral homological monodromy
	for isolated plane curve singularities can be computed effectively. In 
	contrast, for complex hypersurface singularities in higher dimensions no 
	method 
	of computation of the integral homology monodromy is known.
	Connections with mapping class groups, contact and symplectic geometry, and 
	Lefschetz fibrations are also discussed.	
	We conclude by outlining several open problems and conjectures related to 
	the characterization of divides among fibered links, the presentation of 
	geometric monodromy groups, and the existence of symplectic fillings 
	compatible with the natural fibration structures.
\end{abstract}

\section{Introduction}

The concept of divides (or partage in French) has become a significant tool in 
the study of isolated plane curve singularities and their associated monodromy 
groups. Introduced in the 1970s, independently by the first author 
\cite{Acampo_groupe_I} and Sabir 
Gusein-Zade \cite{GZ_dynkin_two,GZ_inter_two}, divides provided a novel 
combinatorial approach to understanding 
the topology of singularities and their deformations. Both researchers were 
deeply influenced by the broader developments in singularity theory initiated 
by John Milnor’s foundational work on the Milnor fibration 
\cite{Milnor}, which showed that isolated hypersurface 
singularities could be understood through their associated fibration structure. 
Divides emerged as a geometric tool that could encode the topological behavior 
of these singularities by associating them with fibered links in 
three-dimensional spaces.

The first author's pioneering work on divides tied tightly these objects to the 
monodromy 
of singularities. His work throughout the years drew a big picture that placed 
these objects as an useful and novel tool to understand plane curves. In 
particular, he showed that a divide associated
with a plane curve singularity encoded every topological piece of information 
associated with the plane curve singularity: the link, a model of the Milnor 
fiber, the geometric monodromy 
as a product of Dehn twists around vanishing cycles and even the reduction 
curves of the Nielsen-Thurston decomposition of the monodromy. This work also
demonstrated that divides could serve as a bridge between knot theory and the 
study of singularities, particularly by offering a combinatorial framework for 
understanding the monodromy representation of some fibered link that share a 
lot of similarities with those coming from singularity theory.

In parallel, Gusein-Zade developed an alternative approach to divides, focusing 
on the real morsifications of singularities and using Chebyshev polynomials to 
generate divides for specific classes of real plane curves. This approach 
offered a more algebraic route, 
allowing the explicit construction of divides from polynomial models. 
Gusein-Zade’s contributions enriched the theory of divides, showing that they 
could be systematically constructed for a wide variety of singularities, 
further connecting the geometric properties of divides with algebraic 
invariants like Puiseux pairs and Newton polygons.

Divides also exhibit 
strong symplectic properties, connecting them to the study of symplectic 
geometry and symplectic fillings. Research into the symplectic properties of 
divides has led to advances in the understanding of how these objects interact 
with higher-dimensional spaces, further expanding their applications.

This survey explores the modern theory of divides, focusing on their 
combinatorial and geometric properties, as well as their applications to the 
study of singularities. We discuss the role of divides in generating fibered 
links, describe their connections to the Milnor fiber and geometric monodromy, 
and highlight recent advances in the symplectic and topological study of these 
objects. Additionally, we state and discuss some open questions, such as 
the broader symplectic 
properties of divides. 

Through this work, we aim to provide a comprehensive 
overview of the current state of the theory of divides, demonstrating its 
central role in singularity theory, knot theory, and beyond. Much of the 
content comes directly from the classical work of the first author. Some proofs 
have been revisited or expanded with some details added and new detailed 
examples are considered. The order in which the material is told, does not 
necessarily respect the chronological order in which it appeared but rather an 
expository one.

\subsection*{Organization of the paper}

In Sect.~\ref{s:preliminary_theory} We fix notation and introduce some basic 
concepts about the theory of mapping class groups and the theory of plane 
cruves. Among others we introduce the Nielsen-Thurston decomposition and the 
important concepts of geometric vanishing cycle and geometric monodromy group 
associated with an isolated plane curve singularity.

In Sect.~\ref{s:definitions} we give the abstract definition of divide and we 
associate to it a link in the three-sphere. Furthermore, in 
Lemma~\ref{lem:existence_adpated_function} we prove the existence 
of adapted functions (Morse functions that define the divide) for each divide.

In Sect.~\ref{s:fibration_thm} we prove a general fibration theorem 
(Theorem~\ref{thm:fibered_link}): the link that we associated to each divide is 
a 
fibered link. The fibration is given by a very explicit map which can be 
constructed from an adapted function.

In Sect.~\ref{s:divides_plane_curves} we start the study of divides coming from 
plane curve singularities. In particular we introduce the concept of totally 
real plane curve singularity and we explain two methods to produce divides for 
totally real plane curves. The first method, presented in 
Sect.~\ref{ss:divides_embedded}, is due 
to the first author and relies on choosing a resolution and then perform a 
series of controlled perturbations of the strict transform and contractions of 
exceptional divisors with self intersection $-1$ alternately. The second 
method, explained in Sect.~\ref{ss:divide_chebyshev} is due to Gusein-Zade and 
it 
relies on the fact that Chebyshev polynomials can be used to parametrize the 
$0$ set of Brieskorn-Pham polynomials. This fact together with an iterative 
technique, allows one to produce divides for any topological type of plane 
curve realized by a particular real model.

In Sect.~\ref{s:description_Milnor} we describe combinatorially the Milnor 
fiber of 
the function constructed in Sect.~\ref{s:fibration_thm}. This combinatorial 
description is crucial in the description of the monodromy given in 
Sect.~\ref{s:description_monodromy}. In the first subsection we describe the 
monodromy as a product of right handed Dehn twists around the geometric 
vanishing cycles of a distinguished basis. In the second subsection we deal 
with special properties of the geometric monodromy associated with a plane 
curve singularity, more concretely we explore the property of being strongly 
invertible and we discuss the state of the art in higher dimensions.

In Sect.~\ref{s:abstract}, we explain how one can, from a divide, to easily 
produce 
an explicit model for the Milnor fiber together with a set of geometric 
vanishing cycles associated with a distinguished basis. Furthermore, if the 
divide is produced by the previously introduced cabling technique, we explain 
how to visualize in this model, the reduction curves of the geometric monodromy 
associated with the divide.

Finally, in Sect.~\ref{s:other_questions} we state a few open questions related 
to 
divides associated with plane curves together with some other properties that 
are not fully explored in this work.

\section{Preliminary theory on mapping class groups and plane curves}
\label{s:preliminary_theory}

We turn now our attention to the theory of mapping class groups. The purpose of 
this section it to collect the results on mapping class group that we use 
throughout the rest of the text as well as  to fix notations and conventions. 
For more insight on this topic we refer to the book by B. Farb and D. Margalit 
\cite{Farb}.

\subsection{Mapping class groups}
\index{mapping class group}
Let $F$ be a compact oriented connected surface with boundary.

\begin{definition}\label{def:mapping_class_group}
	The {\em mapping class 
		group} of $F$ is defined as
	\[
	\mathrm{Mod}(F) := \pi_0(\mathrm{Diff}^+(F, \partial F)).
	\]
	That is, elements of $\mathrm{Mod}(F)$ are isotopy classes of oriented 
	diffeomorphisms of $F$ 
	that fix the boundary point-wise. Each of these classes is called a {\em 
		mapping class}.
\end{definition}

Relative diffeomorphisms $\alpha,\beta$ of a surface 
$(F,\partial F)$  are in fact in the same mapping class if relatively homotopic.

Next we define the basic elements that form the mapping class group: the Dehn 
twists.

\begin{notation}\label{not:preliminaries}	
	Given a simple closed curve $\gamma \subset \Sigma$ disjoint from the 
	boundary on an oriented surface we 
	denote by $$T_\gamma: 
	\Sigma \to \Sigma$$ a {\bf right-handed Dehn twist} around $\gamma$ or its 
	mapping class in the mapping class group of $\Sigma$ relative to its 
	boundary. 
	The support of $T_\gamma$ is concentrated, by definition, in a 
	tubular neighborhood of $\gamma$. Also, the mapping class of $T_\gamma$ 
	only depends on the isotopy class of the simple closed curve $\gamma$. 
	\index{Dehn twist}
	\begin{figure}[h]
		\centering \includegraphics*[scale=0.6]{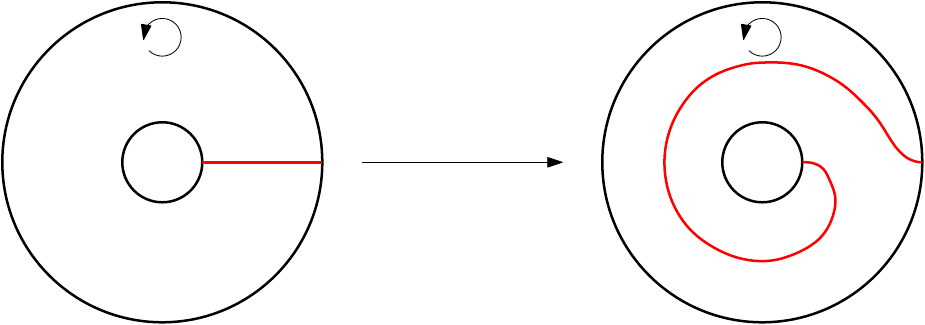}
		\caption{A right-handed Dehn twist acting on a properly embedded 
			segment. The orientation of the annulus is indicated by the 
			curved arrows.}
		\label{fig:r_dehn}
	\end{figure}
	
\end{notation}

Now we turn to our attention to another construction, very much related to the 
one above, that will turn out to be very useful later on this text. We 
introduce the notions: minimal positive pair of Dehn twists, half twist, right 
half Dehn twist and left half Dehn twist.

\begin{definition}\label{def:half_twist}
	Let $F$ be an oriented surface and let $z:S^1 \to F$  be a simple closed 
	parametrized 
	curve on $F$.
	Let $F'$ be the surface obtained from the surface $F$ by cutting $F$ along 
	the image $Z$ of
	$z$ and by gluing back with the diffeomorphism that identifies images of 
	opposite 
	points of $S^1$; denote by $Z'\subset F'$ the image of $Z$; we call this 
	diffeomorphism a {\em half twist}. The 
	surfaces 
	$F$ and $F'$ are of course diffeomorphic and $F\setminus Z$ and 
	$F'\setminus Z'$ are equal as sets. A {\em minimal positive pair 
		of Dehn twists} from $F$ to $F'$ is a pair of 
	diffeomorphisms $(p,q)$ from $F$ to $F'$ such that the following holds:
	
	\begin{enumerate}[label=(\alph*)]
		\item The composition $q^{-1} \circ p:F \to F$ is a right Dehn twist 
		with respect to the orientation of $F$ having 
		the curve $Z$ as core. In addition $p(Z)=q(Z)=Z'$ holds
		
		\item There exists a regular collar neighborhood $N$ of $Z$ in $F$ 
		such that  both $p$ and $q$ coincide with the identity of $F\setminus 
		Z=F'\setminus Z'$
		outside $N$.
		
		\item For some volume form $\omega$ on $N,$ which we think of as a 
		symplectic structure,  we have 
		$p^*\omega=q^*(\omega)=\omega$, and the sum of the Hofer  distances 
		(\cite[Chapter 5.]{Hof_Zeh})
		to the identity of the restrictions 
		of $p$ and $q$ to  $N \setminus z$ is minimal. 
	\end{enumerate}
	Minimal positive pairs of Dehn twists exist and are well defined up to 
	isotopy. \index{Dehn twist!half}
	For a minimal positive pair $(p,q)$ of 
	Dehn twists, the member $p$ is called {\em positive or 
		right half Dehn twist} and the member $q$ is called {\em negative or 
		left half 
		Dehn twist}.  
\end{definition}

\subsubsection*{Nielsen-Thurston decomposition}
\index{Nielsen-Thurston decomposition}
The following decomposition result is a landmark in mapping class group theory.

\begin{theorem}[See \cite{Thu} and Corollary 13.3 from \cite{Farb} 
	]\label{thm:nt_classification}
	Let $\phi: \Sigma \to \Sigma$ be an orientation preserving homeomorphism 
	that 
	restricts to the identity on $\partial \Sigma$. Then there exists $\phi'$ 
	isotopic to $\phi$ and a collection $\mathcal{C}$ of non-null-homotopic 
	disjoint 
	simple closed curves (called \index{reduction curves} {\em reduction 
		curves}) including all boundary components such that:
	\begin{enumerate}
		\item \label{it:invariant} The collection of curves is invariant by 
		$\phi'$, i.e. 
		$\phi'(\mathcal{C})= \mathcal{C}$.
		\item \label{it:per_ano}The homeomorphism $\phi'$ restricted to the 
		union of
		components 
		of each $\phi'$-orbit of connected components of
		$\Sigma \setminus \mathcal{C}$  is isotopic
		either to a periodic or to a pseudo-Anosov homeomorphism.
	\end{enumerate}
\end{theorem}

\begin{definition} \label{def:reduction_system}
	Any collection of curves $\mathcal{C}$ satisfying \ref{it:invariant} above 
	is 
	called 
	a 
	{\em reduction system}. That is, a collection of simple closed curves that 
	is 
	invariant (up to isotopy) by $\phi$ is a reduction system. 
\end{definition}

The decomposition given by Theorem~\ref{thm:nt_classification} is called {\em 
	Nielsen-Thurston decomposition} and is unique up to isotopy if the 
reduction 
system of  
curves $\mathcal{C}$ is minimal by inclusion. We  assume whenever necessary 
that 
the 
representatives of the mapping classes used in this work satisfy 
\ref{it:invariant} and 
\ref{it:per_ano} from the previous 
theorem. This decomposition leads to the following definition.

\begin{definition}\label{def:nt_decomposition}
	Let $\phi: \Sigma \to \Sigma$ be a representative of a mapping class, we 
	say that 
	$\phi$ is {\bf pseudo-periodic} if only 
	periodic pieces appear in its Nielsen-Thurston decomposition.
\end{definition}

Pseudo-periodic homeomorphisms are of special importance in complex singularity 
theory, as all geometric monodromies of  holomorphic germs of functions 
on isolated complex surface singularities are of this kind. In particular, 
monodromies of isolated plane curve singularities are of this type. In general, 
local geometric monodromies of complex hypersurface singularities are 
dynamically restricted: the geometric monodromy can be realized by a distal 
map, which topological entropy vanishes. This  shows that in the case of curve 
singularities the reduction of the geometric monodromy does not have 
pseudo-Anosov components. See also \cite{Acampo_fonction}. 

\subsection{Isolated plane curve singularities}

In this section we review some basic notions on plane curve singularities that 
are mentioned and used throughout this text.

The algebra of convergent power series $\mathbb{C}\{x,y\}$ is a unique 
factorization 
domain and so, up to multiplication by a unit, a series $f$ defining an 
isolated plane curve singularity, can be uniquely expressed as 
$f=f_1\cdot \cdots \cdot f_r$ where each $f_i$ is an irreducible convergent 
power series. Each $f_i$ is called a {\em branch} of $f$ and if $r=1$ we say 
that $f$ is {\em irreducible}.

\subsubsection*{The \index{Milnor fibration} Milnor 
	fibration}\label{ss:milnor_fibration} In 
\cite{Milnor}, Milnor proved that for 
a 
holomorphic 
map $f:\mathbb{C}^{n+1} \to \mathbb{C}$ with an isolated singularity at the 
origin, 
$$\frac{f}{|f|}: S^{2n+1}_\epsilon \setminus K \to S^1$$ is a locally trivial 
fibration. Here $S_\epsilon^{2n+1} \subset \mathbb{C}^{n+1}$ denotes a suitably 
small 
sphere and the {\em link} $K$ is defined by $K:= f^{-1}(0) \cap  
S_\epsilon^{2n+1}$. If we denote by $D_\delta \subset \mathbb{C}$ a small disk 
of 
radius $\delta$ centered at $0$ and by $B_\epsilon \subset \mathbb{C}^n$ a ball 
of 
radius $\epsilon$, it  follows from Ehresmann's fibration lemma that
\begin{equation}\label{eq:milnor_fibration}
	f_{|_{f^{-1}(\partial D_\delta) \cap B_\epsilon}}: f^{-1}(\partial 
	D_\delta) \cap B_\epsilon \to \partial D_\delta
\end{equation} is also a locally trivial fibration for $\epsilon$ small enough 
and $\delta$ small with respect to $\epsilon$. In this case, we say that 
$B_\epsilon$ is a {\em Milnor ball} and that $\epsilon$ is a Milnor radius. 
Milnor proved in 
\cite{Milnor} that these two fibrations are essentially equivalent. 

The fibers of second fibration are connected compact $2 n$-manifolds with 
non-empty boundary $(F,\partial{F})$, and the fibers of first fibration are 
diffeomorphic to the interior $F\setminus \partial{F}$ of the fibers of second 
fibration which carry a complex structure.

Any fiber $F$ of these fibrations is called
{\em the Milnor fiber} of $f$.  When $n = 1$, the second is a connected 
oriented 
compact surface with 
non-empty boundary. The \index{Milnor fiber} Milnor fiber has $r$ boundary 
components, where $r$ is 
the number of branches of $f$. Its first Betti number $b_1(F)$ coincides 
with 
\[
\dim_\mathbb{C} \frac{\mathbb{C}\{x,y\}}{(\partial f/ \partial x, \partial f/ 
	\partial y)}
\]
and 
any of these quantities is called {\em the Milnor number of } $f$ and denoted 
by $\mu_f$ or simply by $\mu$ if there is no ambiguity. The topological 
information of a plane curve singularity is carried by the $r$-component 
oriented classical link $\partial F \subset \partial B_\epsilon$ in its 
oriented ambient $3$-sphere.

\subsubsection*{Puiseux pairs and intersection multiplicities} 
\label{ss:puiseux_pairs}
There are a lot 
of different ways of codifying by numerical invariants the topological 
information of a plane curve 
singularity. 
We cite \cite{Bri}, \cite{Wall} or \cite{Eis} as standard references on this 
topic. 

Next, we briefly recall what \index{Puiseux pairs} Puiseux 
pairs are, since they appear several times on this work. In particular, Puiseux 
pairs codify topologically uni-branch plane curve singularities. We give an 
axiomatic 
approach to this theory.

\begin{definition}
	A finite sequence $Pp$ of pairs of integers $(a_1,b_1),\ldots,(a_k,b_k)$ is 
	a 
	sequence of {\em essential Puiseux pairs} if and only if 
	\[
	\begin{split}
		& 2 \leq a_i < b_i,  \\
		& b_i/(a_1a_2\cdots a_ i) < b_{i+1}/(a_1a_2\cdots a_{i+1}) \text{ and 
		}\\
		&\gcd(b_i, a_1\cdots a_i)= 1
	\end{split}
	\]
	for all $i =1, \ldots,k$.
\end{definition}

A sequence of essential Puiseux pairs $Pp$ defines a family
of topologically equivalent singularities. A specific member $f_{Pp}(x,y)$
of this family is
obtained from a  Puiseux expansion with fractional and strictly increasing
exponents
$$
y=x^{b_1/a_1}+x^{b_2/a_1a_2}+ \dots +x^{b_k/a_1a_2 \cdots a_k}
$$
by the rule, which takes into account the ramification of $x^{1/a_1a_2\cdots 
	a_k}$,
$$
f_{Pp}(x,y)=\prod_{\theta}(y-\theta^{a_2 \cdots a_kb_1}-\theta^{a_3 \cdots 
	a_kb_2}- 
\dots -\theta^{b_k}),
$$
\noindent
where $\theta$ runs 
over the 
$a_1a_2 \cdots a_k$ roots 
of $z^{a_1a_2 \dots a_k}-x=0$ in the algebraic
closure of the field $\mathbb{C}((x))$. The 
coefficients of the polynomial $f_{Pp}(x,y)$ are integers. 

For example,
the Puiseux 
expansion $y=x^{3/2}+x^{7/4}$ with strictly increasing exponents for 
$Pp=((2,3),(2,7))$
leads to the polynomial $f_{(2,3),(2,7)}=(y^2-x^3)^2-4x^5y-x^7$ and
the Puiseux expansion $y=x^{3/2}+x^{11/6}$ to the polynomial
$f_{(2,3),(3,11)}=(y^2-x^3)^3-6x^7y^2-2x^{10}-x^{11}$.

There is a second equivalent set of numerical invariants that are more suitable 
for 
certain matters: {\em Newton pairs}. These are a finite sequence $Np$ of 
coprime 
integers 
$(a_i, \lambda_i)$ that can be computed from the Puiseux pairs by the recursive 
formula 
\begin{equation}\label{eq:newton_pairs}
	\begin{split}
		& \lambda_1:=b_1,  \\
		& \lambda_{i+1}:= b_{i+1} -b_i a_{i+1} + \lambda_i a_{i+1} a_ i .
	\end{split}
\end{equation}

Instead of using the above Puiseux expansion one uses for  given Newton pairs 
$Np$ Newton's Ansatz
$$
y=x^{\lambda_1/a_1}(1+x^{\lambda_2/{a_1a_2}}(1+x^{\lambda_3/{a_1a_2a_3}}(\ldots)))$$
which expands to the corresponding Puiseux expansion.

Now suppose that $f= f_1\cdots f_r$ is the factorization of a reducible 
isolated singularity $f$ in branches $f_i$. In this case, 
Puiseux or Newton pairs of each of the branches are not enough to determine the 
topology of the singularity. One need one extra piece of information,  for 
example {\em intersection multiplicities}. For 
each pair $i,j \in \{1,\ldots,r\}, i \neq j$, we define the {\em intersection 
	multiplicity} between the branches $f_i$ and $f_j$ as 
$$\dim_{\mathbb{C}}\mathbb{C} 
\{x,y\}/(f_i,f_j)$$
which is in fact the linking number of the oriented knots corresponding to 
these branches. 
For an interpretation of the numbers $\lambda_i$ of 
eq.~\ref{eq:newton_pairs} in terms of intersection numbers, see 
Sect.~\ref{ss:divide_chebyshev}.

For instance, Puiseux pairs and Newton pairs can be truncated: denote by $Pp_j$ 
or $Np_j, 0\leq j \leq r,$ the initial $j$ entries of given equivalent $Pp,Np$ 
and let $f_{Pp_j}(x,y)$ define a sequence of plain curve singularities. Here, 
$Pp_0,Np_0$ will be the empty set of \index{Puiseux pairs} Puiseux pairs, 
Newton pairs corresponding 
to  $f_\emptyset(x,y)=y$. The equations $f_{Pp_j}(x,y)=0, j=0,\ldots ,r,$ 
define oriented knots  $K_j\subset S^3_\epsilon$. The Newton pairs get a 
topological interpretation: The knot $K_{j}$ is the $(a_j,\lambda_j)$ cable 
knot on the knot $K_{j-1}, j=1,2,\ldots ,r$. See the book by David Eisenbud and 
Walter Neumann \cite{Eis} for references 
and historical remarks.

\subsection{The versal deformation space and the geometric monodromy 
	group}\label{subsection:versal}
\ 
\index{versal unfolding}
\subsubsection*{The versal unfolding} We briefly recall here the notion 
of the 
{\em versal unfolding} of an isolated singularity; see \cite[Chapter 
3]{ArnII} for more details. Recall the algebra 
\[
A_f = \frac{\mathbb{C}\{x,y\}}{(\partial f/ \partial x, \partial f/\partial y)};
\]
Let 
\[
g_1, \dots, g_\mu \in \mathbb{C}[x,y]
\]
be polynomials that project to a basis of $A_f$, assume (we can always do so) 
that $g_1=1$. For $\lambda = 
(\lambda_1, \dots, \lambda_\mu) \in \mathbb{C}^\mu$, define the function 
$f_\lambda$ by
\[
f_\lambda = f + \sum_{i = 1}^\mu \lambda_i g_i.
\]
The {\em base space of the versal unfolding } of $f$ is the parameter space of 
all $\lambda$ which is naturally isomorphic to $\mathbb{C}^\mu$. The {\em 
	discriminant 
	locus} is the 
subset
\[
\mbox{Disc}  = \{\lambda \in \mathbb{C}^\mu \mid f_\lambda^{-1}(0) \mbox{ is 
	not 
	smooth}\}.
\]
It can be shown that $\mbox{Disc}$ is an algebraic hypersurface. The 
discriminant locus admits different stratifications; for example, it is 
stratified according to the sum of  the Milnor numbers of the singular points 
of the singular fibers lying above. In this sense, the top-dimensional stratum 
(that is, 
the smooth part of $\mbox{Disc}$)
parameterizes curves with a single node.  Denote by $V_f$ a small closed ball 
in $\mathbb{C}^\mu$ centered at the origin.  Define
\begin{equation}\label{eq:taut}
	X_f = \{(\lambda, (x,y)) \mid (x,y) \in f_\lambda^{-1}(0),\ \lambda \not 
	\in \mbox{Disc}\}.
\end{equation}
Then, for $V_f$ small enough and after intersecting $X_f$  with a 
sufficiently small closed polydisk, this family has the structure of a smooth 
surface bundle with base $V_f \setminus \mbox{Disc}$ and fibers diffeomorphic 
to the Milnor fiber $F$ of eq.~\ref{eq:milnor_fibration}.
We fix a point in $V_f \setminus \mbox{Disc}$ and we denote, also by $F$, 
the fiber with boundary lying over it.

\subsubsection*{The geometric monodromy group} 
\index{geometric monodromy}
\index{geometric monodromy!group}
Recall that  a mapping class  (Definition~\ref{def:mapping_class_group}) is an 
isotopy 
class of orientation preserving 
diffeomorphisms of $F$ that restricts to the identity on the boundary, where 
the isotopies are required to fix the boundary {\em point-wise}. Let 
$f:\mathbb{C}^2 
\to \mathbb{C}$ define an isolated singularity at the origin, with 
\index{Milnor fiber} Milnor fiber 
$F$. 

\begin{definition}\label{def:geom_monodromy_group}
	The {\em geometric monodromy group} is the image in $\mathrm{Mod}(F)$ of 
	the 
	monodromy representation
	\[
	\rho: V_f \setminus \mbox{Disc} \to \mathrm{Mod}(F)
	\] of the universal family $X_f$ of eq.~\ref{eq:taut}.
\end{definition}

\begin{definition}\label{def:vanishing_cycle}
	A \index{geometric vanishing cycle} {\em quadratic vanishing cycle} or {\em 
		geometric vanishing cycle} or, for the purposes of this work, 
	simply a {\em 
		vanishing cycle}  is a 
	simple closed curve $c \subset F$ that 
	gets contracted to a point when transported to the nodal curve lying over a 
	smooth point of the discriminant $\mbox{Disc}$.
\end{definition}

Let $B_\epsilon$ be a Milnor ball for an isolated plane curve singularity $f$. 
Then, for any linear form $\ell:\mathbb{C}^2 \to \mathbb{C}$ generic with 
respect to $f$ and 
$\eta>0$ small enough, the map $$\tilde{f}:=f+\eta \ell: B_\epsilon
\to 
\mathbb{C}$$ 
only has Morse-type singularities and the corresponding critical values $c_1, 
\ldots, c_\mu$ are all distinct and close to $0\in \mathbb{C}$. The holomorphic 
map 
$\tilde{f}$ is usually called a {\em morsification} of $f$.

\begin{remark}\label{rem:geom_invariant}
	Two isolated plane curve 
	singularities with the 
	same topological type can be connected by a $1$-parameter $\mu$-constant 
	family 
	and thus, they have the same geometric monodromy group. This last statement 
	follows from the following: in \cite[Theorem 
	3]{GZ_dynkin_two} it is proven that two plane curve singularities with the 
	same 
	topological type can be joined by a $\mu$-constant family or, equivalently, 
	by   
	family of singularities with the same 
	topological type. As a consequence, the intersection matrices (of 
	distinguished 
	basis) of 
	these singularities coincide. Also, by a result of Hamm and L{\^e} 
	\cite{Hamm_Le}, the geometric monodromy group o a plane curve singularity 
	can 
	be computed from any morsification. As a consequence of this discussion, 
	the geometric monodromy group (Definition~\ref{def:geom_monodromy_group}) 
	of an 
	isolated plane curve singularity is a topological invariant.
\end{remark}

\section{Definition of divide and some combinatorial properties}
\label{s:definitions}

In this section we give the first definitions and properties of 
divides. In particular, we prove the existence of {\em adapted functions} to 
divides (Lemma~\ref{lem:existence_adpated_function}) which are \index{Morse 
	function} Morse functions
with 
very particular properties.

\subsection{Definition of divide}
We start with the definition of the central object of this work.

\begin{definition}\label{def:divide}
	A {\em \index{divide} divide} is a generic relative immersion $\alpha:J \to 
	D$ of 
	a 
	compact $1$-dimensional 
	manifold $J$ into a disk $D$. Usually $D$ is a disk in the Euclidean plane 
	$\mathbb{R}^2$ or the Gaussian plane $\mathbb{C}$.
\end{definition}

\begin{notation}
	A divide is usually identified with the image of the immersion and denoted 
	by 
	$P$, that is, with the notation of the above definition $P = \alpha(J)$.
\end{notation}

\begin{remark}
	\begin{itemize}
		
		\item {\em relative} means that $\alpha(\partial J) \subset \partial D$ 
		and 
		$\alpha^{-1}(\partial D)=\partial J$. 
		
		\item {\em generic} means that the restriction of $\alpha$ to $\partial 
		J$ is 
		injective, that $\alpha(J)$ is transverse to $\partial D$, and that the 
		curve 
		$\alpha(J)$
		has only ordinary double points in the interior of $D$
		(this leads to the following Definition~\ref{def:pre_divide})
		
		\item When not important, the radius of the disk will not be specified. 
		If we 
		need to specify the radius of the disk, we will write $D_\rho$ for a 
		disk of radius $\rho$ centered at the origin. 
		
		\item	In certain parts of this survey, we only consider divides which 
		are 
		immersions of disjoint union of segments and, in other parts, we also 
		consider divides which are immersions of disjoint union of segments and 
		circles. It will be pointed out when necessary. If no clarification is 
		made, we assume the second and more general version.
		
	\end{itemize} 
\end{remark}

It will be very useful to deal with a more {\em relaxed} version of the notion 
of divide which is one where we allow ordinary singularities other than double 
points.

\begin{definition}\label{def:pre_divide}
	\index{pre-divide}
	A {\em pre-divide}  is a relative immersion $\alpha:J \to D$ that is 
	generic near $\partial D$ as above and such that for all distinct $p,q \in 
	J$ with
	$\alpha(p)=\alpha(q)$ the images of the tangent spaces 
	\[
	(D \alpha)_p(T_pJ) 
	\neq (D \alpha)_q(T_qJ)
	\] 
	in $T_{\alpha(p)}D$. 
\end{definition}
Note that the only difference between a divide and a pre-divide is that in a 
pre-divide we admit singularities which are multiple crossing points (see 
Fig.~\ref{fig:predivide_divide}).

\begin{figure}
	\centering \includegraphics*[scale=0.6]{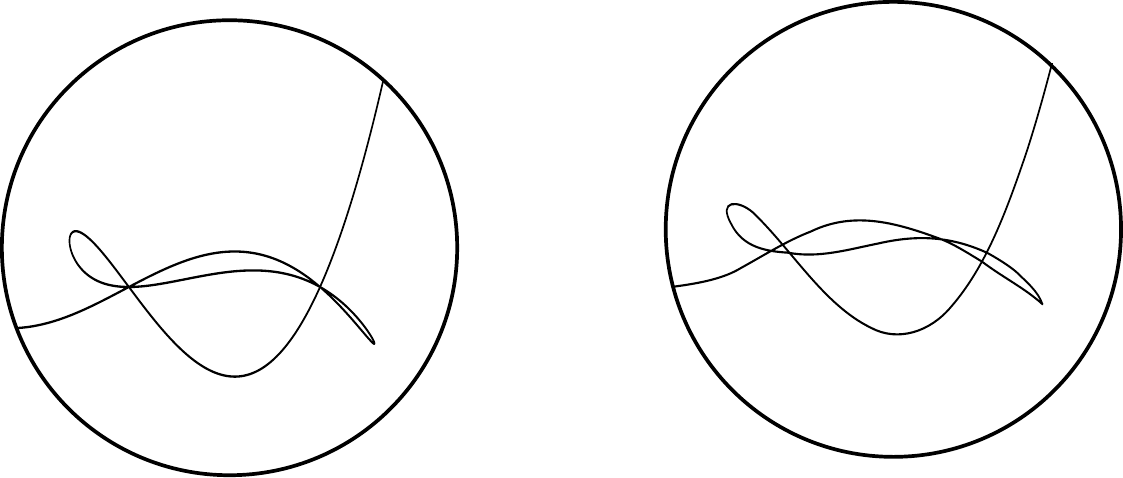}
	\caption{On the left we see a divide where there are two ordinary 
		singularities of order $3$. On the right we see a divide that is obtain 
		from the first one after a small perturbation.}
	\label{fig:predivide_divide}
\end{figure}

We need to prepare the construction of a classical link $L(P)$ in the 
$3$-sphere from a divide. Let $D=D_\rho$ be a disk (for convenience centered at 
the origin) of radius $\rho$ in the Euclidean plane $\mathbb{R}^2$ with norm 
$||x||^2=|x_1|^2+|x_2|^2$. We equip the tangent space 
$T\mathbb{R}^2=\mathbb{R}^2\times\mathbb{R}^2$ with the Euclidean norm
$||(x,u)||^2=|x_1|^2+|x_2|^2+|u_1|^2+|u_2|^2$, where we denote the tangent 
vector $u\in T_x\mathbb{R}^2$ by $(x,u)$. The $4$-ball $B^4_D$ of radius $\rho$ 
over $D$ is the subset in $TD$ defined by 
$B^4_D=\{(x,u)\in 
TD \mid ||(x,u)||\leq \rho \}$. The $3$-sphere $S_D^3$ that will contain the 
link $L(P)$ is the boundary $\partial{B^4_D}$. Observe that 
$\partial{D_\rho}\times\{0\}$ is identified with the set of tangent vectors  
$\{(x,0)\in TD_\rho \mid ||x||=\rho\}$. Let $P$ be a divide or pre-divide in 
$D$. We denote by $T_DP$ the set of all vectors $(x,u)\in TD$ with $x\in P, 
u\not= 0,$ or
$x\in P\cap \partial{D}$ that are tangent to one of the branches of 
$P$.\newline\indent 
In view of the study of complex \index{plane curve singularity} plane curve 
singularities, we think 
$T\mathbb{R}^2$ as the complex plane $\mathbb{C}^2$ via the identification 
$$(x,u)\in T\mathbb{R}^2=\mathbb{R}^2\times\mathbb{R}^2 \mapsto 
(x_1+u_1i,x_2+u_2i)\in \mathbb{C}^2$$ 
and choose to orient $T\mathbb{R}^2$ by the ``complex'' orientation given by 
the 
frame $(\frac{\partial}{\partial x_1} ,\frac{\partial}{\partial u_1},
\frac{\partial}{\partial x_2} ,\frac{\partial}{\partial u_2})$. 
As a consequence, the $4$-ball $B^4_D$ and its boundary $S^3_D$ inherit 
orientations. Also  the smooth part of $T_DP$ becomes an oriented surface.

With this setting we are ready to define the link associated to a given divide. 

\begin{definition}\label{def:link_divide}
	Given a divide $P \subset D$ we define the {\em associated link} $L(P)$ as
	\begin{equation}\label{eq:link_divide}
		L(P) := \{(x,u) \in T_DP \, | \, \|(x,u)\| = \rho\}=T_DP\cap S^3_D  
		\subset S^3_D.
	\end{equation}
\end{definition}
We will later see that not all 
links come from divides. We will also see that links that come from divides are 
fibered and that, even not all fibered links come from divides. For example, in 
\ref{rem:eight_knot_divide}, it is explained why  the figure eight knot 
(Fig.~\ref{fig:figure_eight}) cannot appear as the knot associated to a divide.

\begin{figure}
	\centering \includegraphics*[scale=0.6]{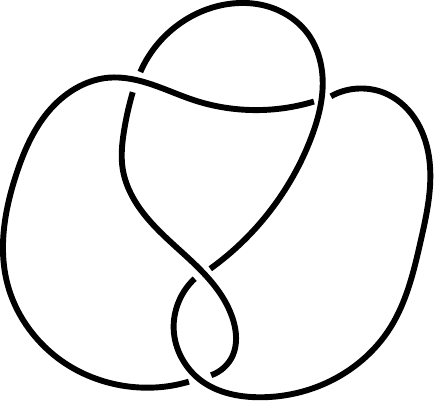}
	\caption{The figure eight knot.}
	\label{fig:figure_eight}
\end{figure}
The projection 
$TD \to D$ restricts to a map $S^3_D\to D$ having a circle as preimage of non 
boundary points of $D$. Those preimages shrink to circles of radius $0$ as the 
point moves to the boundary. The restriction to $L(P)$   is generically $2:1$ 
over the smooth points of $P$ not in 
$\partial D$ and $1:1$ above $\partial D$. It is $4:1$ over the double points 
of $P$.

As example, let $\gamma:[-1,1] \to D$ define a divide $P$ in 
the disk $D$ with end points $A=\gamma(-1), B=\gamma(+1)$ on 
$\partial{D}\subset S^3_D$. 
The divide $P$ is also defined by $\gamma^*(t)=\gamma(-t)$. 
The knot $L(P)$ is the union of two oriented arcs on $S^3_D$, one running from 
$A$ to $B$, the other running from $B$ to $A$. The first arc, using $\gamma$ is 
parametrized by 
$$t\in [-1,+1] \mapsto 
(\gamma(t),\frac{\dot{\gamma}(t)}{||\dot{\gamma}(t)||}(1-||\gamma(t)||^2)^{1/2})\in
S^3_D$$
The second arc, similarly, uses $\gamma^*$. Both oriented arcs $A,B$ fit 
together 
as an oriented knot.

A variant is following construction:
First extend $\gamma$ to an immersion $\Gamma:[-\tau,+\tau]\times [-1,+1] \to 
D$ of a rectangular thickening of the interval $[-1,+1]$. Next, for small 
$\tau'$, restrict $\Gamma$ to a rectangle $S=[-\tau',+\tau']\times^* 
[-1+\tau',+1-\tau']$ where $\times^*$ means that we smooth the corners. Let 
$\partial^+ S$ be the boundary of $S$ with any chosen orientation. The 
restriction $\partial^+\Gamma$ is a generic immersion of an oriented copy of 
$S^1$ into $D$. The lifting of the oriented speed vector of $\partial^+\Gamma$ 
yields an oriented knot which is isotopic to the oriented knot $L(P)$. So knots 
coming from generic immersed intervals, come also from generically immersed 
oriented circles. William Gibson and Masaharu Ishikawa \cite{Gibson} have 
proven  that every 
oriented knot can be 
obtained by lifting a immersion of an oriented circle in $D$. 

Naturally, a small perturbation of a divide, produces an isotopic  links in 
$B^4_D$. We can if necessary make a divide by a small isotopy more ``Euclidean 
friendly'' without changing its link: for instance we may assume  without 
restricting generality that the divide $P$ is near its double points an 
orthogonal intersection of  segments. Furthermore, we can consider more 
involved (and no longer 
necessarily small) perturbations that still produce isotopic links. This leads 
to the following definition and lemma.

\begin{definition}\label{def:regular_isotopy}
	We say that an isotopy of a divide $P$ is admissible if it is generated by 
	Reidermeister moves of type $III$ only. See Fig.~\ref{fig:admissible}.
\end{definition}

\begin{figure}
	\centering \includegraphics*[scale=1]{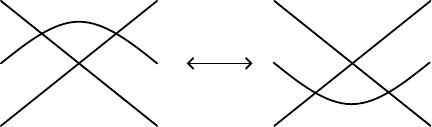}
	\caption{The movement caused by an admissible isotopy near an intersection 
		point.}
	\label{fig:admissible}
\end{figure}

\begin{lemma}\label{lem:regular_isotopy_link}
	If $P$ and $P'$ are related by an admissible isotopy, then $L(P)$ and 
	$L(P')$ 
	are isotopic in $\partial B_\rho$.
\end{lemma}

\begin{proof}
	Following the definition eq.~\ref{eq:link_divide}, it is possible to 
	explicitly 
	lift an admissible isotopy as the one described in 
	Fig.~\ref{fig:admissible} to 
	an 
	isotopy of links of divides as defined by Definition~\ref{def:link_divide}.
\end{proof}

\begin{example}
	Note that admissible isotopies do not yield a complete set of moves to go 
	from 
	one divide $P$ to another one $P'$ with $L(P) \sim L(P')$. 
	Indeed it can be
	verified that the links from the figure yield isotopic knots but one cannot 
	go
	from one to another only admissible isotopies, see Gibson-Ishikawa, two 
	divides for knot $10_{139}$ not related by III-moves.
	
	Knots or links $L(P)$ for divides $P$ are very special. For instance let 
	$L(P)$ be the knot of a divide $P\subset D$. The symmetry given by the 
	involution $(x,u) \mapsto (x,-u)$ fixes point-wise $\partial D\subset 
	S^3_D$  and preserves $L(P)$ globally, with two fixed point on $L(P)$. So 
	the knot $L(P)$ is strongly invertible, see Makoto Sakuma \cite{makoto}. 
	Divides links 
	$L(P)$ 
	are closures of quasi positive braids by Tomomi Kawamura's result 
	\cite{Kawamura}.

	\begin{figure}
		\huge
		\centering 
		\resizebox{0.7\textwidth}{!}{\centering 
			\includegraphics*{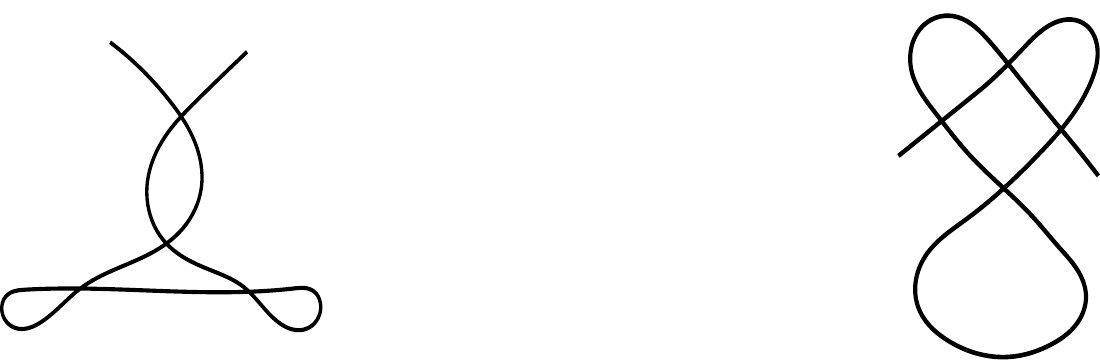}}
		\caption{Two divides that yield the same knot but that are not 
			equivalent 
			through admissible isotopies.}
		\label{fig:two_divides}
	\end{figure}

\end{example}

\subsection{Adapted Morse function $f_P$ on $D$ and its extension $\theta_{P}$ 
	on $B^4_D$ }

\begin{definition} \label{def:adapted_function}
	Let $P \subset D$ be a divide. We say that a smooth function 
	\[
	f_P:D \to \mathbb{R}
	\]
	is {\em adapted} to $P$ if 
	\begin{enumerate}
		\item \label{it:generic} $f_P$ is locally generic, i.e. it only has 
		non-degenerate 
		critical points
		\item \label{it:max_min} each bounded region has exactly one 
		non-degenerate 
		maximum or minimum.
		in its interior with critical values $\pm 1$
		\item \label{it:saddle} the saddle points of $f_P$ are exactly at the 
		double points of $P$ and $P=f_P^{-1}(0)$.
		\item each contractible unbounded region has exactly one non-degenerate 
		maximum or 
		minimum on its intersection with $\partial D$.
		\item  if $P \cap \partial D = \emptyset$ then the restriction of $f_P$ 
		to $\partial D$ 
		is constantly $\pm 1$.
	\end{enumerate}
	Note that $f_P$ is a Morse function in a disk slightly bigger than  $D$ 
	where 
	we allow different critical points to have the same critical values.
\end{definition} 

\begin{lemma}\label{lem:existence_adpated_function}
	Let $P \subset D$ be a connected divide. Then there  exist  adapted 
	functions 
	$f_P:D 
	\to \mathbb{R}$.
\end{lemma}

\begin{proof}
	We may change the divide $P\subset D\subset \mathbb{R}^2$ by a small 
	isotopy in order to make it ``Euclidean friendly'', i.e. at its double 
	points 
	$P$ will consist of a pair of Euclidean orthogonal segments. Every 
	connected component of $D\setminus P$, that we also call region, does not 
	meet $\partial D$ or meets $\partial D$ in a connected set since by 
	hypothesis $P$ is connected. We mark each such connected set by one point 
	if it is a contractible arc in the boundary of $D$.
	We color the regions by signs $+$ and $-$ in a chequerboard way, i.e. two 
	regions  separated by a subinterval in $P$ will have opposite signs. This 
	is possible since $D$ is simply connected and $P$ as planar graph has only 
	vertices of valency $4$. Next, we mark each region that not meets the 
	boundary of $D$ with a point. Let $r>0$ be a real such that the Euclidean 
	balls of radius $r$ in $\mathbb{R}^2$ with center a marked point are 
	disjoint from $P$, from $\partial D$ if its center is not on $\partial D$, 
	and from each other. Assume also that the intersection of $P$ with a ball 
	of radius $r$ and center a double point of $P$ is the union of two 
	Euclidean segments. We denote by $B_p$ the ball of radius $r$ with center 
	$p$ a marked point or double point of $P$. \newline\indent
	Now we can construct the function $f_P$. Define $f_P$ at a marked point 
	$p$ or double point of $P$ to be standard Euclidean  expression defining a 
	local maximum, local minimum, or saddle point according sign or to be a 
	double point of $P$. So on $B_p$ define $f_P(q)=1-||q-p||^2$ if $p$ is in a 
	$+$ region, $f_P(q)=-1+||q-p||^2$ or $f_P(q)=L(q)K(q)$ where $L,K$ are 
	local signed  linear equations of Euclidean norm $1$ for the segments of 
	$P$ at $p$. For a marker point $p\in \partial D$ we define $f_P$ on $B_p$ 
	as a norm $1$ linear boundary maximum or minimum with value $1$ or $-1$. If 
	$P$ does not meet $\partial D$ we define $f_P(q)=\pm 1, q\in \partial D$ 
	according the sign of the adjacent region. Finally put $f_P(q)=0$ for $q\in 
	P$. The so partially defined function extends in a up to isotopy unique way 
	to a smooth function $f_P:D\to \mathbb{R}$ that has no singularities except 
	at the marked points and double points of $P$.
\end{proof}

Let $\beta: \mathbb{R} \to \mathbb{R}$ be a  non negative smooth bump function
that evaluates to $1$ for $t\in [0,r/2]$ and to $0$ for $t\geq r$. Let $\chi: D 
\to \mathbb{R}$ be the smooth function that evaluates to $0$ on the complement 
of the union $\cup B_p$ and to
$\beta(||p-q||), q\in B_p$.\newline\indent
We define a complex function $\theta_{P,\eta}:B^4_D\to \mathbb{C}$
that depends on a real small parameter $\eta>0$ by
\begin{equation}\label{eq:theta_func}
	\theta_{P,\eta}(x,u):=f_P(x)+i \eta\, df_P(x)(u)-
	{1\over 2}\eta^2\chi(x)H_{f_P}(x)(u,u)
\end{equation}
The function $\theta_P=\theta_{P,\eta}$ satisfies the Cauchy-Riemann equation 
at points in $D$ and points in $4$-balls of radius $r/2$ in 
$T\mathbb{R}^2=\mathbb{C}^2$ and center a Morse critical point of $f_P$.   
These properties help to prove in next section a fibration theorem and compute 
the geometric monodromy, mainly following pioneering work of Milnor and 
Brieskorn.

More precisely, the argument function 
$\pi_{P,\eta}=\frac{\theta_{P,\eta}}{|\theta_{P,\eta}|}$ defines on $S^3_P$ an 
oriented open book structure with smooth binding like the argument function 
$\frac{f}{|f|}$ does on the Milnor sphere of an isolated hypersurface 
singularity 
$f=0$. More over $\theta_{P,\eta}$ behaves like a morsification of the
singularity	$f$, which allows a computation of the monodromy of that open 
book following:

\section{General fibration theorem} \label{s:fibration_thm}

\subsection{A connected divide yields a fibered link}
A main results from 
\cite{Acampo_generic_imm}   about divides is:

\begin{theorem}\label{thm:fibered_link}
	Let $P$ be a connected divide. For $\eta > 0$ and 
	sufficiently small, the map $\pi_P:=\pi_{P,\eta}$ induces an oriented open 
	book 
	decomposition of $S_P^3$ with binding $L(P)$, that is:
	\begin{enumerate}
		\item \label{it:fibration} $\pi_P$ is a fibration 
		of the complement of $L(P)$ over $S^1$,
		\item \label{it:near_binding} For all $\alpha \in S^1$, the boundary of 
		each 
		fiber (page of the open 
		book) is the binding, i.e., $\partial \pi_P^{-1}(\alpha) = 
		L(P)$
	\end{enumerate} 
\end{theorem}

\begin{proof}
	There exists a regular product 
	tubular neighborhood $N$ of $L(P),$ such that the map $\pi_{P,\eta}$ 
	for  any $1 \geq \eta > 0$ is on $N \setminus L(P)$  a
	fibration over $S^1$, 
	for which  near $L(P)$ the fibers look like the pages of a book 
	near its back. It is crucial to 
	observe that in the intersection of the link $L(P)$ with 
	the support of the function 
	\[
	(x,u)\in S^3 \mapsto \chi(x)\in \mathbb{R}
	\]
	the kernel of the Hessian of 
	$\theta_{P,\eta}$ and the kernel of the differential of the map 
	
	\[
	(x,u)\in S^3 \mapsto f_P(x)\in  \mathbb{R}
	\]
	coincide. For an alternative proof of this step, look at 
	\ref{rem:rem_fib_near_link} and the discussion before.
	
	For any $\eta > 0$, 
	the map $\pi_{P,\eta}$ is regular at each point of
	$U':=\{(x,u) \in S^3 | x \in U\}$ were $U$ is a small neighborhood of the 
	critical points of $f_P$. Take now local real coordinates 
	$(x,y,u,v)$ 
	around a maximum $p \in D$ 
	so that $f_P(x,y)=1 -x^2-y^2$ and $p = (0,0)$. In these coordinates, seeing 
	$\mathbb{C} \simeq \mathbb{R}^2$, we have that $\theta_P$ looks like
	
	\[
	\theta_P(x,y,u,v) = \left({\left(u^{2} + v^{2}\right)} \eta^{2} - x^{2} - 
	y^{2} 
	+ 
	1,\,-2 \eta (u x + v y)\right)
	\]
	and 
	
	\begin{align*}
		\pi_P(x,y,u,v) =&
		\left(\frac{{\left(u^{2} + v^{2}\right)} \eta^{2} - x^{2} - 
			y^{2} 
			+ 1}{\sqrt{{\left({\left(u^{2} + v^{2}\right)} \eta^{2} - x^{2} - 
					y^{2} + 
					1\right)}^{2} + 4 \eta^2 {\left(u x + v y\right)}^{2}}}, 
		\right.
		\\
		&
		\left. \,-\frac{2 \, {\left(\eta 
				u x + \eta v y\right)}}{\sqrt{{\left({\left(u^{2} + 
						v^{2}\right)} \eta^{2} - 
					x^{2} - 
					y^{2} + 1\right)}^{2} + 4 \eta^2 {\left(u x + v 
					y\right)}^{2}}}\right)
	\end{align*}
	
	Let $(0,0,u,v)$ be a point in $U'$ that projects to $p$. Then, the 
	differential 
	of $\pi_P$ evaluated at that point equals:
	\[
	\left(\begin{array}{rrrr}
		0 & 0 & 0 & 0 \\
		-\frac{2 \, \eta u}{{\left(u^{2} + v^{2}\right)} \eta^{2} + 1} & 
		-\frac{2 
			\, \eta 
			v}{{\left(u^{2} + v^{2}\right)} \eta^{2} + 1} & 0 & 0
	\end{array}\right)
	\]
	and the two non-zero entries of the matrix can't be $0$ at the same time 
	for a 
	point of the form $(0,0,u,v) \in S^3$ since $u^2+v^2=1$. This shows that 
	the 
	jacobian of $\pi_P$ has rank $1$ on those points and hence, $\pi_P$ is a 
	submersion on a neighborhood $U'$. The computations for a minimum or a 
	saddle 
	point are very similar.
	
	Finally, since the norm of the derivative of $\chi$ can be bounded, a 
	similar 
	computation yields that there exists $\eta_0>0$ such that 
	for 
	any $\eta$ with  $0<\eta<\eta_0$, the map $\pi_{P,\eta}$ is regular on 
	$S^3\setminus (N\cup U').$ Hence,  for $\eta$ sufficiently small the map 
	$\pi_{P,\eta}$  is  a 
	submersion, so since already a fibration near 
	$L(P),$ it is a fibration by Ehresmann fibration theorem. 
	
\end{proof}

\section{Divides for plane curve singularities} \label{s:divides_plane_curves}

In this section we construct a divide $P$ associated to a given isolated plane 
curve 
singularity in the sense that $L(P)$ is a model for the link of the plane curve 
sngularity. 
What we actually produce is a \index{real morsification} {\em real 
	morsification} of a totally real plane 
curve singularity; and this real morsification yields a divide. We do so, 
following two different programs: the first one 
Sect.~\ref{ss:divides_embedded} gives an algorithm to construct such a divide 
from 
an embedded resolution of the plane curve, the second one, in 
Sect.~\ref{ss:divide_chebyshev} gives a closed formula that produces a divide 
from a 
complete set of topological invariants of a plane curve.

\subsection{Totally real models of plane curve singularities}

For both methods,  it is necessary to start with a {\em totally real} 
model of our isolated plane curve singularity.

\begin{definition}\label{def:totally_real}
	Let $f\in \mathbb{C}\{x,y\}$ be a converging power series and let $f = f_1 
	\cdots 
	f_r$ be a factorization into irreducible factors. We say that $f$ is {\em 
		real} or that it defines a {\em real plane curve singularity} if it is 
	given by a real equation, that is, if $f \in \mathbb{R}\{x,y\}$. We 
	say that $f$ is {\em 
		totally real} if $f_i \in \mathbb{R}\{x,y\}$ for all $i = 1, \ldots, r$.
	
	Equivalently, $f$ is real if the curve $(f^{-1}(0),0)$ is invariant under 
	complex conjugation and $f$ is totally real if each of its branches is real.
\end{definition}

Once one has a real plane curve singularity, one can try to perturb it in a 
controlled way with the hope that all the critical points (as a complex 
function) are actually real and as simple as possible: Morse. It turns out that 
if one is able to do this, the real picture picture of the deformed curve tells 
us a lot about the complex singularity. Actually, this contains all the 
topological information of the plane curve singularity. In particular one is 
able to recover a \index{distinguished basis} distinguished basis of 
\index{geometric vanishing cycle} 
geometric vanishing cycles. These 
special types of deformations are 
called real morsifications:

\begin{definition}\label{def:real_morsification}
	Let $f:\mathbb{C}^2 \to \mathbb{C}$ define an isolated real plane curve 
	singularity, let 
	$B_\epsilon \subset \mathbb{R}^2$ be a Milnor ball for $f$ and let 
	$D_{\delta}$ be 
	a Milnor disk for $f$. We say that a family of convergent power series 
	$\{f_t\}_{t \in [0, \tau)}$ is a  {\em real morsification} of $f$ if
	\begin{enumerate}
		\item it is a deformation of $f$, that is, $f_0=f$,
		\item $f_t$ is real for all $t$,
		\item for all $t \in (0,\tau)$ the map $f_t$ only has Morse critical 
		points and they all lie in $B_\epsilon$. Furtheremore, the number of 
		Morse critical points does not depend on $t$.
		\item all its critical points are real, that is, they lie in 
		$\mathbb{R}^2 
		\subset \mathbb{C}^2$.
		\item all saddle points lie over $0$.
	\end{enumerate}
\end{definition}

In the next two sections we prove, in two different ways, that there exist real 
morsifications for totally real isolated plane curve singularities. It is still 
an open question if there exist real morsifications for all isolated real plane 
curve singularities. See \cite{Lev} for a partial solution to this question and 
see \cite{Fom} for the implications of a solution of this problem in other 
conjectures.

\begin{remark}\label{rem:totally_real_model}
	It follows from the topological classification of isolated plane curve 
	singularities that, given a topological type, it is always possible to find 
	a 
	totally real isolated plane curve singularity realising that topological 
	type 
	(see \cite{ZarI,ZarII,ZarIII} for a classical reference or \cite[Corollary 
	5.3.3 and Lemma 2.3.1]{Wall} for a more recent and concrete reference).  
	See also the previous related \ref{rem:geom_invariant}
\end{remark}

\subsection{Divides from an embedded resolution} \label{ss:divides_embedded} In 
this section we reproduce 
the proof of one of the main results (Th\'eor\`eme 1.) from 
\cite{Acampo_groupe_I} which proves 
the existence of divides for plane curve singularities. More concretely, the 
proof gives an algorithm to construct a polynomial real deformation of a plane 
curve singularity under the hypothesis that the singularity is given by a 
totally real polynomial.

\begin{theorem}\label{thm:existence_divides}
	Let $f(x, y)$ be a totally real polynomial defining an isolated plane curve 
	singularity. Then there exists a real polynomial deformation $f(x, y; t), t 
	\in \mathbb{R}$, such that $f(x, y; 0)=f(x, y)$ and that for all $t \neq 0, 
	t 
	\in \mathbb{R}$ and small enough, we have
	
	\begin{enumerate}
		\item the real curve $C_{t}=\left\{ (x, y) \in D_{\varepsilon} | f(x, 
		y; 
		t)=0\right\}$ is a divide with $r$ branches in a small disk 
		$D_{\varepsilon}$ 
		near the origin.
		
		\item the number $k$ of double points of $C_{t}$ verifies $2 
		k-r+1=\mu$, where 
		$\mu$ is the Milnor number of $f$ at $O$.
	\end{enumerate}
\end{theorem}

Next we explain the main construction on which its based the proof of 
Theorem~\ref{thm:existence_divides} which consists in iteratively blowing up 
until, 
exactly, the strict transform is resolved and consists of smooth branches 
meeting transversely (but not in normal crossings) the exceptional divisor. At 
this point, we translate slightly the strict transforms along some exceptional 
divisor and blow-down back all the exceptional divisors, proving that the 
translation yields a deformation of the 
original plane 
curve in $\mathbb{C}^2$. Performing this operation as many times as necessary, 
yields 
a pre-divide (recall Definition~\ref{def:pre_divide})  which can be further 
polynomially 
deformed into an actual divide. 

Next we explain this process in more detail.
Let $f(x, y)$ be a totally real polynomial defining an isolated plane curve 
singularity $C$. Let
\[
\pi_{n}: X_{n} \rightarrow X_{n-1}, \quad 1 \leq n \leq N
\]
be a sequence of transformations that yield an embedded resolution of $C_0$, 
where

\begin{enumerate}
	\item $X_{0}=\mathbb{C}^{2}$
	\item $\pi_1:X_1 \to X_0$ is the blow-up at the origin.
	\item $\pi_k:X_k \to X_{k-1}$ for $i=2, \ldots, N-1$ is a blow-up at a 
	point in the exceptional divisor $E_{k-1}$ of $\pi_1 \circ \cdots \circ 
	\pi_{k-1}$
	\item $H_{N-1}=\left\{ z \in X_{N} | f \circ \pi_{1} \circ \cdots \circ 
	\pi_{N-1}(z)=0\right\}$ is a union of pairwise transversal smooth curves.
	\item $\pi_N:X_N \to X_{N-1}$ is a composition of blow ups that resolves 
	all the ordinary singularities of $H_{N-1}$ where two or more components of 
	the strict transform of $C$ meet
	\item  $H_{N}=\left\{ z \in X_{N} | f \circ \pi_{1} \circ \cdots \circ 
	\pi_{N}(z)=0\right\}$ is  a hypersurface in $X_{N}$ having as singularities 
	only normal crossings.
\end{enumerate}
In other words, $\pi_1, \ldots, \pi_{N-1}$ is an usual sequence of blow ups of 
the minimal resolution of the plane curve defined by $f$ just until the strict 
transform consists of smooth curves not meeting in normal crossings, and 
$\pi_N$ is the composition of all the remaining blow ups that give an embedded 
resolution of the plane curve.

For $1 \leq n \leq N$, 
we denote by
\[
H_{n}=\left\{ z \in X_{n} | f \circ \pi_{1} \circ \cdots \circ 
\pi_{n}(z)=0\right\}
\] 
the total transform in $X_n$, by
\[
E_{n}=\left\{ z \in X_{n} | \pi_{1} \circ \cdots \circ \pi_{n}(z)=0\right\} 
\subset H_{n} \subset X_{n}
\] 
the exceptional divisor in $X_n$, and by $B_{n}$ the component of $E_{n}$, 
which 
is the exceptional divisor corresponding to the blow-up $\pi_{n}$.

Let $\left[x_{n}, y_{n}\right]$ be a homogeneous coordinate system on $B_{n}$ 
such that the affine chart $A_{n}=\left\{ y_{n}=1 \right\}$ contains the
intersection points of $B_{n}$ with the other branches of $H_{n}$. Let $  
V_{n}=\left\{\left(\lambda_{n}, x_{n}\right)\right\}, \lambda_{n} \in 
\mathbb{C}, x_{n} \in A_{n}$ be a chart of $X_{n}$ such that we have $B_{n} 
\cap 
V_{n}=A_{n}$ and $B_{n} \cap V_{n}=\left\{\lambda_{n}^{m_{n}}=0\right\}$ with 
$m_{n}$  the multiplicity of $B_{n}$ in $X_{n}$.

Since $f(x,y)$ is totally real by hypothesis, we can choose the  blow-ups 
$\pi_{n}: X_{n} 
\rightarrow X_{n-1}$ for $1 \leq n <N$ having as center a real point of 
$X_{n-1}$. So we can 
also choose the coordinates $\left[x_{n}, y_{n}\right]$ and 
$\left(x_{n}, \lambda_{n}\right)$ to be real at each step of the resolution 
process. The function
\[
f_{N-1}=f \circ \pi_{1} \circ \cdots \circ \pi_{N-1}
\]
is written in the chart $V_{N-1}$ like
\[
f_{N-1}\left(x_{N-1}, \lambda_{N-1}\right)=\lambda_{N-1}^{m_{N-1}} 
\tilde{f}_{N-1}\left(x_{N-1}, \lambda_{N-1}\right) .
\]
We decompose $\tilde{f}_{N-1}=h+g$ with $h$ the initial homogeneous part. 
Assume that $B_M$,  with $1 \leq M \leq N-2$, is the component of $E_{N-1}$ 
that 
intersects $B_{N-1}$. Then, $h$ is of 
degree $m_{N-1}-m_{M}$. So the set
\[
\left\{\left(x_{N-1}, \lambda_{N-1}\right) \in V_{N-1} | h\left(x_{N-1}, 
\lambda_{N-1}\right)=0\right\}
\]
is the union of the tangents to the non-exceptional branches of $H_{N-1}$ at 
the points of  $B_{N-1} \cap H_{N-1}$. For $t_{N-1} \in \mathbb{R}$ we put
\[
\begin{aligned}
	\bar{f}_{N-1}\left(x_{N-1}, \lambda_{N-1}, t_{N-1}\right) = \quad & 
	h\left(x_{N-1}, 
	\lambda_{N-1}-t_{N-1}\right) \\
	&  \times\left[1+t_{N-1} \cdot \prod_{i=1}^l \left(x_{N-1}-x_{i}\right) 
	\cdot 
	\exp \left(x_{N-1}\right)\right] \\
	& +g\left(x_{N-1}, \lambda_{N-1}-t_{N-1}\right)
\end{aligned}
\]
where $x_{1}, \ldots, x_{l}$ are the roots of $h\left(x_{N-1}, 0\right)=0$. For 
$t_{N-1} \in \mathbb{R}$ and $t_{N-1}\neq 0$ small enough, the equation
\[
\bar{f}_{N-1}\left(x_{N-1}, 0, t_{N-1}\right)=0
\]
has exactly $m_{N-1}-m_{M}$ distinct roots on 
$B_{N-1} \cap V_{N-1}$  (and not in $B_{M}$) and
\[
\bar{f}_{N-1}\left(\infty, 0, t_{N-1}\right)=\infty \quad \text { at the point 
} 
\quad \infty \in B_{N-1} \setminus V_{N-1}.
\]
We can see in Fig.~\ref{fig:chartV} the branches of $H_{N-1}$ that cross 
$B_{N-1}$ and $B_{M}$.
\begin{figure}
	\tiny
	\centering \resizebox{0.9\textwidth}{!}{\centering 
		\includegraphics*{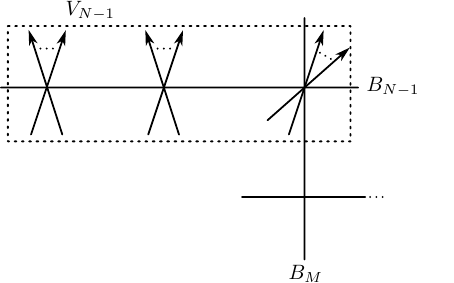}}
	\caption{$H_{N-1}$ near $B_{N-1}$.}
	\label{fig:chartV}
\end{figure}
For $t_{N-1}$ real, non-zero and small, Fig.~\ref{fig:perturbedHN} is the 
actual 
real drawing of
\[
H_{N-1, t_{N-1}} \cap V_{N-1} 
=\left\{\left(x_{N-1}, \lambda_{N-1}\right) \in V_{N-1} | 
\lambda_{N-1}^{m_{N-1}} \bar{f}_{N-1}\left(x_{N-1}, \lambda_{N-1}, 
t_{N-1}\right)=0\right\}
\]

\begin{figure}
	\tiny
	\centering \resizebox{0.8\textwidth}{!}{\centering 
		\includegraphics*{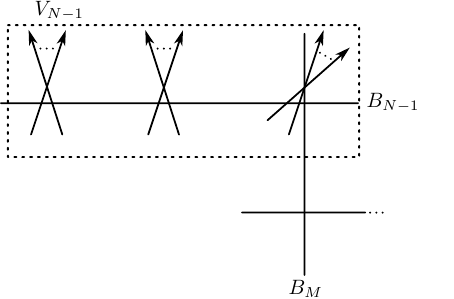}}
	\caption{We can see $H_{N-1, t_{N-1}}$ in the chart $V_{N-1}$ 
	}
	\label{fig:perturbedHN}
\end{figure}

Neighbouring points on the $B_{N-1}$ branch are translated out and next to 
$B_{N-1}$. So $H_{N-1, t_{N-1}}$ meets $B_{N-1}$ in normal crossings. The 
presence of the factor $\lambda_{N-1}^{m_{N}-1}$ in $f_{N-1, 
	t_{N-1}}$ allows us to push the function $f_{N-1, i_{N-1}}$ to a function 
$f\left(x, y, 
t_{N-1}\right)$ on 
$\mathbb{C}^{2}$. We observe that by construction,
\[
f(x, y; 0)=f(x, y)
\]
and for $t_{N-1} \neq 0, t_{N-1}$ real and small enough, the hypersurface
\[
H_{N-1}\left(t_{N-1}\right)=\left\{(x, y) \in\mathbb{C}^2 | f\left(x, y, 
t_{N-1}\right)=0\right\}=\pi_{1} \circ \cdots \circ \pi_{N-1}\left(H_{N-1, 
	t_{N-1}}\right)
\]
has in  $B_{\varepsilon}$ the singularities:
\begin{enumerate}[label= \alph*)]
	\item  an isolated singularity in $0  \in B_{\varepsilon}$, 
	\item  if  $N_{i}^{N-1}\left(t_{N-1}\right), 1 \leq i 
	\leq l_{N-1}$ are the translates of neighbouring points on $B_{N-1}$; then 
	$H_{N-1} (t_{N-1})$ has at points 
	\[
	w_{i}^{N-1}\left(t_{N-1}\right)=\pi_{1} 
	\circ \cdots \circ \pi_{N-1}\left(N_{i}^{N-1}\left(t_{N-1}\right)\right)
	\] 
	isolated ordinary singularities. That is, transversal intersections of 
	smooth branches.
\end{enumerate}
The blow-ups $\pi_{1}, \ldots, \pi_{N-2}$ resolve the singularity at $0 \in 
B_{\varepsilon}$ of $H_{N-1}\left(t_{N-1}\right)$ so that the branches of 
$H_{N-1, t_{N-1}}=\left(\pi_{1} \circ \cdots \circ \pi_{N-2}\right)^{-1}$ - 
$\left(H_{N-1}\left(t_{N-1}\right)\right)$ meet each other transversely.  We 
can therefore start the construction again and obtain a 
deformation of $f\left(x, y, t_{N-1}\right)$.
\[
f\left(x, y, t_{N-1}, t_{N-2}\right), t_{N-2} \in \mathbb{R}
\]
such that the hypersurface $\left\{f\left(x, y, t_{N-1}, 
t_{N-2}\right)=0\right\} 
\subset\mathbb{C}^2$ has an isolated singularity in $0 \in B_{\varepsilon}$. 
In addition, one must take care so that
\[
f\left(x, y, t_{N-1}\right) \equiv f\left(x, y, t_{N-1}, t_{N-2}\right)
\]
at a high enough order at the points $w_{i}^{N-1}\left(t_{N-1}\right), 1 \leq i 
\leq l_{N-1}$. This operation can be iterated $N-1$ times to obtain a 
family of functions on $\mathbb{C}^{2}$.
\[
f\left(x, y, t_{N-1}, t_{N-2}, \ldots, t_{1}\right),\left(t_{N-1}, \ldots, 
t_{1}\right) \in \mathbb{R}^{N-1}
\]
We finally define
\[
\begin{aligned}
	\bar{f}(x, y, t) & =f(x, y, t, t, \ldots, t), \quad t \in \mathbb{R} \\
	\bar{H}_{t} & =\left\{x, y \in\mathbb{C}^2 | \bar{f}(x, y, 
	t)=0\right\}
\end{aligned}
\]

From this construction, the following lemma follows:

\begin{lemma}\label{lem:pre_divide}
	For $t \in \mathbb{R}, t \neq 0$ and small enough, the curve 
	$\bar{C}_{t}=\bar{H}_{t} \cap D_{\varepsilon}$ is an immersion 
	$\bar{\alpha}: J 
	\rightarrow D_{\varepsilon}$ such that
	\begin{enumerate}[label= \alph*)]
		\item The number of components of $J$ equals the number $r$ of branches 
		of $f$,
		
		\item \label{it:connected} $\bar{\alpha}(\partial J) \subset \partial 
		D_{\varepsilon}$, and 
		the set 
		$ \bar{\alpha}(J) 
		\subset D_{\varepsilon}, \alpha(J)$ is connected.
		
		\item \label{it:ordinary} $\bar{\alpha}(J)=\bar{C}_{t}$ has as as its 
		only singularities, 
		ordinary 
		singularities.
	\end{enumerate}
	In particular, \ref{it:connected} and \ref{it:ordinary} say that 
	$\bar{\alpha}$ is a 
	connected pre-divide (recall Definition~\ref{def:pre_divide}).
\end{lemma}

The number of double points 
appearing near a multiple point with  $r_{i}$ branches of $\bar{C}_{t}$, is
$$
\frac{1}{2} r_{i}\left(r_{i}-1\right).
$$
So the number of double points of any divide $\alpha : J \rightarrow 
D_{\varepsilon}$ next to $\bar{\alpha}: J \rightarrow D_{\varepsilon}$ is
\begin{equation}\label{eq:double_points}
	\frac{1}{2} \sum_{1 \leq i \leq l} r_{i}\left(r_{i}-1\right).
\end{equation}
In the next theorem we complete the proof of 
Theorem~\ref{thm:existence_divides}, by 
proving that $\bar{\alpha}$ can be polynomially perturbed into a divide and we 
also show that the $\delta$ invariant of the singularity coincides with the 
number of double points of any divide coming from a totally 
real singularity is  we only need to prove.

\begin{lemma} \label{lem:generic_perturbation}
	Let $\alpha: J \rightarrow D_{\varepsilon}$ be a divide resulting from a 
	small
	perturbation of $\bar{\alpha}: J \rightarrow D_{\varepsilon}$. Then there 
	exists a real polynomial $M(x, y, t)$ so that the function
	\[
	f(x, y; t)=\bar{f}(x, y, t)+M(x, y, t), \quad t \in \mathbb{R}
	\]
	is a deformation of $f(x, y)$, such that its divide $C_{t} \subset 
	D_{\varepsilon}$ has the configuration of the divide defined by $\alpha$ 
	and 
	$C_{t} \subset  D_{\varepsilon}$ is close to the image of $\bar{\alpha}$.
	
	Moreover, if
	$w_{i}(t), 1 
	\leq i \leq l$ are the crossing points (ordinary singularities) of 
	$\bar{C}_{t}$, and $r_{i}$ is
	the number of branches of $\bar{C}_{t}$ at points $w_{i}(t)$. 
	Then we have
	\[
	\sum_{1 \leq i \leq l} r_{i}\left(r_{i}-1\right)-r+1=\mu.
	\]
	Equivalently, the number of double points of the divide $D$ coincides with 
	the $\delta$ invariant of the singularity and
	\[
	D = \delta = \frac{1}{2} \sum_{1 \leq i \leq l} r_{i}\left(r_{i}-1\right)
	\]
\end{lemma}

\begin{proof}
	Let $\left(x_{1}(t), y_{1}(t)\right), \ldots,\left(x_{e}(t), 
	y_{e}(t)\right) \in D_{\varepsilon}$ the crossing points (double or 
	multiple) 
	of 
	of $\bar{C}_{t}$. Assume that $w_{1}(t)=\left(x_{1}(t), y_{1}(t)\right)$ 
	is a multiple crossing point since if there are only double points, the 
	lemma 
	is proven because our deformation $\bar{C}_{t}$ is given by a polynomial.
	
	Let $r_{1}$ be the number of branches at the point $w_{1}(t)$. Let, for $1 
	\leq i 
	\leq r_{1}$,
	$$
	l_{i}(x, y, t)=0
	$$
	the equations of the tangents at $w_{1}(t)$ to the branches of 
	$\bar{C}_{t}$. 
	Any 
	configuration of a divide which is close to $\bar{C}_{t}$ can be described 
	near 
	$w_{1}(t)$ by an equation
	\[
	L_{1}(x, y, t)=\prod_{1 \leq i \leq r_{1}} t\left(l_{i}(x, y, t)-t 
	a_{i}\right)=0
	\]
	where $a_{i} \in \mathbb{R}$. Thus, the 
	equation
	\begin{equation}\label{eq:deformation}
		\bar{f}(x, y, t)+L_{1}(x, y, t)=0 
	\end{equation}
	gives for $t \neq 0$ and small enough the same configuration. Let 
	let $w_{1,1}(t), \ldots, w_{1, s}(t)$ be the $s=1 / 2 
	r_{1}\left(r_{1}-1\right)$ 
	double points that appear near $w_{1}(t)$ for the equation 
	eq.~\ref{eq:deformation}. 
	Let 
	$P_{1}(x, y; t)$ be a polynomial such that for a high  enough order we have
	
	\[
	\begin{aligned}
		P_{1}\left(w_{1, j}(t) ; t\right) \equiv 1, & & 1 \leq j \leq s, \\\\
		P_{1}\left(x_{i}(t), y_{i}(t) ; t\right) \equiv 0, & & 2 \leq i \leq e.
	\end{aligned}
	\]
	Then
	\[
	\bar{f}_{(1)}(x, y, t)=\bar{f}(x, y, t)+P_{1}(x, y ; t) L_{1}(x, y, t)
	\]
	is a deformation of $f(x, y)$ such that the curve
	\[
	C_{(1)}(t)=\left\{(x, y) \in D_{\varepsilon} | \tilde{f}_{(1)}(x, y, 
	t)=0\right\}
	\]
	has near $w_{1}(t)$ the desired configuration and $\left(x_{i}(t), 
	y_{i}(t)\right), 2 \leq i \leq e$, are still multiple points of 
	$\bar{C}_{(1)}(t)$. Repeat the same construction for another multiple 
	crossing 
	point until there are no more. This proves the first part of the statement.
	
	\begin{figure}
		\centering \includegraphics*[scale=0.6]{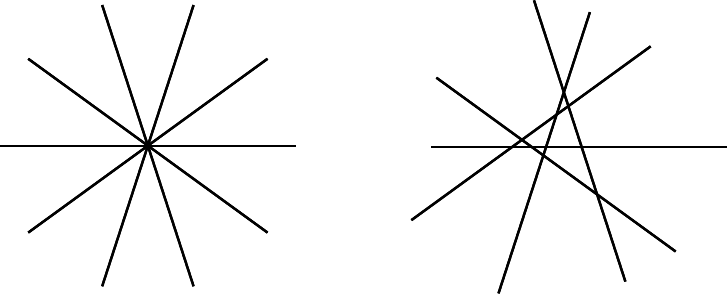}
		\caption{On the left $5$ local branches and on the right the perturbed 
			branches that lead to $10$ double points.}
		\label{fig:perturb_branches}
	\end{figure}
	
	The second part follows from the following observation: if one has an 
	ordinary real singularity consisting of $r_i$ branches, after a small 
	perturbation of these branches, there appear $\frac{1}{2} r_i(r_i-1)$ 
	double points (Fig.~\ref{fig:perturb_branches}). Let's denote
	\[
	D = \frac{1}{2} \sum_{1 \leq i \leq l} r_{i}\left(r_{i}-1\right)
	\]
	the number of double points of the divide. We first recall  Milnor's 
	formula  \cite{Milnor} relating the 
	Milnor number 
	$\mu$, the delta invariant $\delta$ and the number of branches	$\mu =  2 
	\delta -r + 1$. Let $b$ be the number of bounded regions of the divide. 
	Since on each bounded region there lies a critical point of $f(x,y;t)$ and 
	all the other critical points are the nodes of the divide, we have
	\begin{equation} \label{eq:bpd}
		b + D = \mu.
	\end{equation}
	The connectedness of the divide implies that the union of the closure of 
	the bounded regions is a topological disk, hence a direct calculation of 
	the Euler characteristic yields
	\begin{equation}\label{eq:bmd}
		b - (2 D - r) + D = 1 \quad \Rightarrow \quad b - D = 1 - r
	\end{equation}
	Taking eq.~\ref{eq:bpd} minus eq.~\ref{eq:bmd} yields $2D = \mu - 1 + r$ 
	which, 
	after Milnor's formula implies $\delta = D$.
\end{proof}

\begin{example} \label{ex:divide_translate}
	This is an example of a divide using a similar technique as the one 
	explained in Theorem~\ref{thm:existence_divides}. Note that the 
	perturbations 
	chosen here are not exactly the same.
	
	Start with the isolated plane curve singularity defined by the polynomial 
	$f(x,y)= -x^{8} - x^{7} - 3  x^{5} y + y^{3}$.
	
	\begin{figure}
		\centering \includegraphics*[scale=0.6]{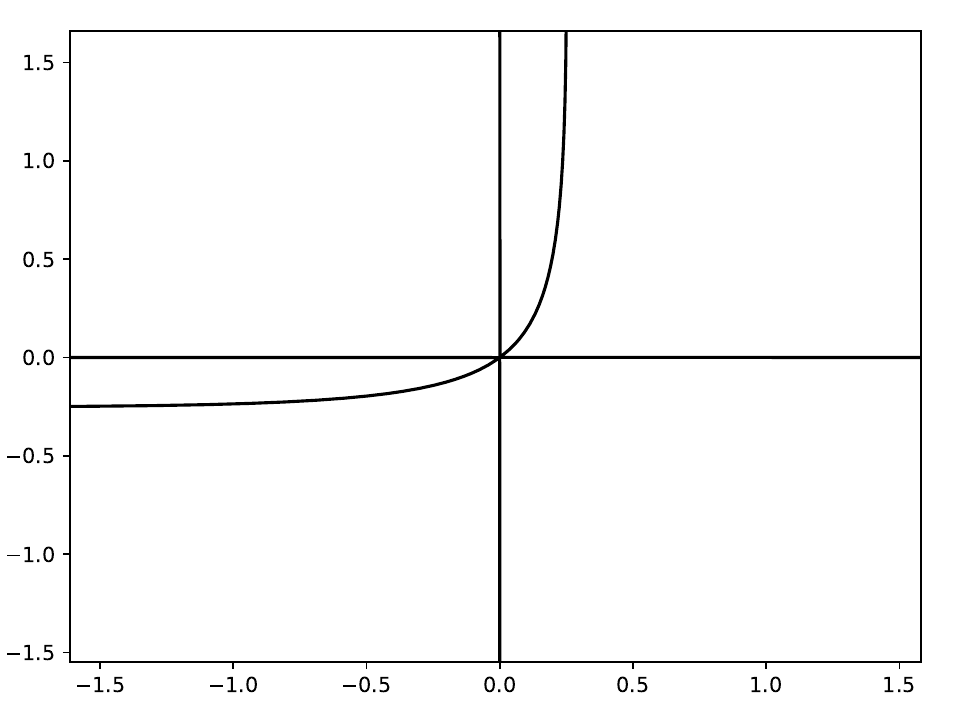}
		\caption{Situation after ther first $4$ blow ups of the minimal 
			embedded resolution.}
		\label{fig:last_blow_up}
	\end{figure}
	
	After four blow ups we arrive at the situation of 
	Fig.~\ref{fig:last_blow_up}. 
	In a 
	local chart, near the intersection point of the divisors $E_4$ and $E_2$, 
	the total transform is defined by the expression
	\[
	f_4(x,y)=-{\left(x^{2} y^{2} + 3 \, x y + x - y\right)} x^{6} y^{14}
	\]
	We can perturb the strict transform  by considering the polyomial :
	\[
	f_{4,t}(x,y)=-{\left(x^{2} y^{2} + 3 \, x y + t + x - y\right)} x^{6} y^{14}
	\]
	which, for $t>0$ defines the picture of 
	Fig.~\ref{fig:last_blow_up_perturbed}.
	\begin{figure}
		\centering \includegraphics*[scale=0.6]{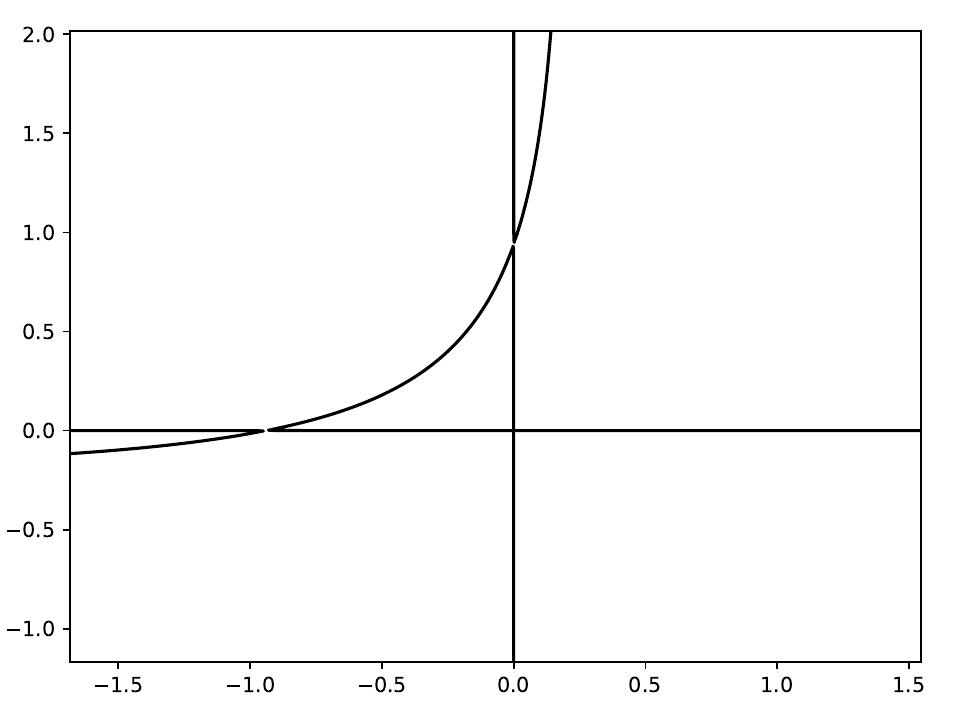}
		\caption{Situation after perturbing the strict transform near $E_4$ and 
			$E_2$.}
		\label{fig:last_blow_up_perturbed}
	\end{figure}
	Now, 
	we observe that the contraction of $E_4$, yields the equation
	\[
	\hat{f}_{4,t}(x,y)=-{\left(x^{2} y + t y + 3 \, x y - y^{2} + x\right)} 
	x^{6} y^{7}
	\]
	which, for $t>0$ describes the situation depicted in 
	Fig.~\ref{fig:contraction_perturbed}
	\begin{figure}
		\centering \includegraphics*[scale=0.6]{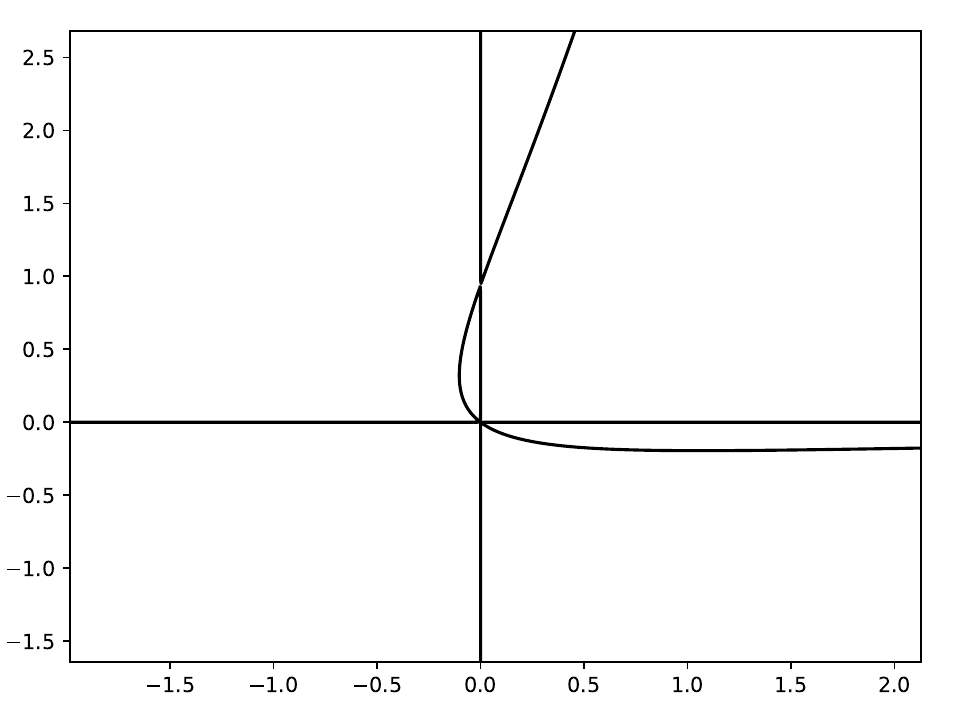}
		\caption{The divisor $E_4$ is contracted and the strict transform is 
			only perturbed by the previously defined perturbation. We see that 
			the strict transform and the divisor do not meet in normal 
			crossings.}
		\label{fig:contraction_perturbed}
	\end{figure}
	Again, we perturb the strict transform in order to achive normal crossings 
	by considering the following 
	polynomial
	\[
	f_{3,t}(x,y)= -{\left(x^{2} y + t y + 3 \, x y - y^{2} + t + x\right)} 
	x^{6} y^{7}.
	\]
	We arrive to the situation of 
	\begin{figure}
		\centering \includegraphics*[scale=0.6]{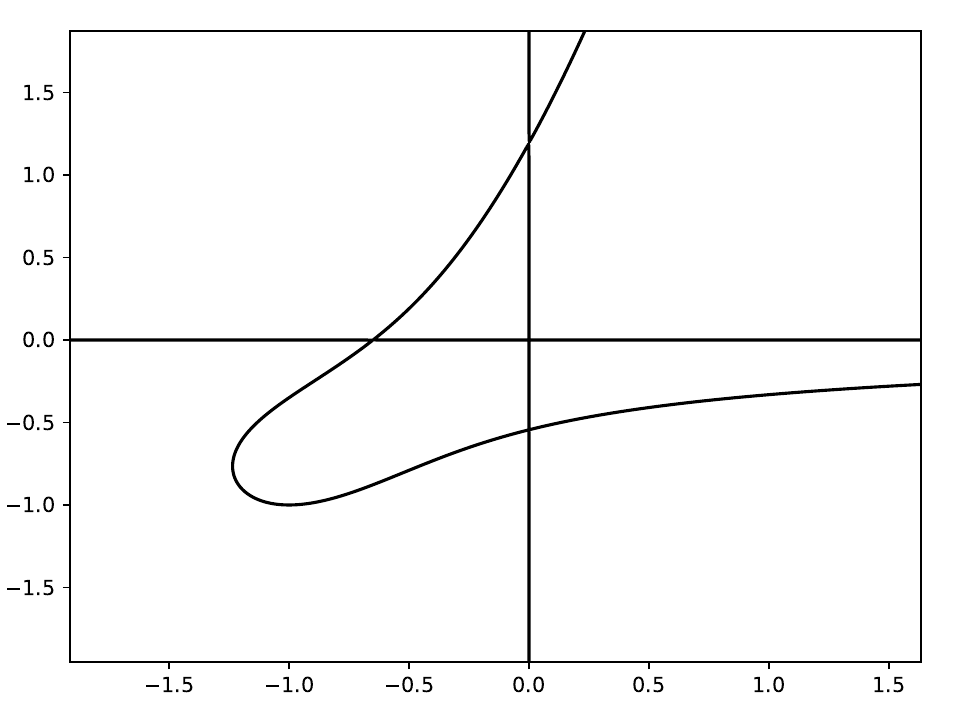}
		\caption{After a small perturbation of the strict transform we arrive 
			at a normal crossings situation.}
		\label{fig:perturbed_transversal}
	\end{figure}
	To avoid cumbersome notation, we have used the same parameter $t$ for the 
	perturbation. One can think of this as two steps: first one considers a 
	perturbation over a two-parameter family; then one considers the 
	perturbation over a disk contain in this family. By the choice of our 
	pertubations, we are considering the pertubation over the diagonal of this 
	two-parameter family. Note that Theorem~\ref{thm:existence_divides} 
	proposes, a 
	priori different pertubations to ensure that they behave well with respect 
	to the blow-down process. In our case, we verify this at each step even 
	though we make different choices that the ones proposed by the theorem.
	In the above equation, $y=0$ defines locally $E_3$ and $x=0$ defines $E_2$. 
	After the contraction of $E_3$ we arrive to the polynomial
	\[
	\hat{f}_{3,t}(x,y)=-{\left(t y^{2} - y^{3} + x^{2} + t y + 3 \, x y + 
		x\right)} x^{6}
	\]
	which, for $t>0$ defines the situation depicted in 
	Fig.~\ref{fig:manual_divide_cubic_perturbed}.
	\begin{figure}
		\centering 
		\includegraphics*[scale=0.6]{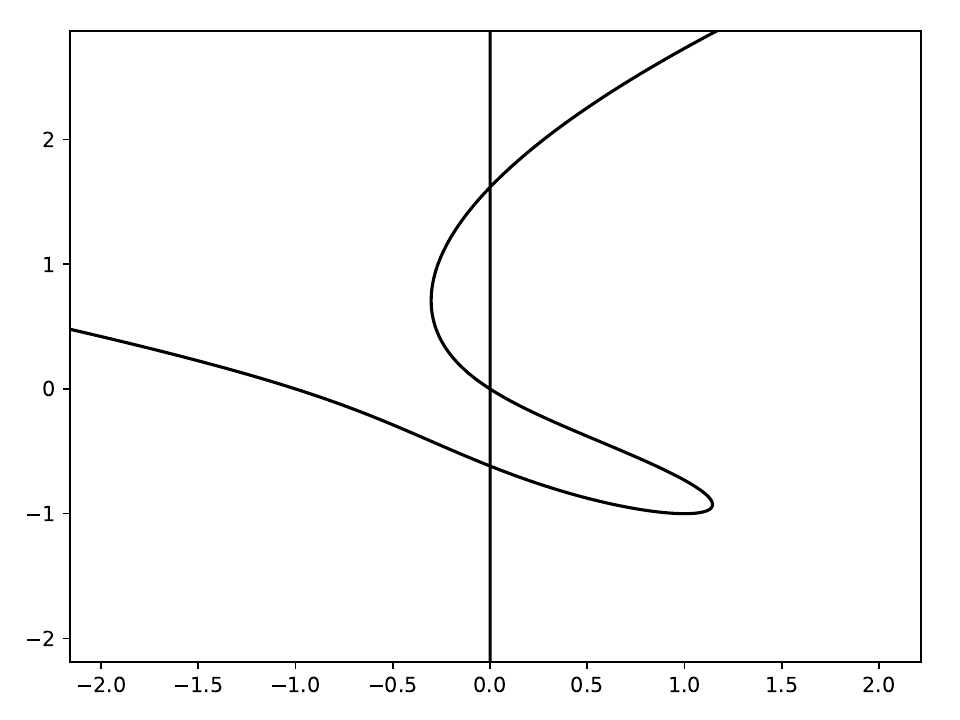}
		\caption{The divisor $E_3$ is contracted and the strict transform is 
			perturbed by the previously defined perturbation. In this case we 
			automatically achieve normal crossings.}
		\label{fig:manual_divide_cubic_perturbed}
	\end{figure}
	In this case, the strict transform already intersects the 
	divisor $E_2$ transversely. So we don't perform a perturbation at this 
	stage and we directly contract $E_2$ to arrive to the equation
	\[
	\hat{f}_{2,t}(x,y)=-{\left(x^{5} + x^{4} + t x^{2} y + 3 \, x^{3} y + t x 
		y^{2} - y^{3}\right)} x^{3}.
	\]
	which describes the local situation depicted at 
	Fig.~\ref{fig:manual_divide_last_before_trans}.
	\begin{figure}
		\centering 
		\includegraphics*[scale=0.6]{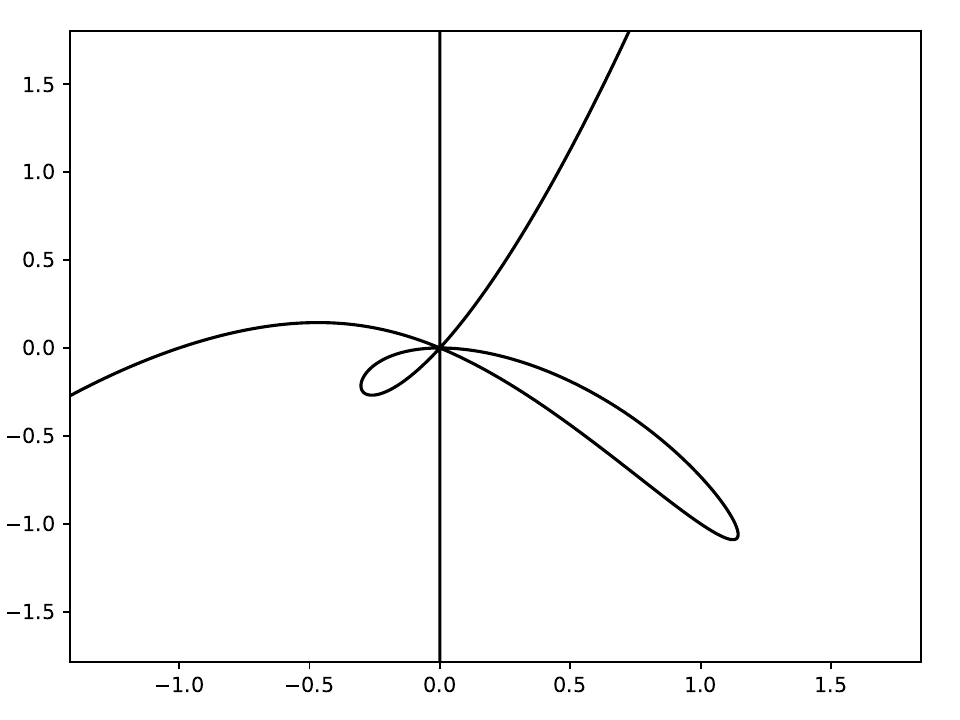}
		\caption{The divisor $E_2$ is contracted.}
		\label{fig:manual_divide_last_before_trans}
	\end{figure}
	As we see, the (perturbed) strict transform does not meet in normal 
	crossings the divisor $E_1$. So we need to perturb it a little bit. We do 
	so by considering the polynomial
	\[
	\begin{split}
		f_{1,t}(x,y)=-&\left(t^{10} + 5 \, t^{8} x + t^{8} + 10 \, t^{6} x^{2} 
		+ 4 
		\, t^{6} x + 
		10 \, t^{4} x^{3} + 3 \, t^{6} y + 6 \, t^{4} \right.\\
		& \quad x^{2} + 5 \, t^{2} x^{4} + 
		t^{5} y + 9 \, t^{4} x y + 4 \, t^{2} x^{3} + x^{5} + 2 \, t^{3} x y + 
		9 \, 
		t^{2} x^{2} y \\
		&\left. \quad + t^{3} y^{2} + x^{4} + t x^{2} y + 3 \, x^{3} y + t x 
		y^{2} 
		- y^{3}\right) x^{3}.
	\end{split}
	\]
	which results from substituting $x$ by $x+t^2$ in the strict transform of 
	the equation defined by $\hat{f}_{2,t}(x,y)$. For $t>0$ the corresponding 
	situation is depicted in Fig.~\ref{fig:manual_divide_last}.
	\begin{figure}
		\centering \includegraphics*[scale=0.6]{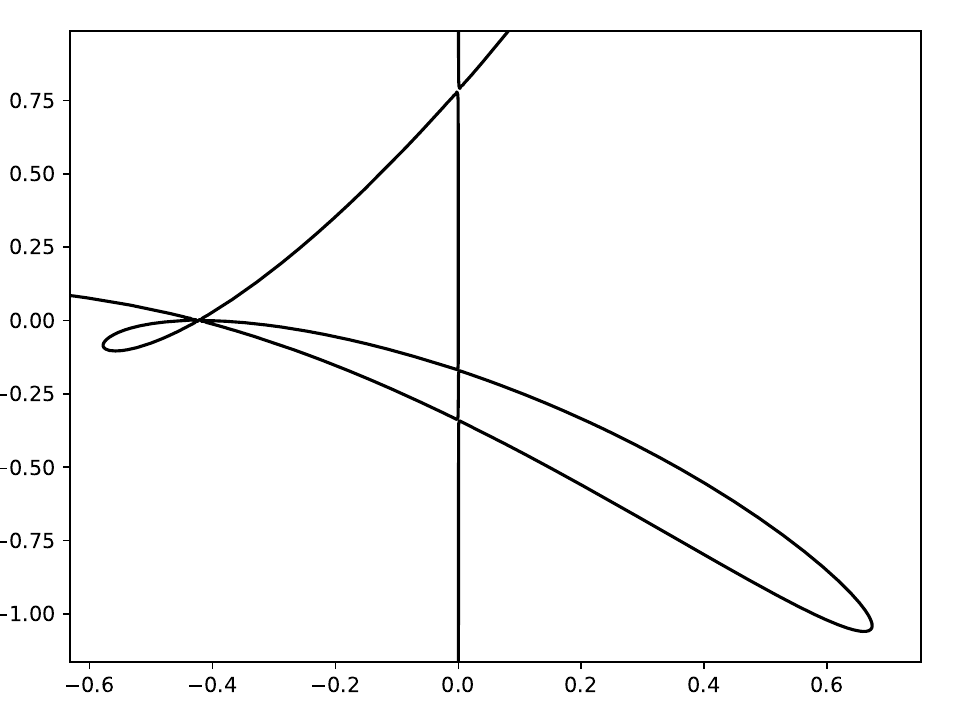}
		\caption{A small translation is performed to achive normal crossings.}
		\label{fig:manual_divide_last}
	\end{figure}
	
	We are ready to contract the divisor $E_1$ to arrive to the polynomial
	\[
	\begin{split}
		\hat{f}_{1,t}(x,y)=&-t^{10} x^{3} - 5 \, t^{8} x^{4} - t^{8} x^{3} - 10 
		\, 
		t^{6} x^{5} - 4 \, t^{6} x^{4} - 10 \, t^{4} x^{6} - 6 \, t^{4} x^{5} \\
		&- 5 \, t^{2} x^{7} - 3 \, t^{6} x^{2} y - 4 \, t^{2} x^{6} - x^{8} - 
		t^{5} 
		x^{2} y - 9 \, t^{4} x^{3} y - x^{7} - 2 \, t^{3} x^{3} y \\
		&- 9 \, t^{2} 
		x^{4} y - t x^{4} y - 3 \, x^{5} y - t^{3} x y^{2} - t x^{2} y^{2} + 
		y^{3}.
	\end{split}	
	\]
	This is already a polynomial defined in the (original) $\mathbb{C}^2$ that 
	is a 
	deformation of $f(x,y)=f_{1,0}(x,y)$ and that defines a pre-divide (recall 
	Definition~\ref{def:pre_divide}) which is depicted in 
	Fig.~\ref{fig:manual_divide_final}.
	\begin{figure}
		\centering \includegraphics*[scale=0.6]{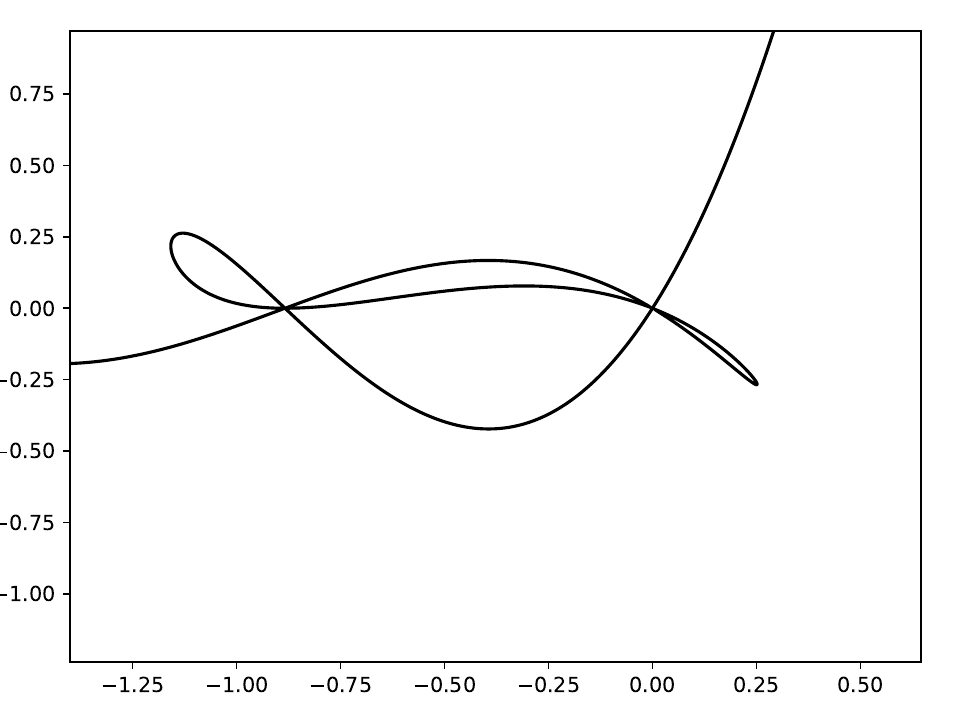}
		\caption{A small translation is performed to achive normal crossings.}
		\label{fig:manual_divide_final}
	\end{figure}	
	After performing a generic pertubation as explained in 
	Lemma~\ref{lem:generic_perturbation} we arrive at a divide like the one in 
	Fig.~\ref{fig:manual_divide_final_perturbed} which is a divide for our 
	original 
	singularity.
	\begin{figure}
		\centering 
		\includegraphics*[scale=0.6]{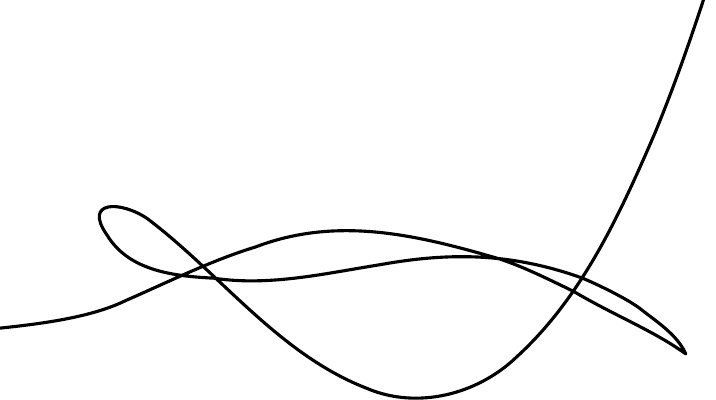}
		\caption{A small perturbation can be performed to get rid of the 
			ordinary singularities and get only double points.}
		\label{fig:manual_divide_final_perturbed}
	\end{figure}	
\end{example} 

\subsection{Divides from  Chebyshev polynomials} \label{ss:divide_chebyshev}

In this section we explain a different way to obtain divides for plane curves. 
This was envisaged by Sabir Gusein-Zade in 
\cite{GZ_inter_two,GZ_dynkin_two} and was also further explored by the first 
author.

It was observed by Ren\'e Thom that the \index{Chebyshev polynomial} Chebyshev  
polynomials $T:\mathbb{C} 
\to \mathbb{C}$ 
up to affine equivalence of functions are precisely the  polynomial mappings
from $\mathbb{C}$ to $\mathbb{C}$ with two or less critical values and with 
only quadratic
singularities \cite{Thom_equi}. The standard Chebyshev 
polynomial $T(p,z)$ is a polynomial in the variable $z$ of degree $p$ and has 
the 
following properties:
\begin{enumerate}
	\item its critical  are values $+1,-1$, 
	\item it has the symmetry $T(p,z)=(-1)^p T(p,-z)$, and 
	\item the coefficient of $z^p$ is $2^{p-1}$. 
\end{enumerate}
As a consequence of these properties we find that, for example, the  map 
$T(1,z)$ has  no critical values and the map $T(2,z)$ has only the critical 
value $-1$. Furthermore, the 
Chebyshev  polynomial $T(p,z)$ satisfies the identity
\[
T(p,\cos(x))=\cos(px),
\]
and its restriction to $[-1,1]$ is defined by
\[
T(p,t)=\cos(p\arccos(t)).
\]
Sabir Gusein-Zade \cite{GZ_inter_two,GZ_dynkin_two} constructed real 
morsifications for
real plane curve singularities with Chebyshev
polynomials. The building block for his construction 
is the real morsification for the map $f(x,y)=2^{p-1}x^p-2^{q-1}y^q$ 
given by
\[
f_s(x,y)=s^{pq}(
T(p,x/s^q)- T(q,y/s^p)), \quad s \in [0,1].
\]
In particular, it is a direct computation that for each $s \in ]0,1]$, the
function $f_s:\mathbb{C}^2 \to \mathbb{C}$ has $\mu_f=(p-1)(q-1)$ quadratic 
singularities all at points with real coordinates, its critical values  
are contained in $\{-2s^{pq},0,2s^{pq}\}$
and 
\[
\lim_{\substack{s \to 0 \\ s>0}}f_s=f.
\]
If the exponents $p$ and $q$ are relatively prime to each other, the 
level set $\{f(x,y)=0\}$ can be
parametrized by the monomial map 
$$
t\in \mathbb{C} \mapsto (t^q/2^{p-1},t^p/2^{q-1}) \in \mathbb{C}^2.
$$
The Chebyshev composition identity (see \cite{Riv}) 
\[
T(p,T(q,z)) = 
T(pq, z),
\] 
implies that in this case, 
the Chebyshev polynomials can also be used to parametrize the 
level sets $\{f_s(x,y)=0\}$ as well. Indeed, the map 
$$
t \in \mathbb{C} \mapsto (s^q T(q,t/s),s^p T(p,t/s)) \in \mathbb{C}^2
$$ 
parametrizes the level set $\{f_s(x,y)=0\}$. Equivalently, one can verify that
\begin{equation}\label{eq:chebyshev_inequality}
	\lim_{s \to 0} \frac{s^p T(p,t/s)}{2^{p-1}} = t^q
\end{equation}

\subsubsection*{Gusein-Zade method} As a consequence of this, Gussein-Zade 
\cite{GZ_dynkin_two} showed that if 
\[
t \mapsto (t^m, \sum \lambda_k t^k)) 
\]
is a parametrization of a branch with $\lambda_k \in \mathbb{R}$, then
\begin{equation}\label{eq:divide_gusein}
	(t,s) \mapsto \left( \frac{s^m T(m,t/s)}{2^{m-1}}, \sum \frac{ \lambda_k 
		s^p 
		T(k,t/s)}{2^{k-1}}  \right)
\end{equation}
defines, for $s \neq 0$, a divide for that branch. Observe that any real branch 
admits a real parametrization. 

\begin{example}
	We compare now this method with the method described in the previous 
	Sect.~\ref{ss:divides_embedded}. In particular, we consider the plane curve 
	singularity defined by the polynomial 
	\[
	f(x,y)= -x^{8} - x^{7} - 3  x^{5} y + y^{3}
	\]
	which is the one used in Example~\ref{ex:divide_translate}. A direct 
	computation 
	yields that the map
	\[
	(t^3, t^7 + t^8)
	\]
	is a parametrization of the branch $\{f(x,y)=0\}$ near the origin. And so 
	the formula eq.~\ref{eq:divide_gusein}, yields that the following describes 
	a 
	divide for each $s \neq 0$:
	\[
	\begin{split}
		(t,s) \mapsto & \left( \frac{s^3 T(3,t/s)}{2^{2}}, \, \frac{s^7 
			T(7,t/s)}{2^{6}} + \frac{s^8 T(8,t/s)}{2^{7}} \right) \\
		= \quad & \left( -\frac{3}{4} \, s^{2} t + t^{3}, \right. \\
		& \quad \left. -\frac{7}{64} \, s^{6} t + \frac{7}{8} \, s^{4} t^{3} - 
		\frac{7}{4} \, s^{2} t^{5} + t^{7} + \frac{1}{128} \, s^{8} \right).
	\end{split}
	\]
	
	Where the equality is just the evaluation of the Chebyshev polynomials and 
	simplification of the expressions. For $s \neq 0$ we obtain the divide 
	depicted in  Fig.~\ref{fig:divide_478}. Compare with that of 
	Fig.~\ref{fig:manual_divide_final_perturbed}
	\begin{figure}
		\centering \includegraphics*[scale=0.6]{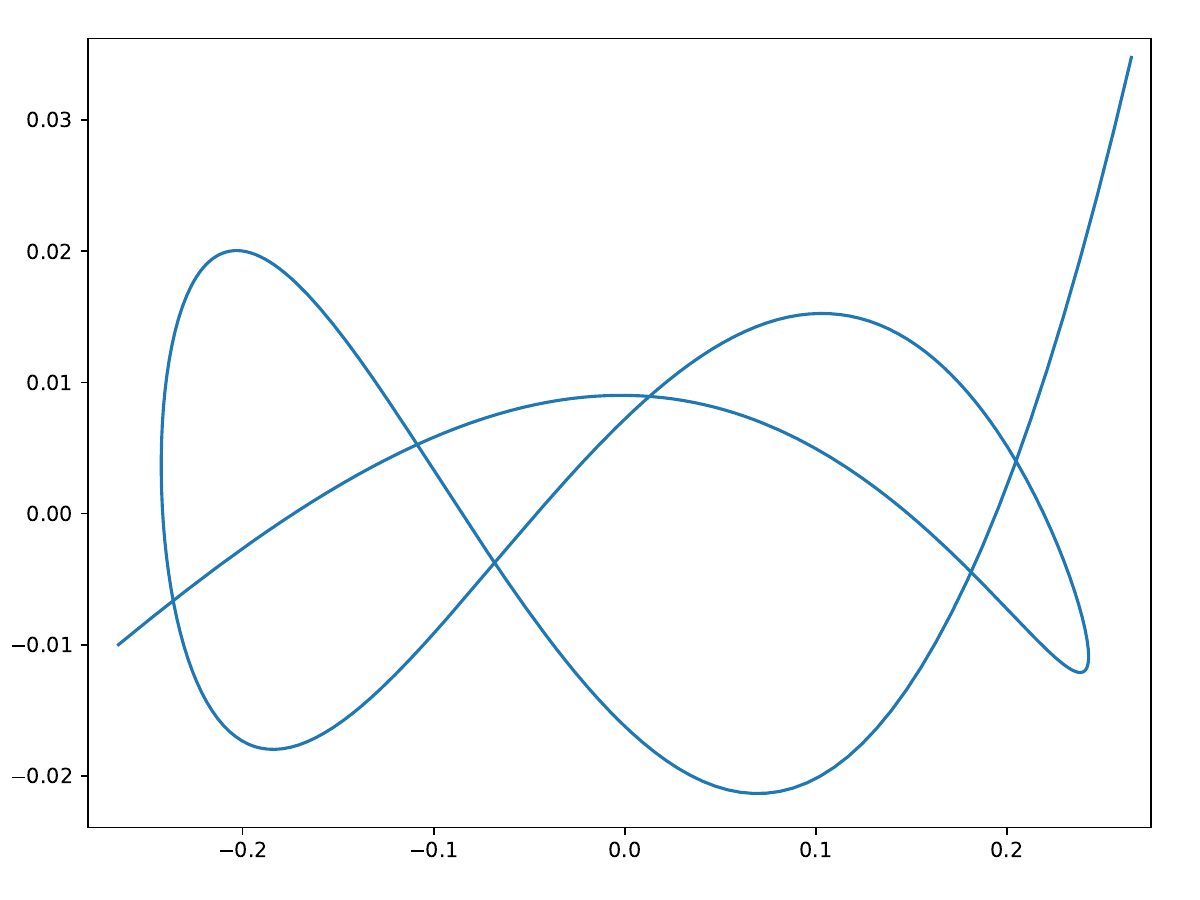}
		\caption{Divide for the plane curve defined by $-x^{8} - x^{7} - 3  
			x^{5} y 
			+ y^{3}$ using Gusein-Zade method.}
		\label{fig:divide_478}
	\end{figure}
\end{example}

It follows from the defining 
property $T(x,cos \theta) = \cos (n \theta)$ that the curve defined by the 
above parametrization has 
\begin{equation}\label{eq:double_points_pq}
	\delta(P_{p,q}) = (p-1)(q-1)/2
\end{equation}
double points near the origin. Since 
$f_s$ has $(p-1)(q-1)$ critical 
points (all of them  Morse), we have proven that the previous 
parametrization yields a divide. More concretely, for each 
$s\in ]0,1]$ the intersection 
\[
P_{p,q;s}:=\{f_s(x,y)=0\}\cap D
\]
is a divide for the Brieskorn-Pham singularity  defined by the polynomial 
$2^{p-1}x^p-2^{q-1}y^q$ 
at $0\in \mathbb{C}^2$. 
The curve  
$P_{p,q}:=\{f_1(x,y)=0\}$ can be drawn in a rectangular box as in 
Fig.~\ref{fig:7_5_divide}. 
As a first type of building block 
we will need 
the box $B:=[-1,1]\times [-1,1]$ with the curve $P_{p,q}$. 
If $(p,q)=1$ 
holds, the curve $P_{p,q}$ is the image of
$$
\begin{aligned}
	T_{p,q}:[-1,1] &\longrightarrow  B \\
	t &\mapsto  (T(p,t),T(q,t))
\end{aligned}
$$
which leaves the box through the corners. In general, the immersed 
curve has several components, which are immersions of the interval or of the
circle. At most two components are immersions of the interval, 
which leave the box through the corners. 

\begin{figure}
	\centering \includegraphics*[scale=0.3]{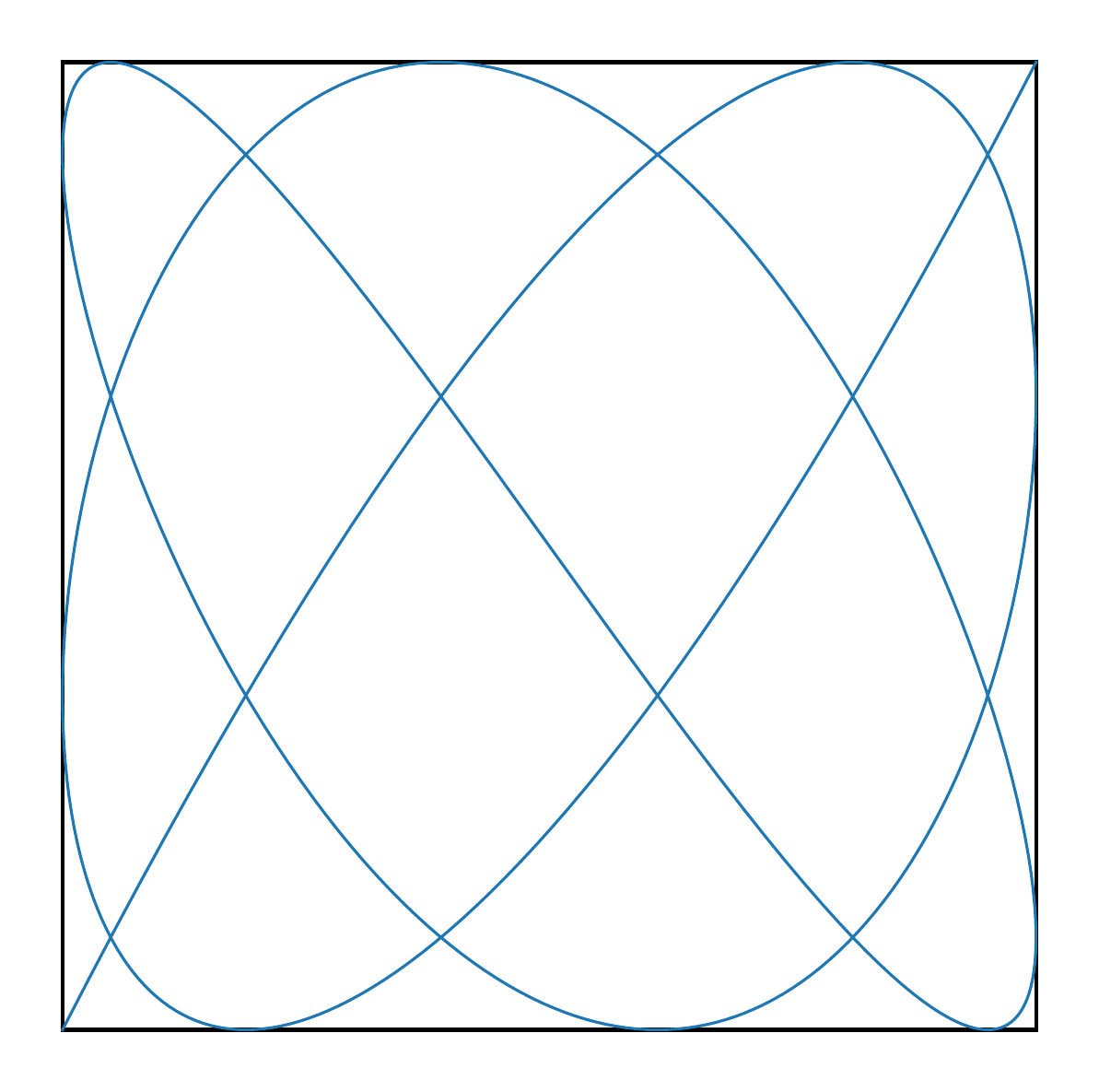}
	\caption{Divide in box $[-1,1]\times [-1,1]$ for  $2^6x^7-2^4y^5=0$.}
	\label{fig:7_5_divide}
\end{figure}

\subsubsection*{Cabling divides}
Let $P$ be any divide having one branch given by an immersion 
$\gamma: [-1,1] \to D$. We assume, that the speed vector 
$\dot{\gamma}(t)$ and the position vector $\gamma(t)$ 
are proportional at $t=\pm 1$, i.e. the
divide $P$ meets $\partial{D}$ at right angles.
Let 
\[
N\gamma:[-1,1]\times [-1,1] \to D
\] 
be an immersion of a
rectangular box around $P$, that is, the restriction 
\[
N\gamma |_{[-1,1]\times \{0\}}
\]
is the immersion $\gamma$ and the image of $N\gamma$ is in a small tubular
neighborhood of $P$. For instance, 
for a small value of the parameter $\eta \in \mathbb{R}_{>0}$
the following
expression defines
such an immersion $N\gamma:B \to \mathbb{R}^2$ of the  rectangular box 
$B:=[-1,1]\times [-1,1]$: 
\[
N\gamma(s,t):=\gamma(t)+s\eta {J(\dot{\gamma}(t)) \over \|\dot{\gamma}(t)\|},
\]
where $J$ is the rotation of $\mathbb{R}^2$ over $\pi/2$, equivalently, $J$ is 
multiplication by $i$ if we see the divide as a subset of $\mathbb{C}$.
The four corners $N\gamma(\pm 1,\pm 1)$ are on the circle of
radius 
$\sigma:=\sqrt{1+\eta^2}.$ 
We finally define 
\[
N\gamma(s,t):=N_{\eta}\gamma(s,t):={1 \over \sigma} \left(\gamma(t)+
s\eta {J(\dot{\gamma}(t)) \over \|\dot{\gamma}(t)\|}\right),
\]
that is an immersion $N\gamma:B \to D$ mapping the corners of 
the box $B$ 
into $\partial{D}$.

\begin{figure}
	\centering
	\centering \includegraphics*[scale=0.6]{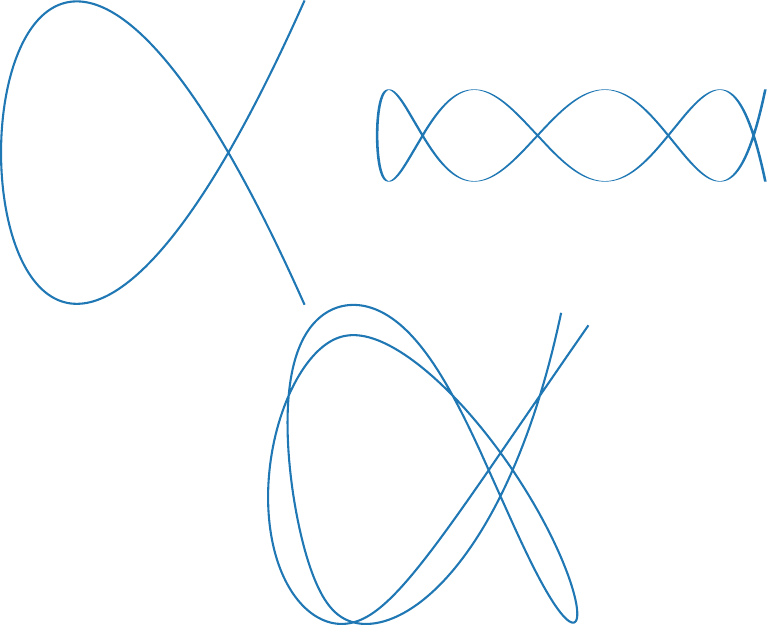}
	\caption{On the upper left part we see the divide $P_{2,3}$, on the upper 
		right part the divide $P_{2,9}$ and on the lower part, the divide 
		$P_{2,9}*P_{2,3}$.}
	\label{fig:29_23}
\end{figure}

\begin{definition}\label{def:cabling_divide}
	Let $P$ be a divide and let $p,q$ be two natural numbers with $\gcd(p,q)$, 
	we 
	denote by 
	\[
	P_{p,q}*P
	\]
	the divide in $D$, 
	which is the image by $N\gamma:B \to D$ of 
	$P_{p,q} \subset B$. We call $P_{p,q}*P$ the {\em $(p,q)$-iterated divide 
		around $P$}. 
\end{definition}
Note that by definition, an iterated $*$-composition of divides
has to be evaluated from
the right to the left.

The number of double
points $\delta(P_{p,q}*P)$ of $P_{p,q}*P$ is computed inductively 
from the number of double points $\delta(P)$ of $P$ by:
\begin{equation}\label{eq:double_points_iterated}
	\delta(P_{p,q}*P)=(p-1)(q-1)/2+\delta(P)p^2.
\end{equation}
Indeed, observe that the divide has:
\begin{enumerate}
	\item  By construction, at least as many 
	crossings as $P_{p,q}$ had (recall 
	eq.~\ref{eq:double_points_pq}), and hence the first summand. 
	\item  Near each intersection point of $P$, there appear after the cabling 
	operation, $p$ lines crosssing with other $p$ 
	lines, and hence the second summand.
\end{enumerate}
Let $R_{\eta}\gamma$ be the union of the image of 
$N_{\eta}\gamma$ with the two
chordal caps at the endpoints of $\gamma$. One can think of $R_{\eta}\gamma$ as 
a thick version of $\gamma$. The connected components
of $D \setminus R_{\eta}\gamma$ correspond via inclusion to
the connected components of $D \setminus P$. We declare a connected 
component of $D \setminus P_{p,q}*P$ to be signed by $+$, if the component 
contains a component of $D \setminus R_{\eta}\gamma$, that corresponds
to a $+$ component of $D \setminus P$. In this case we 
call the connected component
of $D \setminus R_{\eta}\gamma$ a $P_+$-{\it component}. Observe 
that there exists
a chess board sign distribution for the components of 
$D \setminus R_{\eta}\gamma$ that makes $P_+$-components indeed to $+$
components.

The field $\Phi_{p,q}$ of cones on the box $B \subset \mathbb{R}^2$ is 
the subset in the tangent space of 
$TB$ given by:
$$
\Phi_{p,q}:=\{(x,u) \in TB | \,|<u,e_1>_{\mathbb{R}^2}| \geq 
\cos(\alpha(x))\|u\|\}
$$
where $e_1=(1,0)\in \mathbb{R}^2$ and where $\alpha:B \to \mathbb{R}$ is a 
function, such
that for every $(x,u) \in TB$ with $x \in P_{p,q}$ and $u \in T_xP_{p,q}$
we have the equality 
$$
|<u,e_1>_{\mathbb{R}^2}| = \cos(\alpha(x))\|u\|.
$$
Moreover,
$\alpha$ has the boundary values 
$\alpha(\pm 1,t)=0$ and $\alpha(s,\pm 1)=\pi/2$.
We interpolate
the function $\alpha$ on $B$ by upper and lower convexity, i.e
such that ${\partial^2\over{\partial{t}}}\alpha< 0$ and 
${\partial^2\over{\partial{s}}}\alpha> 0$.
The definition of $\alpha(x)$ seems to be 
conflicting at the double points of
the curve $P_{p,q}$; 
at a double point $x=(x_1,x_2)$ of the curve 
$P_{p,q}$ the two tangents lines to $P_{p,q}$
have opposite slopes $\tan(\alpha(x))$ and $-\tan(\alpha(x))$, since
the curve $P_{p,q}$ is defined by  the  equation 
$$
T(q,x_1)-T(p,x_2)=0
$$ 
that separates the variables. 
For example, a nice such function
$\alpha$ is given by:
$$
\alpha(x_1,x_2):=\arctan({q\sqrt{1-x_2^2} \over p\sqrt{1-x_1^2}}).
$$

The interest of the field $\Phi_{p,q}$ comes from the following 
lemma, that follows from the definitions.

\begin{lemma} Let the image of $\gamma:[-1,1] \to D$ be a divide $P$, that
	meets $\partial{D}$ at right angles.
	For $\eta > 0$ small enough, the intersection of $S^3 \subset 
	T\mathbb{R}^2$ 
	with the image in $T\mathbb{R}^2$ 
	of the
	field of sectors $\Phi_{p,q}$ under the differential 
	of $N_{\eta}\gamma$ is a
	tubular neighborhood of the knot $L(P)$.
	The composition of $P_{p,q}:[-1,1]  \to B:=[-1,1] \times [-1,1]$ and of
	$N\gamma:B \to D$ 
	is again a divide, whose knot is a torus cable knot of type $(p,q)$ of the 
	knot $L(P)$.  
\end{lemma}

The image of the field of sectors $\Phi_{p,q}$ of $B$ under 
the differential of 
$N_{\eta}\gamma$ will be denoted by 
$\Phi_{\eta,p,q}\gamma$ and for small $\eta$ contains 
those vectors, that 
have feet near $P$ and form a small 
angle with the tangent vectors of the divide $P$.

\subsubsection*{Divides for plane branches}

Let $\{f_{a,b}(x,y)=0\}$ be a
singularity having one branch and with essential Puiseux pairs (recall 
Sect.~\ref{ss:puiseux_pairs})
$(a_i,b_i)_{1\leq i \leq n}$.
The theorem of S. Gusein-Zade \cite{GZ_inter_two} very efficiently
describes a divide
for the singularity $\{f_{a,b}(x,y)=0\}$ in a closed form, namely
the iteratively composed divide
$$
P_{a_n,b'_n}* \dots *P_{a_2,b'_2}*P_{a_1,b_1},
$$
where the numbers $b'_2, \dots ,b'_n$ can be computed recursively, as
we will show here below.
We denote by $S_k, 1\leq k \leq n$ the divide
$$
P_{a_k,b'_k}*P_{a_{k-1},b'_{k-1}} \dots *P_{a_2,b'_2}*P_{a_1,b_1}
$$ and
let $f_k(x,y)$ be a specific  equation
for a singularity with essential \index{Puiseux pairs} Puiseux pairs 
$(a_i,b_i)_{1\leq i \leq k}$.

Remember, that the product $a_1 a_2 \dots a_k$ is the
multiplicity at $0$ of the curve $\{f_k(x,y)=0\}$ and that 
the linking number $\lambda_k$ of 
$L(S_k)$ and $L(S_{k-1})$ in $S^3$ can be computed recursively by:
$$
\lambda_1=b_1,\, \lambda_{k+1}=b_{k+1}-b_ka_{k+1}+\lambda_ka_ka_{k+1}.
$$

The linking number $\lambda_k$ is equal to the intersection multiplicity 
$$
\dim \mathbb{C}[[x,y]]/(f_k(x,y),f_{k-1}(x,y))
$$ 
at $0$
of the curves $\{f_k(x,y)=0\}$ and $\{f_{k-1}(x,y)=0\}$ (recall the definition 
of Newton pairs eq.~\ref{eq:newton_pairs}).
We have also for the linking number $\lambda_k$ an interpretation in terms 
of divides (see the next section for the first equality) which leads to the 
definition of the numbers $b_k'$:
$$
\lambda_k = \#(S_k \cap S_{k-1})=2a_k \delta(S_{k-1}) + b_k'.
$$
Remembering, that we already have 
computed recursively the numbers $\delta(S_k)$ and $\#(S_k \cap S_{k-1})$,
we conclude that $b_k'$ too can be computed recursively. 

\begin{figure}
	\centering \includegraphics*[scale=0.35]{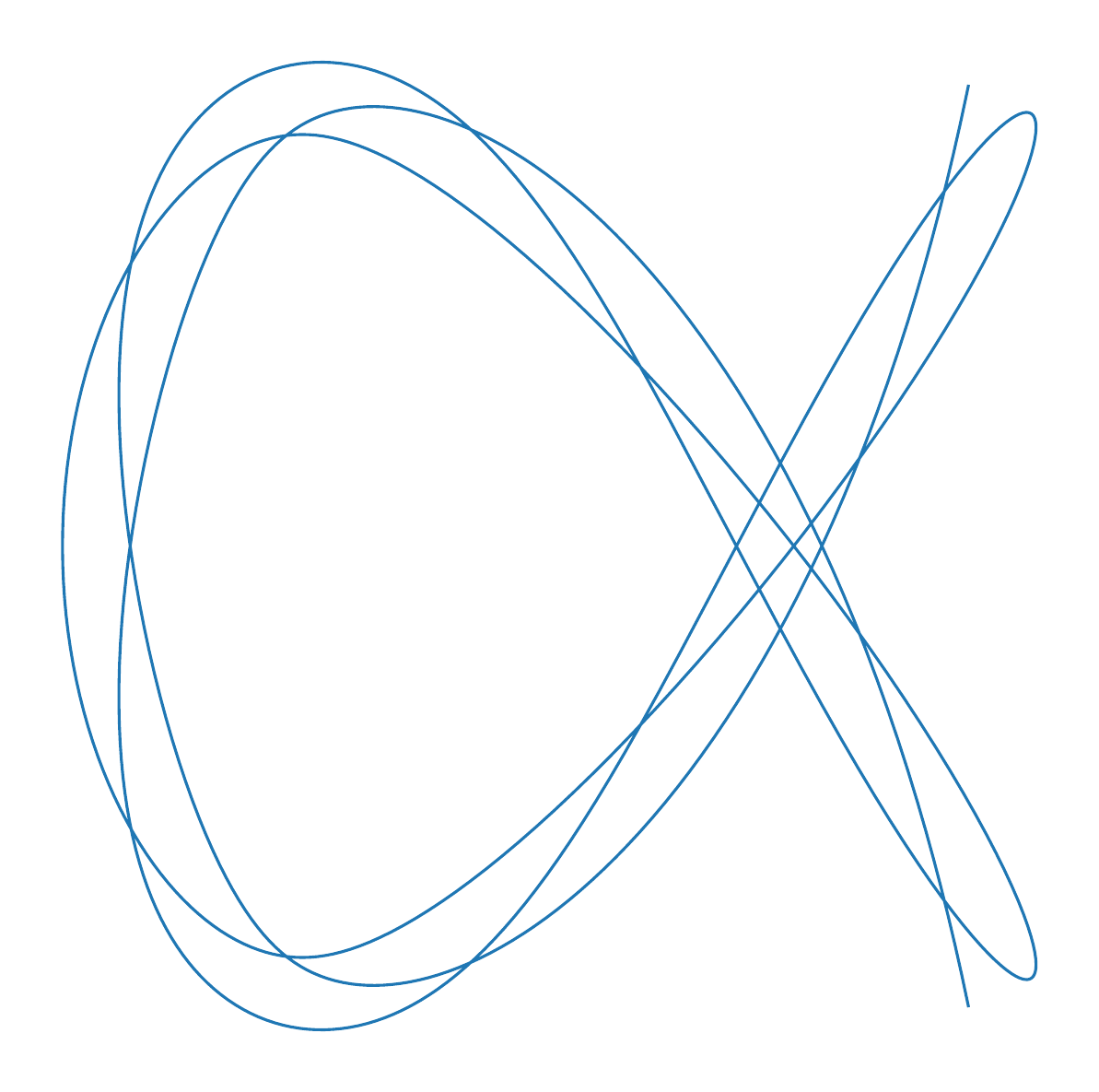}
	\caption{The divide $P_{3,14}*P_{2,3}$ for
		$(y^2-x^3)^3-6x^7y^2-2x^{10}-x^{11}$.}
	\label{fig:3_14_2_3_divide}
\end{figure}

For  example, for
the Puiseux expansion $y=x^{3/2}+x^{7/4}$ we 
have: $\lambda_1=3,\, \lambda_2=7-3\cdot2+3\cdot2\cdot2=13,\, 
b'_2=13-2\cdot2\cdot1=9$.
Hence,  the divide for the irreducible singularity with Puiseux expansion
$y=x^{3/2}+x^{7/4}$ is the divide $P_{2,9}*P_{2,3}$, see Fig.~\ref{fig:29_23}. 
For the Puiseux expansion
$y=x^{3/2}+x^{11/6}$ we found: $\lambda_1=3,\, 
\lambda_2=11-3\cdot3+3\cdot2\cdot3=20,\,
b'_2=20-2\cdot3\cdot1=14$. Hence, the divide for 
its singularity 
$\{(y^2-x^3)^3-6x^7y^2-2x^{10}-x^{11}=0\}$ is 
$P_{3,14}*P_{2,3}$, 
see Fig.~\ref{fig:3_14_2_3_divide}.

Using the contents of the (independent) Sect.~\ref{s:abstract} (c.f. 
\cite{Acampo_real_def,Acampo_generic_imm}), we can read off from this 
divide the Milnor fibration of the
singularity $\{f_{a,b}(x,y)=0\}$. In particular we can describe the Milnor
fiber with a distinguished basis of quadratic vanishing 
cycles. Using the above
iterated cabling construction, we will also be able to read off 
from the divide, the reduction system (Definition~\ref{def:reduction_system}) of
the geometric monodromy of an irreducible plane curve singularity (see 
Sect.~\ref{ss:reduction_system}), as 
described in \cite{Acampo_sur_la}. For instance, 
intersection numbers in the sense of Nielsen of \index{geometric vanishing 
	cycle} quadratic vanishing cycles and
reduction cycles can be computed.

\subsubsection*{Divides for all real plane curve singularities} In general, for 
an isolated singularity of a real polynomial 
$f(x,y)$ having
several local branches, the divide $\{f_1(x,y)=0\} \cap D$ of a 
real morsification $f_t(x,y)$ may have immersed circles as 
components. The above cabling construction 
$P_{p,q}*P$ does not work if the divide $P$ consists
of an immersed circle. Of course, if one is willing to 
change the equation
of the singularity to an equation, which defines a 
topologically equivalent
singularity and which has only real local branches, one will only
have to deal with divides consisting of 
immersed intervals. If we do not want
to change the real equation, we conjecture how a different type of block would 
produce real morsifications for these real singularities. Note that the 
question whether any real singularity admits a morsification or not is still 
open. See \cite{Lev} for some partial results to answer the question in the 
positive and see \cite{Fom} for other recent conjectures that relie on this 
open question.

We conjecture that, if the real plane curve contains pairs of complex conjugate 
branches, then we could still produce divides using a second type of building 
blocks for
a cabling construction, see Fig.~\ref{fig:lissajous_block}.
\begin{figure}
	\centering \includegraphics*{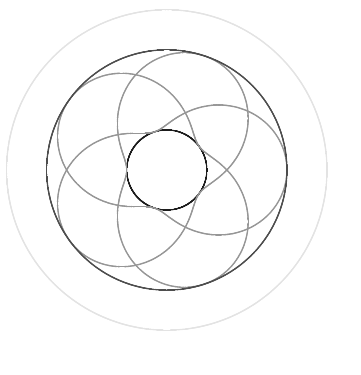}
	\caption{The building block  $L_{3,5}.$ of Lissajous type}
	\label{fig:lissajous_block}
\end{figure}

These building blocks are the divides 
$L_{p,q}$
in the annular region 
$A:=\{(x,y) \in D | 1/4 \leq \sqrt{x^2+y^2} \leq 3/4 \}$.
If for the integers $(p,q)=1$ holds, the divide $L_{p,q}$ is the Lissajous
curve 
$$
s \in [0,1] \mapsto (1/2+1/4 \sin(2\pi q s))( \sin(2\pi p s),\cos(2 \pi p s))
$$
in $A \subset D$. The curve $L_{p,q}$ has $p$-fold rotational
symmetry. If $(p,q)=r > 1$ the divide $L_{p,q}$ is defined
defined as the union of 
$r$ rotated copies of 
$L_{p/r,q/r}$ with rotations of angles $2\pi k/p, k=0 \dots r-1,$ of $D$.
Again, the system of curves $L_{p,q}$ has a $q$-fold rotational symmetry.

The star-product $L_{p,q}*P$ can be defined as above 
if the divide $P$ consists
of one immersed circle. 
The two types of building blocks $P_{p,q}$ and $L_{p,q}$ 
together with the star-products $P_{p,q}*P$ and $L_{p,q}*P$ will 
allow one to describe the iterated cablings of real plane curve
singularities in general.

\subsubsection*{Natural orientations} The link of a divide $P$ 
is 
naturally oriented by the following recipe. Let
$\gamma:]0,1[\to D$ be a local regular 
parametrization of $P$. The orientation of $L(P)$ is such that the map
$t\in ]0,1[ \mapsto \left(\gamma(t), \lambda(t)\dot{\gamma}(t) \right)\in L(P)$ 
is
oriented. Here $\lambda(t)$ is a positive scalar function, which ensures that
the map takes its values in $L(P)$. For a connected divide $P$ we orient
its fiber surface $F_P$ such that the oriented boundary of 
$(F_P \cup L(P),L(P))$ coincides with the orientation of $L(P)$. 

Vanishing
cycles $\delta_c$ where $c$ is a critical point of $f_P$ do not carry
a natural orientation, since, for example, the third power of the geometric 
monodromy
of the singularity $\{x^3-y^2=0\}$ reverses the orientations on the vanishing 
cycles. 

We orient the tangent space 
$TD=D\times \mathbb{R}^2$ so that the orientation of its unit sphere $S^3$ 
as the boundary of its unit 
ball gives that the linking numbers of $L(P_1)$ and $L(P_2)$ are positive 
for generic
pairs of
divides $P_1$ and $P_2$. In fact, the
orientation $TD$ is opposite to its orientation as tangent space. With this 
convention, we have $Lk_{S^3}(L(P_1),L(P_2))=\#(P_1\cap P_2)$, a fact which was 
already used in the previous section. 

\section{Description of the Milnor fiber}\label{s:description_Milnor}

In this section, we give a description of the Milnor fibers associated to a 
divide that lie over the points $\pm 1$ and $\pm i$. The description is done in 
terms of more simple pieces. This decomposition is useful in the forthcoming 
sections where we describe the geometric monodromy associated with the 
divide.

\subsection{Description of the Milnor fibers over $+1$ and $-1$}

\index{Milnor fiber}

Let $P$ be a connected divide and let $\pi_P:S^3\setminus L(P) \to S^1$ 
be its fibration of Theorem~\ref{thm:fibered_link}. In this subsection, we show 
how to 
read off geometrically
the fibers $F_1:=\pi_P^{-1}(1)$ and  
$F_{-1}:=\pi_P^{-1}(-1)$. For our construction we assume the disk $D \subset 
\mathbb{C}$ oriented. We think of its orientation as an orthogonal complex 
structure $J:TD \to TD$. Define
\[
P_+:=\{x \in D \setminus \partial D | f_P(x) > 0, df_P(x) \not=0 \}
\] 
The level
curves of $f_P$ define  a oriented foliation $F_+$ on $P_+$, where 
a tangent vector $u$ to a level of $f_P$ at $x \in P_+$ is oriented so that
$df_P(x)(Ju) > 0.$ Put 

\[
P_{+,+}:=\{(x,u) \in S^3 | x\in P_+,\  u \in T(F_+)\}
\]
and 
\[
P_{+,-}:=\{(x,u) \in S^3 | x\in P_+,\  u \in T(F_-)\},
\] 
where $F_-$ is the foliation with the opposite orientation. 
Put 
\[
F_M:=\{(x,u) \in S^3 | x=M\}
\]  
for a maximum $M,$ and 
\[
F_m:=\{(x,u) \in S^3 | x=m \}
\] 
for a minimum $m$ of 
$f_P$. For saddle points $s$ of $f_P$ (equivalently, double point of $P$), 
define 
\[
F_{s,+}:=\{(x,u) \in S^3 | x=s,\  H_{f_P}(x)(u,u) < 0\}
\]
and 
\[
F_{s,-}:=\{(x,u) \in S^3 | x=s,\  H_{f_P}(x)(u,u) > 0\}
\] 
Observe that the angle in between 
$u,v \in F_m$ or $u,v \in F_M$ is a natural distance function on $F_m$ or 
$F_M,$ which allows us to identify $F_m$ and $F_M$ with a circle. Finally, put 
\[
\partial D_+:=\{x \in \partial D | f_P(x) > 0\}
\]
Let $p_{\mathbb{R}}:S^3 \to D$ be the projection $(x,u) \mapsto x.$
The projection $p_{\mathbb{R}}$ maps each of the 
sets $P_{+,+}$ and $P_{+,-}$ 
homeomorphically to $P_+.$ The sets $F_m$ or $F_M$  are homeomorphic 
to $S^1,$ if $M$ or $m$ is a maximum or minimum of $f_P$ respectively. And 
the sets $F_{s,\pm}$ 
are homeomorphic to a disjoint union of two  open 
intervals if $s$ is a crossing point of $P.$ The set $\partial D_+$ is 
homeomorphic to a disjoint union of 
open intervals. 
We have the following decomposition of $F_1$:
\[
F_1= P_{+,+} \cup P_{+,-} \cup \partial D_+\cup \bigcup_{s \in  P_d} F_{s,+} 
\cup \bigcup_{M \in  P_+} F_M.
\]
Where $P_d$ denotes the set of double points of the divide.
Observe that for $(x,u) \in P_{+,+}\cup P_{+,-}$ we 
have $\theta_P(x,u) \in 
\mathbb{R}_{>0}$ (recall eq.~\ref{eq:theta_func}) since $f_P(x)>0,\ 
df_P(x)(u)=0$ 
and 
$\chi(x)H_{f_P}(u,u)\leq 0$. So $P_{+,+}\cup P_{+,-}$ is an open and dense 
subset in $F_1$.
Accordingly, with the obvious changes of signs, we get a similar 
description 
for $F_{-1}$:
\[
F_{-1}= P_{-,+} \cup P_{-,-} \cup \partial D_-\cup \bigcup_{s \in P_d} F_{s,-} 
\cup \bigcup_{m \in P_-} F_m.
\]

\subsubsection{Combinatorial description} 
Forming 
the closure of $P_{+,+}\cup P_{+,-}$ in $F_1$ leads to
the following combinatorial 
description of the above decomposition. First, we add to 
the open surface $F_1$ its boundary and get 
\[
\bar F_1:=F_1 \cup L(P).
\]
Now let $R$ be a connected component of $P_+.$ The inverse image
$p_{\mathbb{R}}^{-1}(R) \cap \bar F_1$ in $\bar F_1$ are two disjoint open 
cells or
cylinders 
$R_+ \subset P_{+,+}$ and $R_- \subset P_{+,-}$ 
which are in fact subsets of $F_1.$ The closure of 
$R_+$ in  $\bar F_1$ is a surface $\bar R_+$ with boundary  and corners. 
The set $F_M$ is a common boundary component without corners of 
$\bar R_+$ and $\bar R_-$ if $M$ is a maximum in $R.$ If there is 
no maximum in $R$ the closures  $\bar R_+$ and $\bar R_-$ meet along the 
component of $\partial D_+$ which lies in the 
closure of $R.$ Let $S,R$ be  connected components of $P_+$ such that the 
closures of $R$ and $S$ have a crossing point 
$s$ in common. The closures of $R_+$ and $S_-$ in $F_1$ meet along one of 
the components of $F_{s,+}$ and the closures of 
$R_-$ and $S_+$ in $F_1$ meet along the other component of $F_{s,+}.$ The 
closure of 
$F_{s,+} \cap \bar R_+$ in 
$\bar R_+$ intersects $L(P)$ in $2$ corners, that are  also corners of 
the closure of $F_{s,+} \cap \bar S_-$ in 
$\bar S_-$ (see Fig.~\ref{fig:gluing_fol}). Notice that the foliation $F_+$ on 
$P_+$ does 
not lift to a
foliation, which extends to an oriented foliation on $F_1$.

\begin{figure}
	\centering
	\centering \includegraphics*{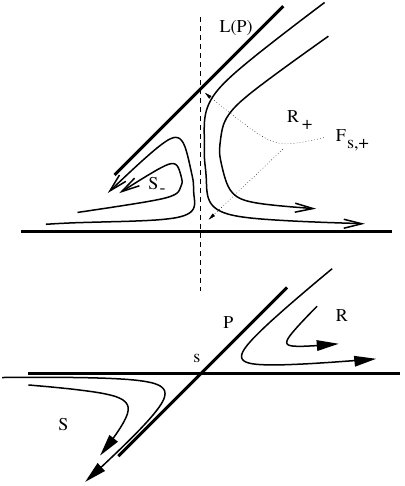}
	\caption{Gluing of the lifts of $R$ 
		with $+$ and $S$ with $-$ foliation to $F_1$.}
	\label{fig:gluing_fol}
\end{figure}

\subsection{Description of the Milnor fibers over $+i$ and $-i$}

\index{Milnor fiber}
Now we quickly work out the fibers $F_i:=\pi_P^{-1}(i)$ and 
$F_{-i}:=\pi_P^{-1}(-i).$ Observe that 
$F_i$ and $F_{-i}$ are projected  to a subset of 
$P \cup \mathrm{supp}(\chi)$ by  $p_{\mathbb{R}}$. Put

\[
F_{i,P}:=\{(x,u) \in S^3 | x \in P,\  \chi(x)=0,\ df_P(x)(u)>0 \}.
\]
For a crossing point $c$ of $P$ we put
\[
F_{i,c}:=\{(x,u) \in S^3 | \chi(x) > 0,\  df_P(x)(u) > 0 ,\  
f_P(x)- {1\over 2}\eta^2\chi(x)H_{f_P}(c)(u,u)=0\}.
\]
Observe that 
\[
f_P(x)- {1\over 2}\eta^2\chi(x)H_{f_P}(c)(u,u) = 
\mathrm{Re}(\theta_{P,\eta}(x,u))
\]
and recall (eq.~\ref{eq:theta_func}). In order to get nice sets it is necessary 
to 
choose 
a nice bump function 
$\chi$ as before.
The set 
\[
F_{i,P}\cup F_{i,c}
\]
is an open and dense subset in $F_i$ and forming 
its closure results in
\[
\bar F_i:=F_i \cup L(P)
\]
leads to a 
combinatorial description of $F_i$ similar to the one given above. 

We would like to  direct the interested reader to some beautiful pictures in 
the works of Sebastian Baader, Pierre 
Dehornoy and Livio Liechti \cite{Deho,Liv} which complement the ones contained 
in this work and which exemplify how the Milnor fiber(s) are recovered from the 
divide. The referee called our attention to these works and we are thankful for 
that.

\section{Descriptions of the monodromy}
\label{s:description_monodromy}

In this section we give a complete description of the \index{geometric 
	monodromy} geometric monodromy of 
the  fibered link $L(P)$.

In Sect.~\ref{ss:monodromy_product}  we give the monodromy as a composition of 
right-handed Dehn 
twists. This \index{geometric monodromy!factorization} factorization, certainly 
completely determines the monodromy but 
it is general not trivial to give the Nielsen-Thurston decomposition from a 
factorization. In Sect.~\ref{ss:other_decompositions} we take care of this and 
give 
the 
minimal reduction system of  curves of  the Nielsen-Thurston decomposition of 
the monodromy.

\subsection{Monodromy as a product of Dehn twists} \label{ss:monodromy_product}
We will use the integral 
curves of  the distribution $J(\ker(df_P)),$ which 
pass through the crossing points of the divide $P.$ In a connected component 
$R$ of $D \setminus P$, those 
integral curves of  
$J(\ker(df_P))$ meet at the critical point of $f_P$ in the component 
$R$ with distinct tangents,
or they go to distinct points of $\partial D.$ 

\begin{figure}
	\centering
	\centering \includegraphics*[scale=0.8]{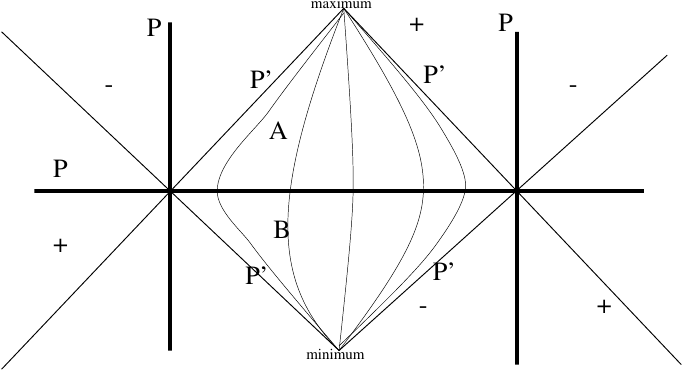}
	\caption{ Two tiles with the $J(\ker(df_P))$ foliation.}
	\label{fig:Jker_fol}
\end{figure}

We denote by $P'$ the union of the integral curves of $J(\ker(df_P)),$ 
which pass through the crossing points of $P.$ The complement in $D$ of the 
union  $P' \cup P \cup \partial D$ is a disjoint 
union of tiles, which are homeomorphic to open squares or triangles 

\begin{notation}
	We call 
	a pair $(A,B)$ of tiles {\em opposite}, if $A\not= B$ and 
	the closures of $A$ and $B$ in $D$ have a segment of $P$ in common. For an 
	opposite pair of tiles $(A,B)$ let $A|B$ be the 
	interior in $D$ of the union of the closures of $A$ and $B$ in $D$. 
\end{notation}

The set
is foliated by the levels of $f_P$ and also by the 
integral lines of the distribution $J(\ker(df_P)).$ Both foliations are 
non-singular and 
meet in a $J$-orthogonal way (see Fig.~\ref{fig:Jker_fol}). 

\begin{notation}
	For a pair $(A,B)$ of opposite we put,	
	\[
	F_{1,A|B}:=\{(x,u) \in F_1 | x \in A|B \}
	\]
	and
	\[
	F_{-1,A|B}:=\{(x,u) \in F_{-1} | x \in A|B \}
	\]
	The sets $F_{1,A|B}$ and $F_{-1,A|B}$ each have  two connected components:
	
	\[
	F_{1,A|B}=F_{1,+,A|B} \cup F_{1,-,A|B}
	\]
	where
	\[
	F_{1,+,A|B}:=\{(x,u) \in F_1 | x \in A|B,\  df_P(Ju) > 0 \}
	\]
	\[
	F_{-1,+,A|B}:=\{(x,u) \in F_{-1} | x \in A|B,\  df_P(Ju) > 0 \}
	\]
	and
	\[
	F_{1,-,A|B}:=\{(x,u) \in F_1 | x \in A|B,\  df_P(Ju) < 0 \}
	\]
	\[
	F_{-1,-,A|B}:=\{(x,u) \in F_{-1} | x \in A|B,\  df_P(Ju) < 0 \}
	\]
\end{notation}
The closures of $F_{1,\pm,A|B}$ in $\bar F_1$ and of $F_{-1,\pm,A|B}$ in 
$\bar F_{-1}$ are  polygons with $6$ edges: let 
$M,c,c'$ be the vertices of the triangle $A$; the six edges of the closure 
$H$ of $F_{1,+,A|B}$ in $\bar F_1$ are
\[
\begin{array}{ll}
	\{(x,u) \in H \mid x=M\}, \quad
	& \{(x,u) \in H \mid x \in [c,M]\}, \\[0.3em]
	\{(x,u) \in H \mid x=c\}, 
	& \{(x,u) \in H \mid x \in [c,c']\}, \\[0.3em]
	\{(x,u) \in H \mid x=c'\}, 
	& \{(x,u) \in H \mid x \in [c',M]\},
\end{array}
\]
where $[M,c]$ and $[M,c']$ are segments included in $P'$ and $[c,c']$ is a 
segment in $P.$ Next, we define two diffeomorphisms
\[
S_{i,A|B}: F_{1,A|B} \to  F_{-1,A|B}
\]
and
\[
S_{-i,A|B}: F_{1,A|B} \to  F_{-1,A|B}
\]
for each pair of opposite tiles $(A,B)$. To do so we choose 
the adapted function $f_P:D 
\to \mathbb{R}$ (recall 
Definition~\ref{def:adapted_function} and 
Lemma~\ref{lem:existence_adpated_function}) 
such 
that the maxima are of value $1$ and the minima of value $-1.$ Moreover, we
modify the function $f_P$ at the boundary $\partial{D}$ such that
along each of the
integral lines of the foliation given by the distribution $J(\ker(df_P))$ 
the function $f_P$ takes  
all values in
an interval $[-m,m]$ with $1 \geq m > 0$. The latter modification of $f_P$ is
useful if the tile $A$ or $B$ meets $\partial{D}$. We also 
need the 
rotations $J_{\theta}:T(D) \to T(D)$ about the angle $\theta \in [-\pi,\pi].$ 
Recall that the complex structure $J$ is precisely $J_{\pi/2}$.

We are now ready to define the map $S_{i,A|B}$. Let $(x,u) \in F_1$ with $x\in 
A|B$ and do as follows:
\begin{enumerate}
	\item  let 
	$y\in A|B$ be the point in the opposite tile on the integral line of the 
	distribution $J(\ker(df_P))$ with $f_P(x)=-f_P(y);$ 
	\item  now move $x$ to $y$ along
	the integral curve $\gamma_{x,y}(t), t \in [f_P(x),f_P(y)]$  which connects 
	$x$ and $y$ with the parameterization  
	$f_P(\gamma_{x,y}(t))=t$; 
	
	\item consider the rotation angle function 
	$\theta(x,t):={(|f_P(x)|-t)\pi \over 2|f_P(x)|}$ and move the vector
	$u$ along the  path
	\[
	\left( \gamma_{x,y}(t),U_{x,y}(t)\right) :=\left(  \gamma_{x,y}(t),
	s(x,t)( J_{\theta(x,t)}(|f_P(x)| u_{x,y}(t)/2)+ u_{x,y}(t) )  \right),
	\]
	where $(\gamma_{x,y}(t),u_{x,y}(t)) \in S^3$ is 
	the continuous vector field 
	along $\gamma_{x,y}(t)$ such that $df_P(u_{x,y}(t))=0$, 
	we have the equality $u_{x,y}(f_P(x))=U_{x,y}(f_P(x))=u$, 
	and the stretching factor $s(x,t) \geq 1$ is chosen such that
	$(\gamma_{x,y}(t),U_{x,y}(t)) \in S^3$ holds.
	\item We finally
	define
	\[
	S_{i,A|B}((x,u)):=\left( y,u_{x,y}(f_P(y)) \right)= 
	\left(y,U_{x,y}(f_P(y))\right)
	\]
\end{enumerate}
The definition of $S_{-i,A|B}$ 
is analogous, but uses rotations  in the sense of $-J$. 

The names $S_{i,A|B}$ or $S_{-i,A|B}$ indicate that the flow lines 
$(\gamma_{x,y}(t),u_{x,y}(t))$ pass through the fiber 
$F_i$ or  $F_{-i}$ respectively. The flow lines defining $S_{i,A|B}$ or 
$S_{-i,A|B}$ are 
different. However, the maps $S_{i,A|B}$ and 
$S_{-i,A|B}$ are equal. The system of paths $(\gamma_{x,y}(t),U_{x,y}(t)) \in 
S^3$ 
is local
near the link $L(P)$, i.e. for every neighborhood $N$ in $S^3$ of a point
$(x',u') \in L(P)$ there exists a neighborhood $M$ of $(x',u')$ in $S^3$ 
such that
each path $(\gamma_{x,y}(t),U_{x,y}(t))$ with $(x,u) \in F_1 \cap M$ stays in
$N$. It will follow that the flow lines of the monodromy 
vector field are  meridians of the link $L(P)$ in its neighborhood. 

\begin{remark}\label{rem:rem_fib_near_link}
	After verifying that the constructed flow lines form a monodromy flow for 
	the 
	$\pi_P$, this last fact proves that the map $\pi_{P,\eta}$ is a fibration 
	near 
	the 
	link $L(P)$. Proving this was a step in the proof of 
	Theorem~\ref{thm:fibered_link} 
	and this is an alternative proof.
\end{remark}
The partially defined diffeomorphisms $S_{i,A|B}$ and 
$S_{-i,A|B}$ glue to diffeomorphisms
\[
S_i,S_{-i}:\bigcup_{A|B} F_{1,A|B} \to \bigcup_{A|B} F_{-1,A|B}
\]
where the sum ranges over all
opposite pairs of tiles $(A,B)$ with $A\subset P_+$. The gluing poses 
no problem 
since those  unions are disjoint, but  the diffeomorphisms $S_i$ and $S_{-i}$
do not extend continuously to $F_1.$ We will see that the 
discontinuities, which are the obstruction for extending $S_i$ and 
$S_{-i},$ can be 
compensated by a composition of  right half Dehn twists (recall 
Definition~\ref{def:half_twist}).

\begin{figure}
	\centering
	\centering \includegraphics*{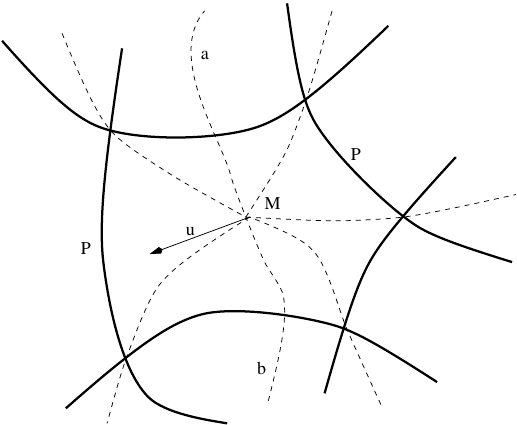}
	\caption{The discontinuity at $F_M.$}
	\label{fig:disc_F_M}
\end{figure}

At a maximum $M \in D$ of $f_P$ each vector $(M,u) \in S^3$ belongs to $F_1.$ 
Let $a$ and $b$ be the integral curves 
of $J(\ker(df_P))$ with one endpoint at $M$ and  orthogonal to $u.$ We 
assume that neither $a$ nor $b$ passes through 
a crossing point of $P$ (see Fig.~\ref{fig:disc_F_M}) and that $a$ and $b$ 
belong to
different pairs of opposite tiles. A continuous extension of the 
maps $S_i$ or $S_{-i}$ has to map the vector $(M,u)$ 
to  two vectors based at the other 
endpoint of $a$ and $b.$ Since these endpoints 
differ in general, a continuous extension is impossible. 

In order to allow a continuous extension at the common endpoint of 
$a$ and $b$ we make a new surface $F'_1$ by 
cutting $F_1$ along the cycles $F_M,$ where $M$ runs through all the 
maxima of 
$f_P$ and by gluing back after a rotation of angle $\pi$ of each of the 
cycles  
$F_M.$ In the analogous manner, we make the surface $F'_{-1}$ in doing 
the half twist along $F_m,$ where $m$ runs through the minima of $f_P.$ The 
subsets $F_{1,A|B}$ do not meet the support of the 
half twists, so they are canonically again subsets of $F'_1,$ which we denote 
by $F'_{1,A|B}.$ Analogously, we have 
subsets $F'_{-1,A|B}$ in $F'_{-1}.$ 
A crucial observation is that the partially defined diffeomorphisms 
\[
S'_i,S'_{-i}:\bigcup_{A|B}F'_{1,A|B} \to \bigcup_{A|B}F'_{-1,A|B}
\]  
have less discontinuities, which are the obstruction for a continuous extension.
We denote by $a'$ and $b'$ the arcs on $F'_1,$ which correspond to 
the arcs $a$ and $b$ on $F_1.$ 
Indeed, the continuous extension  at the  end points of $a'$ and $b'$  
is now possible. 

Let $s$ be a crossing point of $P$ and let $I_{s,+}$ be the segment 
of $P',$ which passes through $s$ and lies in $P_+.$ The inverse 
image of $Z^{\circ}_s:=p_{\mathbb{R}}^{-1}I_{c,+} \cap F_1$  is not a cycle, 
except 
if both endpoints of $I_{s,+}$ lie on $\partial D.$ If a maximum 
$M$ of $f_P$ is an endpoint of $I_{s,+},$ the inverse image 
$p_{\mathbb{R}}^{-1}(M) \cap F_1$ consists of $2$ points on $F_M,$ 
which are antipodal. On the new surface $F'_1$ the inverse image 
$p_{\mathbb{R}}^{-1}(I_{s,+}) \cap F'_{-1}$ is a cycle. 
An extension of 
$S'_i$ and $S'_{-i}$ will be discontinuous along this cycle 
(see Fig.~\ref{fig:disc_Z_s}).
We now  observe that the partially defined diffeomorphisms $S'_i$ 
and 
$S'_{-i}$
have discontinuities along the cycle 
$p_{\mathbb{R}}^{-1}(I_{s,+}) \cap F'_{-1}$, which can be compensated 
by half twists along the 
inverse images $p_{\mathbb{R}}^{-1}(I_{s,-}) \cap F'_{-1},$ where $s$ 
runs through the crossing points of $P.$ 
Note that for a crossing point $s$ of $P$ the curve
$Z'_{s,-1}:=p_{\mathbb{R}}{-1}(I_{s,-}) \cap F'_{-1}$ is in fact a simply 
closed curve on $F'_{-1}$.

\begin{figure}
	\centering
	\centering \includegraphics*{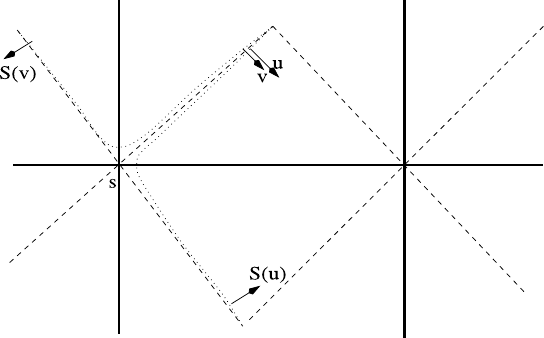}
	\caption{The discontinuity along $Z_s$.}
	\label{fig:disc_Z_s}
\end{figure}

For a crossing point $s$ of the divide $P$ we now define a simply 
closed curve on $F_1,$ by putting:
$$
Z_s:=Z^{\circ}_s \cup \bigcup_{M \in \partial I_{s,+}} F_{s,M}
$$
where for an endpoint $M$ of $I_{s,+},$ which is a maximum of $f_P,$ 
the set 
$ F_{s,M}$ is the simple arc of $F_M,$ which connects the two points 
of $Z^{\circ}_s \cap F_M$ and contains an 
inward tangent vector of $I_{s,+}$ at $M.$ As we already have noticed 
the set $Z^{\circ}_s \cap F_M$ has only 
one element if $M \in \partial D,$ so we define $F_{s,M}:=\emptyset$ 
in that case.

We have the inclusion $F_m \subset F_{-1}.$ We now define the cycle 
$Z_m \subset F_1.$ Define for a minimum $m$ of $f_P$ the region
$$
B_m:= \bigcup_{A|B, m \in \bar{B}} F_{-1,A|B}
$$
Let $B_{m,\epsilon}$ be the level curve
$$
B_{m,\epsilon}:=\{(x,u) \in B_m | f_P(x)=-\epsilon\}
$$
For a small $\epsilon$ the set 
$$(S_i)^{-1}(B_{m,\epsilon} \cap \bigcup_{A|B} F_{-1,A|B})$$ 
is a union of copies of an open interval and is not a cycle but 
nearly a cycle. The union closes up to a cycle by 
adding small segments which project to the integral lines through the 
crossing points of $P.$ We denote this cycle by 
$Z_m \subset F_1$.

We are now able to state the main theorem.

\begin{theorem} \label{thm:factorization_monodromy}
	Let $P$ in $D$ be a connected divide. Let 
	$\pi_P:S^3\setminus L(P) \to S^1$ be the fibration of 
	Theorem~\ref{thm:fibered_link} 
	The counter 
	clockwise monodromy of the fibration
	$\pi_P$ is the composition of right handed Dehn twists 
	\[
	T:=T_- \circ T_. \circ T_+:F_1 \to F_1,\]
	where $T_-$ is 
	the product of the right handed twists along $Z_m,$ $m$ running through the 
	minima of $f_P,$ $T_.$ is the product of the right handed Dehn 
	twists along  the cycles $Z_s,$ $s$ 
	running through the crossing points of $P,$ and $T_+$ is the product 
	of the right handed twists along $F_M,$  $M$ running through the maxima of 
	$f_P.$
\end{theorem}

\begin{proof} We need to introduce 
	one more surface. 
	Let $F''_{-1}$ 
	be the surface obtained from the surface $F'_{-1}$ by cutting $F'_{-1}$ 
	along the cycles $Z'_{s,-1}$ and by gluing back after a half
	twist along each $Z'_{s,-1},$ with $s$ running through the 
	crossing points of $P$. We still have partially defined diffeomorphisms
	$$
	S''_i,S''_{-i}:\bigcup_{A|B}F'_{1,A|B} \to \bigcup_{A|B}F''_{-1,A|B}
	$$
	since the cutting was done in the complement of 
	$\bigcup_{A|B}F'.$ By a direct inspection we see that the 
	diffeomorphisms extend continuously to
	$$
	S''_i,S''_{-i}:F'_1 \to F''_{-1}
	$$
	Let 
	$$(p_+,q_+):F_1 \to F'_1$$
	$$(p_.,q_.):F''_{-1} \to F'_{-1}$$
	$$(p_-,q_-):F'_{-1} \to F_{-1}$$
	be minimal positive pairs of Dehn twists (recall 
	Definition~\ref{def:half_twist}).
	A direct inspection  shows that the composition
	$$
	p_+ \circ S'_i \circ p_. \circ p_- \circ (q_+ \circ S'_{-i} \circ q_. \circ 
	q_-)^{-1}:F_1 \to F_1
	$$
	is the monodromy of the fibration $\pi_P.$
	This composition evaluates to
	$$
	T_- \circ T_. \circ T_+:F_1 \to F_1
	$$
	
\end{proof}

\subsubsection*{Some consequences}

Next we state some inmediate consequences of the previous constructions.
\begin{remark}
	We list some special properties of the monodromy of links and knots of 
	divides.
	The number of Dehn twists of the above decomposition of the  monodromy 
	equals the 
	first betti number $\mu=2\delta - r +1$ of the fiber, and the total number 
	of 
	intersection points among the core curves of the involved Dehn twists 
	is less then $5\delta.$ This means that the {\em complexity} of
	the monodromy is bounded by a function of $\mu$. For instance, the
	coefficients of the Alexander polynomial of the link of a divide are 
	bounded
	by a quantity, which depends only on the degree of the 
	Alexander polynomial. This observation suggests the
	following definition for the complexity $C$ of an element of the mapping 
	class
	group $\phi$ of a surface: the minimum of 
	the quantity $L+I$ over all decompositions 
	as product of Dehn twists of $\phi,$ where $L$ is the number of factors and
	$I$ is the number of mutual intersections of the core curves. We do not know
	properties of this exhaustion of the mapping class group. Notice, that the
	function $(\phi,\psi) \mapsto C(\psi^{-1}\circ \phi)\in \mathbb{N}$ defines 
	a left
	invariant distance  on the mapping class group.
\end{remark}

\begin{remark}
	It can be seen (see next Sect.~\ref{ss:other_decompositions}) that for any 
	link of a divide the  
	monodromy diffeomorphism and its inverse are conjugate by an
	orientation reversing element in the mapping class group. In our previous
	notations this conjugation is given by the map 
	\[
	(x,u)\in F_1 \mapsto (x,-u) \in
	F_1,
	\]
	which moreover  realizes geometrically the symmetry of G. Torres 
	\cite{Torres}
	\[
	t^{\mu}\chi(1/t)=(-1)^{\mu}\chi(t)
	\]
	for the Alexander polynomial $\chi(t)$ of knots.
\end{remark}

\begin{remark}
	In fact the proof of Theorem~\ref{thm:factorization_monodromy} shows that 
	the 
	fibration of the link of a 
	connected divide $P$ can be filled with a singular fibration in the 
	$4$-ball, which has $3$ singular fibers with only quadratic singularities, 
	as in the case of a divide of the singularity of a complex plane curve. 
	The filling has only two singular fibers if the 
	function $f_P$ has no maxima  or no minima. By this construction  
	from a connected divide we obtain a contractible $4$-dimensional piece with 
	a 
	Lefschetz pencil. This is part of what is usually called {\em Hurwitz 
		equivalence}. For more on this topic, see the classical references by 
	Kas 
	\cite{Kas} or Matsumoto \cite{Mat}, or the more recent by Baykur and Hayano 
	\cite{Bay_Hay}.
\end{remark}
We make the following important observation that we have not lost generality by 
considering only divides which are generic immersions of intervals.
\begin{remark}
	It is important to note that 
	Theorems~\ref{thm:fibered_link} and \ref{thm:factorization_monodromy} 
	remain true for 
	generic 
	immersions of 
	disjoint unions of intervals and circles in the 
	$2$-disk. It is also possible to start with a generic immersion of a
	1-manifold $I$ in an oriented compact connected surface with boundary $S.$
	The pair $(S,I)$ defines a link $L(S,I)$ in the $3$-manifold 
	\[
	M_S:=T^+(S)/\hbox{\rm zip},
	\]
	where $T^+(S)$ is the
	space of oriented tangent directions of the surface $S$ and 
	where $\hbox{\rm zip}$ is the
	identification relation, which identifies $(x,u),(y,v)\in T^+(S)$ 
	if and only if 
	$x=y\in \partial S$ or if $(x,u)=(y,v).$ In order to get a fibered 
	link,  the topological pair  
	$(R,R \cap \partial S)$ has to be contractible for each connected 
	component $R$
	of $S\setminus I$ and moreover, the complement $S\setminus I$ has to allow a
	chess board coloring in positive and negative regions. For more details on 
	this 
	construction, see \cite{Ishi}.
\end{remark}
\subsection{Other decompositions of monodromy} 
\label{ss:other_decompositions}
In this subsection we study yet a decomposition of the {\em algebraic} 
monodromy in 
terms of involutions. The contents of this section come mainly from 
\cite{Acampo_mon_real}.

\subsubsection*{Digression in higher dimensions} We 
start with a discussion about geometric monodromies of isolated hypersurface 
singularities in general dimension (recall the discussion of the subsection
\ref{ss:milnor_fibration} on page \pageref{ss:milnor_fibration}).

Let $f:\mathbb{C}^{n+1} \to \mathbb{C}$ be a map defined by a 
polynomial. We assume that $f(0)=0$ and that $0 \in
\mathbb{C}^{n+1}$ is an isolated critical point of $f$. For $p \in 
\mathbb{C}^{n+1}$ let 
\[\|p\|
=
\sqrt{|z_0(p)|^2+|z_1(p)|^2+ \cdots +|z_n(p)|^2}.
\]
Let $B_{\epsilon}$ be a Milnor ball for the singularity of $f$ 
and let 
\[
{\rm Tube}_{\epsilon,\delta}:=\{ p \in B_{\epsilon} \mid |f(p)| \leq \delta 
\}, 0 < \delta << \epsilon,
\]
be a regular tubular neighborhood of $\{p \in B_{\epsilon} \mid f(p)=0 \}$ in
$B_{\epsilon}$. A {\em monodromy vector field} $X$ for the singularity is a 
smooth 
vector field
\[
p \in {\rm Tube}_{\epsilon,\delta} \mapsto X_p \in \mathbb{C}^{n+1}
\]
such  that we have the following properties for 
$p \in {\rm Tube}_{\epsilon,\delta}$ 
\begin{itemize}
	
	\item $(df)_p(X_p)=2\pi i f(p)$, 
	
	\item $X_p$ is tangent to $\partial{B_{\epsilon}}$ if 
	$p \in \partial{B_{\epsilon}}$,
	
	\item trajectories of $X$  starting at $p \in \partial{B_{\epsilon}}$
	are periodic with period $1$ and are the boundary of a smooth disk in 
	$\partial{B_{\epsilon}}$,
	that is transversal to the function $f$. For this property to be satisfied, 
	the 
	hypothesis that $f$ has a critical point is crucial.
\end{itemize}

Using partition of unity, one can construct monodromy vector fields. The
flow at time $1$ of a monodromy vector field $X$ defines a
monodromy diffeomorphism 
$T=T_X\colon F \to F$, where the manifold with boundary
$(F,\partial(F)):=\{ p \in B_{\epsilon} \mid f(p)=\delta \}$ is 
the \index{Milnor fiber} Milnor fiber of the singularity. The 
relative isotopy class 
of the diffeomorphism $T$ is independent from the chosen 
monodromy vector field and is called the {\em geometric monodromy} of the 
singularity. 
The geometric monodromy is a topological invariant of the singularity 
(see \cite[Theorem 3]{King} for $n \neq 2$ and \cite{Perr} for $n=2$).

From now on we will assume in addition, that the polynomial $f$ is real 
meaning that its coefficients are real numbers. Let $c:\mathbb{C}^{n+1} \to 
\mathbb{C}^{n+1}$
denote the involution on complex space 
given by the complex conjugation of coordinate values.
Hence with the above notations, we 
have $c({\rm Tube}_{\epsilon,\delta})={\rm Tube}_{\epsilon,\delta}$ and 
$c(F)=F$. We denote by $c_F:F \to F$ the restriction of the involution 
$c$ to $F$.

Let $X\colon {\rm Tube}_{\epsilon,\delta} \to \mathbb{C}^{n+1}$ be a 
monodromy vector field for the isolated singularity of $f$. We may 
assume that we have 
constructed the vector field $X$ with more care near the 
boundary of 
the Milnor ball in order to achieve  
that for some 
$\epsilon' < \epsilon$
we have the symmetry
\[
c(X_p)=-c(X_{c(p)}), \text{ for } p \in {\rm Tube}_{\epsilon,\delta}, \text{ 
	with } 
\|p\| > \epsilon'.
\]
Since $f$ is real, we have
$$
c((df)_p(X_p))=(df)_{c(p)}(c(X_p))=-2\pi i f(c(p))
$$
hence, we see (by substituting $q$ for $c(p)$ and accordingly $c(q)$ for $p$) 
that the vector field $X^c$ 
defined by:
$$
q \in {\rm Tube}_{\epsilon,\delta} \mapsto X^c_q:=-c(X_{c(q)}) \in 
\mathbb{C}^{n+1},
$$
is  a monodromy vector field too. Let 
$Y:{\rm Tube}_{\epsilon,\delta} \to \mathbb{C}^{n+1}$ be the  vector field
\[
Y:={X+X^c \over 2},\]
which due to the extra care is also 
a monodromy 
vector field. We have $Y^c=Y$. 
The following is an important symmetry of the
geometric monodromy:

\begin{lemma} Let $T_Y$ be a  monodromy diffeomorphism, which has 
	been computed with a monodromy vector field $Y$ satisfiyng  
	$Y^c=Y$. We have the symmetry
	$$
	c_F \circ T_Y \circ c_F =T_Y^{-1}.
	$$ 
	The geometric
	monodromy $T$ satisfies (up to relative isotopy) the symmetry
	$$
	c_F \circ T \circ c_F =T^{-1}.
	$$
\end{lemma}

\begin{proof} The restriction 
	of complex conjugation 
	$c_{{\rm Tube}_{\epsilon,\delta}}\colon {\rm Tube}_{\epsilon,\delta} 
	\to {\rm Tube}_{\epsilon,\delta}$ maps the monodromy vector field
	$Y$ to $-Y$ and $F$ to $F$. Hence, since $c_{{\rm Tube}_{\epsilon,\delta}}$ 
	reverses the orientations of the 
	trajectories, we have 
	$$
	T_Y^{-1}=(c_F)^{-1} \circ T_Y \circ c_F=c_F \circ T_Y \circ c_F.
	$$
	Since the geometric monodromy $T$ is in  the  relative mapping class 
	represented by $T_Y$, we have the symmetry $c_F 
	\circ T \circ c_F =T^{-1}$ at the mapping class group level.
\end{proof}

Symmetries of monodromies as in the lemma 
can occur in the more general context of  so-called 
{\it strongly invertible knots}, see for instance 
\cite{Teissier,Heil_Toll,Kawa}.

See also the work of Sabir 
Gusein-Zade \cite{GZ_Index}, in which he shows
among other results, that the integral \index{monodromy!homological} 
homological monodromy of an isolated 
complex hypersurface singularity
with real defining equation $f$ is the composition of two involutions, both 
conjugated to the action of complex
conjugation on the homology of a real regular fiber. 

\begin{remark}
	The symmetry property $c_F \circ T \circ c_F =T^{-1}$ expresses that
	the geometric monodromy $T$ of a complex hypersurface with real 
	defining equation is conjugate in the mapping class group by an 
	element of order $2$ to its 
	inverse $T^{-1}$. This is a statement  in the mapping class group of the 
	Milnor 
	fiber and
	does not refer to  any complex conjugation, so it can be stated for
	any complex hypersurface singularity with complex defining equation. We say 
	that the singularity  is {\em strongly invertible} if its geometric 
	monodromy
	diffeomorphism $T$ is conjugate by an element of order $2$  in the relative 
	mapping class group of the Milnor fiber  to its
	inverse $T^{-1}$.  Thus, the property of strong invertibility is a 
	topological 
	property for hypersurface singularities.
\end{remark}
We can rewrite the symmetry property as follows:
$T_Y \circ c_F \circ T_Y \circ c_F= {\rm Id}_F$. We see that 
$T_Y \circ c_F:F \to F$ is an involution of $F$. It follows the

\begin{corollary} The geometric monodromy $T$ of an 
	isolated  complex hypersurface singularity, which is 
	defined by a real equation, is the composition of two involutions of the 
	fiber $F_\delta, \delta \in \mathbb{R},$ namely: 
	$T= (T \circ c_F) \circ c_F$, where $c_F$ is the restriction of the complex 
	conjugation.
\end{corollary}

For $k \in \mathbb{Z}$ we also have the
relation
$T_Y^k \circ c_F \circ T_Y^k \circ c_F= {\rm Id}_F$,
which shows that 
$T_Y^k \circ c_F:F \to F, k \in \mathbb{Z},$ is a sequence of involutions of 
$F$.

The above observations can be applied to plane curve singularities 
in general, since
it follows from the Theory of Puiseux Pairs that
every plane curve singularity is topologically equivalent to a singularity
given by a real equation and hence plane curve singularities are 
strongly invertible.

For complex hyper\-surface singu\-larities of higher dimension the situation
seems to be  op\-posite.
In $\mathbb{C}^{n+1},\,n>1,$ there exist iso\-lated hyper\-surface 
singula\-rities  
which are
not topologically equivalent to a 
singularity with a real defining equation \cite{Teissier}.
We expect that in general the geometric monodromy $T$ 
of a complex hypersurface singularity fails to be strongly invertible.  

Mathias Schulze  communicated to the first author a proof of strong 
invertibility of  the  homological monodromy with real coefficients (this was 
never published and stayed as a personal communication).
The following is a strengthening of his
result to the case of homology with rational coefficients.

\begin{theorem}\label{thm:hom_strong}
	The rational homological monodromy of a complex hypersurface singularity 
	is strongly invertible. 
\end{theorem}

\begin{proof}
	Let
	$E=\bigoplus_i E_i$ be a finest possible direct sum 
	decomposition in $T_*$-invariant 
	$\mathbb{Q}$-subspaces of $E$. The characteristic polynomial of the 
	restriction
	$A_i:E_i \to E_i$ 
	of $T_*$ to a summand $E_i$ is a power $\psi(t)=\phi(t)^L$ of a cyclotomic 
	polynomial $\phi(t)$. We have $\psi(A_i)=0$ by the Hamilton-Cayley 
	theorem and we have $\psi(A_i^{-1})=0$ since a power of cyclotomic 
	polynomial
	satisfies $\psi(1/t)=\pm t^{{\rm degree}(\psi)} \psi(t)$. 
	
	Since the decomposition has no refinement, we may
	choose a  vector $e_1 \in E_i$,  that is 
	cyclic for $A_i$ and for $A_i^{-1}$. 
	The systems
	$$
	e_1,e_2:=A_i(e_1), \cdots ,e_{\dim E_i}:=A^{\dim E_i -1}(e_1),
	$$
	$$
	f_1:=e_1,f_2:=A_i^{-1}(e_1), \cdots ,f_{\dim E_i}:=A^{-\dim E_i +1}(e_1)
	$$ 
	are basis
	for the space $E_i$. Let $b_i:E_i \to E_i$ be the linear map defined by
	$b_i(e_j)=f_j, 1 \leq j \leq \dim E_i$. 
	
	We have
	$A_i^{-1}b_i(e_j)=b_iA_i(e_j), 1 \leq j < \dim E_i$. The polynomial 
	$\psi_1(t):=-\psi(t)+t^{\dim E_i}$ satisfies 
	$\psi_1(A_i^{\pm 1})=A_i^{\pm \dim E_i}$. Since the degree of 
	the polynomial
	$\psi_1(t)$  is less than $\dim E_i$, we have
	$$
	A_i^{-1}b_i(e_{\dim E_i})=A^{-\dim E_i}b_i(e_1)=\psi_1(A_i^{-1})b_i(e_1)=
	$$
	$$
	b_i\psi_1(A_i)(e_1)=b_iA_i^{\dim E_i}(e_1)=b_iA_i(e_{\dim E_i}).
	$$
	We conclude that $A_i^{-1}b_i=b_iA_i$ and $A_i^{-1}=b_iA_ib_i^{-1}$ hold.
	
	The polynomial
	$\psi_0(t):={ {\psi(t)-\psi(0)} \over {-t\psi(0)} }$ satisfies
	$A_i^{-1}=\psi_0(A_i)$ and $A_i=\psi_0(A_i^{-1})$. 
	We deduce $b_iA_i^{-1}=b_i\psi_0(A_i)=\psi_0(A_i^{-1})b_i=A_ib_i$ and 
	conclude $A_i^{-1}=b_i^{-1}A_ib_i$. We observe at this point that both 
	the conjugates of $A_i$ by $b_i$ and by
	$b_i^{-1}$ are equal to the inverse $A_i^{-1}$.
	
	For $0 \leq j < \dim E_i$ we have (remember $e_1=f_1=b_i(e_1)=b_i(f_1)$)
	$$
	b_i(f_{j+1})=b_iA_i^{-j}(f_1)=A_i{^j}b_i(f_1)=A_i{^j}(e_1)=e_{j+1}.
	$$
	Hence
	$b_i$ is of order two, which shows that $A_i$ is strongly invertible over
	$\mathbb{Q}$. The direct sum $b:=\bigoplus b_i$ is a rational strong 
	inversion  
	for 
	$T_*$. 
\end{proof}

{\bf Question:} Is the integral homological monodromy of a complex hypersurface 
singularity
strongly invertible? For example, take the isolated surface 
singularity
\[f(x,y,z)=L_1 L_2^2 L_3^3 L_4^4 L_5^5 L_6^6 L_7^7 L_8^8 L_9^9 
+x^{46}+y^{46}+z^{46}\]
where $L_j, 1 \leq j \leq 9$ are linear forms on $\mathbb{C}^3$, such that the 
$9$ 
lines $\{L_j=0\}$ in the complex 
projective plane
span the $9_3$ configuration of flex tangents to a nonsingular cubic. No real 
equation for this singularity
can exist, since the configuration $9_3$ cannot be realized in the real 
projective plane, see \cite{Teissier,Hil_Cohn}. Is its monodromy strongly 
invertible?

\subsubsection*{Going back to plane curves and divides}

Now we focus on the study  of complex
conjugation on the topology of plane curve singularities via divides. First we 
show that the fiber of the link of a
connected divide carries naturally a cellular decomposition with tri-valent
$1$-skeleton.

Let $f:\mathbb{C}^2 \to \mathbb{C}$ define an isolated plane 
curve singularity  at $0 \in \mathbb{C}^2$ given by a real convergent power 
series  
$f\in \mathbb{R}\{x,y\}$. Let $f=f_1f_2 \cdots f_r$ be the factorization in
local branches.  The topology of a plane curve singularity $f$ is completely 
encoded in 
a divide for $f$ as we have seen in Sect.~\ref{s:fibration_thm} or in the 
previous 
section where we have given a factorization of the monodromy from the divide. 
(c.f. 
\cite{Acampo_real_def,Balk_Kaend}), see also 
\cite{Acampo_groupe_I,Acampo_groupe_II,GZ_inter_two}. 

We state the results  more generally for links of 
connected divides, since the involution given by complex conjugation 
on real isolated plane curve singularities 
corresponds to a natural involution of
links of divides as explained below, (c.f.
\cite{Acampo_generic_imm,Acampo_real_def}).

Denote by  $\delta_P$ the number of double 
points of the divide $P$.
The local topology of a plane curve singularity is obtained from
a divide for the singularity.  More precisely, 
the link $L_f \subset \partial{B_{\epsilon}}$ of an isolated 
plane curve singularity
$\{f=0\}, f \in \mathbb{R}\{x,y\},$
is equivalent to the link  $L_P$ of a divide $P$, see 
Sect.~\ref{s:divides_plane_curves} (c.f. \cite{Acampo_real_def}).  

We recall, 
that for a totally real (Definition~\ref{def:totally_real}) plane curve 
singularity $f$ 
one can obtain
a divide $P$ by performing 
a small real deformation $f_s, 0 \leq s \leq 1,$ of the 
singularity, where for $0 < s \leq 1$ the restriction of $f_s$ to 
the euclidean disk $D_{\epsilon}:=B_{\epsilon} \cap \mathbb{R}^2$ 
of radius $\epsilon$ in $\mathbb{R}^2$ is a Morse function with 
$\mu(f)$ critical points and such
that the $0$-level is connected and contains all the saddle points (this is the 
content of Theorem~\ref{thm:existence_divides}). 
The divide $P$ for $f$ is the curve $\{ p \in D_{\epsilon} \mid f_1(p)=0 \}$, 
which we rescale by the factor ${1 \over \epsilon}$ from 
$D_{\epsilon}$ into the 
unit disk $D$.

\begin{figure}
	\centering \includegraphics*[scale=0.4]{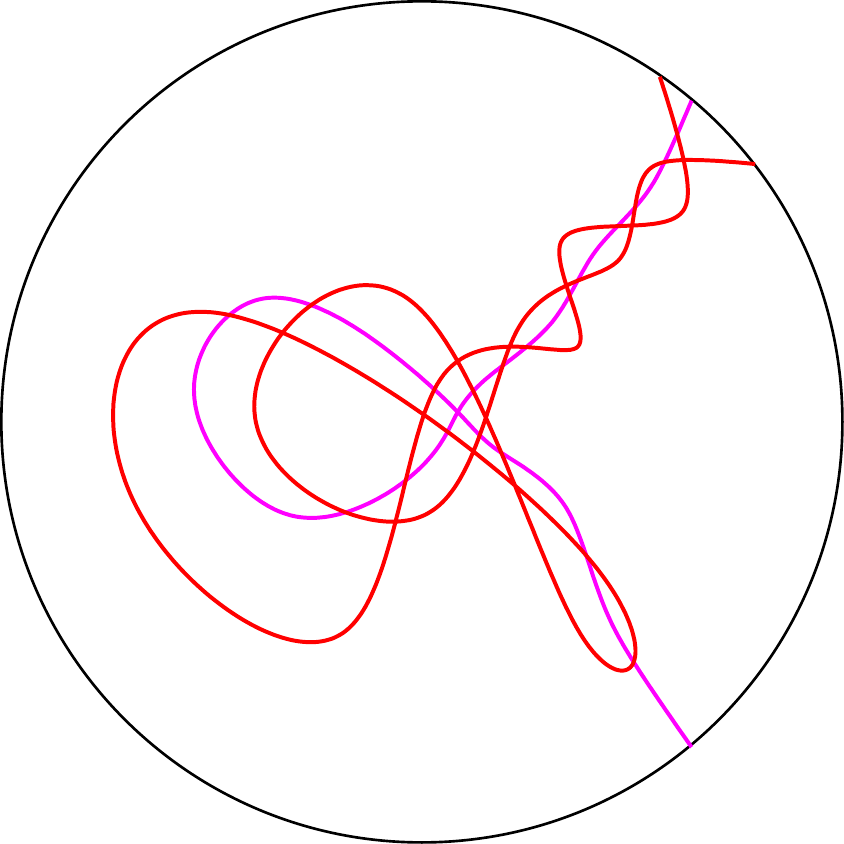}
	\caption{The divide corresponding to the singularity $(x^3 - y^2)(( x^3 - 
		y^2 )^2 - 
		4x^8y)$.}
\end{figure}

The homology $H_1(F_P,\mathbb{Z})$ can be described combinatorially in terms of 
the 
divide $P$ as a direct sum 
$H_1(F_P,\mathbb{Z})=E_- \bigoplus E_0 \bigoplus  E_+$, where
$E_-,E_0$ and $E_+$ are the subspaces in 
$H_1(F,\mathbb{Z})$, which are freely generated  
as follows: 
\[
\begin{split}
	&E_+ :=[\delta_1, \cdots ,\delta_{\mu_+}],
	\\
	&E_0 \,:=[\delta_{\mu_++1}, \cdots ,\delta_{\mu_+ + \mu_{0} }], \quad 
	\text{ 
		and}
	\\
	&E_- :=[\delta_{\mu_++\mu_{0}+1}, \cdots , \delta_{\mu_-+\mu_{0}+\mu_+}],
\end{split}
\]
where $(\delta_i)_{ 1 \leq i \leq \mu}$ is 
the oriented system of vanishing cycles of the 
divide $P$ with positive upper triangular Seifert form 
\[
S:H_1(F,\mathbb{Z}) \to H^1(F,\mathbb{Z})={\rm 
	Hom}(H_1(F,\mathbb{Z}),\mathbb{Z}),
\] 
see. We define $N:=S - {\rm Id}$, 
which is 
upper-triangular nilpotent matrix. The monodromy
$T_*:H_1(F,\mathbb{Z}) \to H_1(F,\mathbb{Z})$ is given by:
\[
T_*=(S^t)^{-1} \circ S.
\]

For the rest of this section, recall the construction of the 
fiber 
surface  $F_P$ of a divide $P$  in the unit euclidian 
disk $D$. In particular, recall the notation from Sect.~\ref{s:definitions} and 
Sect.~\ref{s:divides_plane_curves}.

\begin{figure}
	\centering \includegraphics*[scale=0.3]{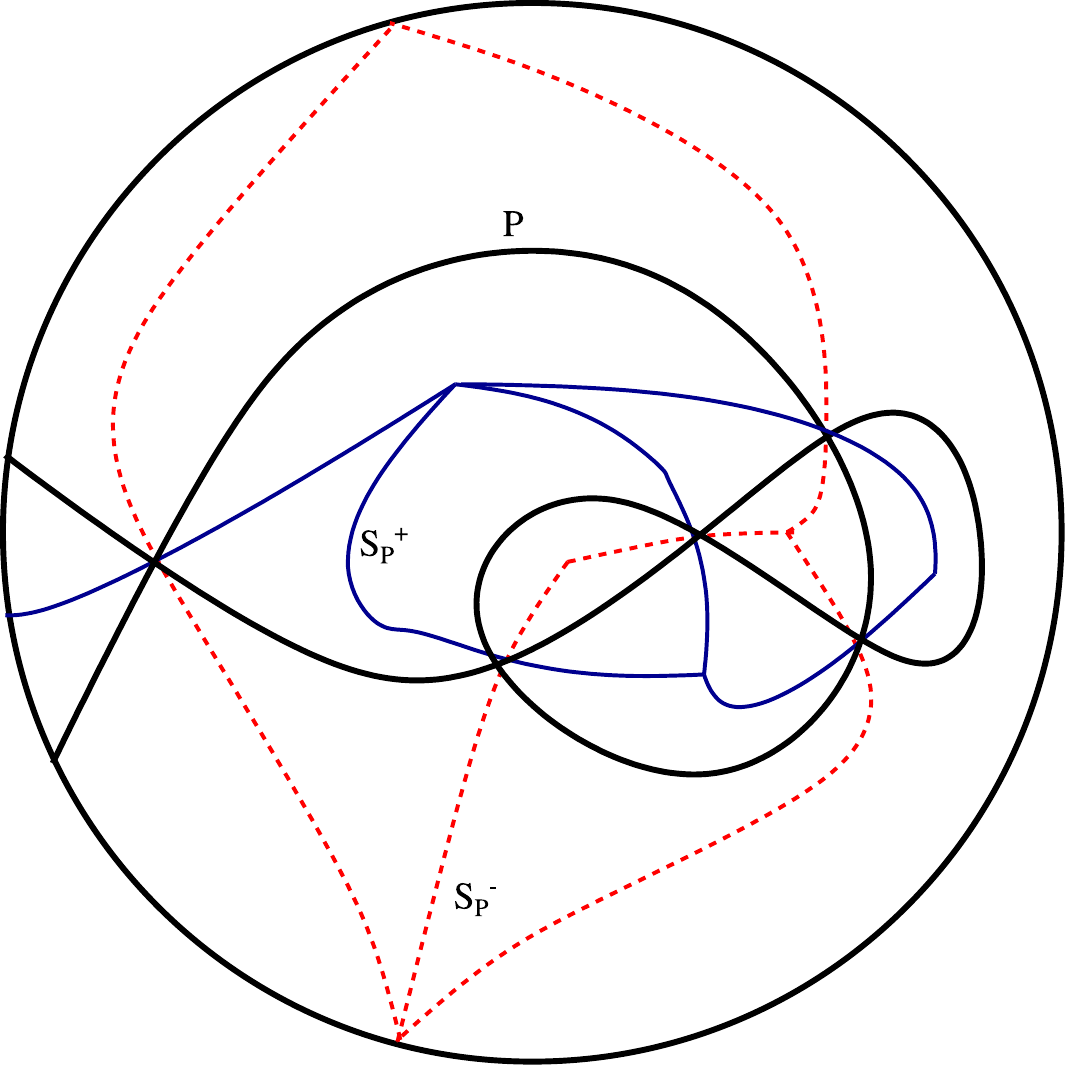}
	\caption{ Divide and graph of divide.}
\end{figure}

Let $S_P$ be a system
of gradient lines of $f$, which connect saddle points of $f$ with  
local or relative maxima or minima of $f$. In fact, near the boundary of 
$D$ the lines of the system $S_P$ are only gradient-like, but end in a 
relative critical point of $f$. We call the  
set $S_P$ the graph of the divide $P$. By a slalom 
construction, see \cite{Acampo_slalom} (not included in this work), one can 
reconstruct
from each of the sets $S_P^{\pm}:=\{x \in S_P \mid \pm f(x) \geq 0 \}$ 
the divide $P$. We assume that at the 
maxima and minima of $f$, that different
arcs of $S_P$ have different unoriented tangent directions. Let $\Sigma_P$
be the following subset of the fiber $F_P$ over $1$ of the natural fibration
over $S^1$ of the complement of the link of $P$: the subset $\Sigma_P$ is
the closure in the tangent space $TD$ of $D$ of
$$
\Sigma'_P:=\{(x,u) \in TD \mid x \in S_P, f(x)> 0,
(df)_x(u)=0, \|x\|^2+\|u\|^2=1 \}.
$$

\noindent
We recall that the union of fiber $F_P$ with the link $L_P$
is the closure in $TD$ of
$$
F'_P:=\{(x,u) \in TD \mid x \in D, f(x) > 0,
(df)_x(u)=0, \|x\|^2+\|u\|^2=1 \},
$$
\noindent
which is a surface of genus $g(P)$ equal to the number of double points in $P$.

\begin{figure}
	\centering \includegraphics*[scale=0.6]{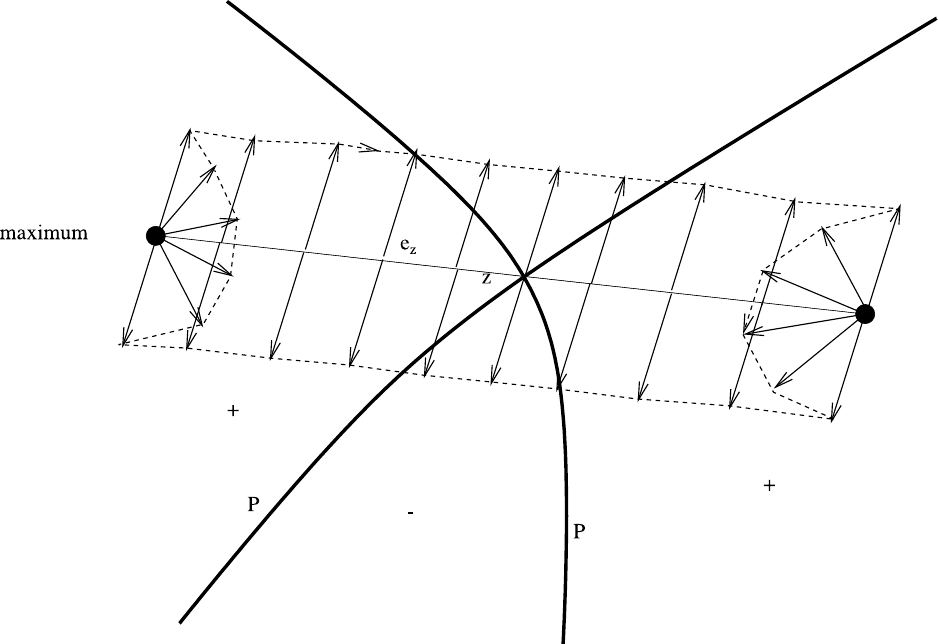}
	\caption{ Vanishing cycle $\delta_z$ for a saddle point.}
	\label{fig:vanishing_saddle}
\end{figure}

The system of cycles $\delta_i, 1 \leq i \leq \mu,$ can be drawn on 
$\Sigma_P$ as follows. The vanishing cycle 
$\delta_i, \mu_-+\mu_0 < i \leq \mu,$
which corresponds to a local maximum $M$ of $f$ is the circle
$\{ (M,u) \mid u \in T_MD, \|M\|^2+\|u\|^2=1\} $ oriented counter clock-wise.
The vanishing cycle $\delta_i, \mu_- < i \leq \mu_-+\mu_0,$ which 
corresponds to a saddle point $z$ of $f$ is a curve in the set
\[E_z:=\{(x,u) \in TD \mid x \in e_z,
u\in {\rm Ker}(df)_{e_z}, \|x\|^2+\|u\|^2=1 \},\] as drawn in 
Fig.~\ref{fig:vanishing_saddle}. 
The curve $\delta_i$ is a piecewise smooth embedded  
copy of $S^1$ in $F_P$ with image in $E_z$
and with non-constant projection to the disk $D$. The orientation 
is chosen such that at both ends the orientation agrees with the oriented
vanishing cycle of the maximum. Moreover, the inward tangent vectors
$(x,u)$ to $e_z$ at end points $x \in e_z$ do  belong to $\delta_i$. 
The vanishing cycle $\delta_i, 1 \leq  i \leq \mu_-,$ which 
corresponds to a local minimum of $f$ projects to an oriented circuit
$e_1,e_2, \cdots ,e_k$ of edges of the graph $S_P$. The circuit 
surrounds the region in counter clock-sense 
to which 
the minimum corresponds. The circuit is a polygon and bounds a cell in $D$.
The curve $\delta_i$ is a subset of
$\{(x,u) \in TD \mid x \in \cup_{1 \leq j \leq k}  e_j, u\in {\rm 
	Ker}(df)_{x}, \|x\|^2+\|u\|^2=1 \}$
and is the image of a piecewise smooth  embedding of 
$S^1$ as drawn in Fig.~\ref{fig:vanishing_min}. The vectors
$(x,u)$, which belong to $\delta_i$, point out of the cell of the circuit. 
Smooth representatives for the system of \index{geometric vanishing cycle} 
vanishing cycles can be obtained 
using {\em tears} as in \cite{Acampo_quadratic} (see also 
Sect.~\ref{ss:visual_vanishing}).

\begin{figure}
	\centering \includegraphics*[scale=0.7]{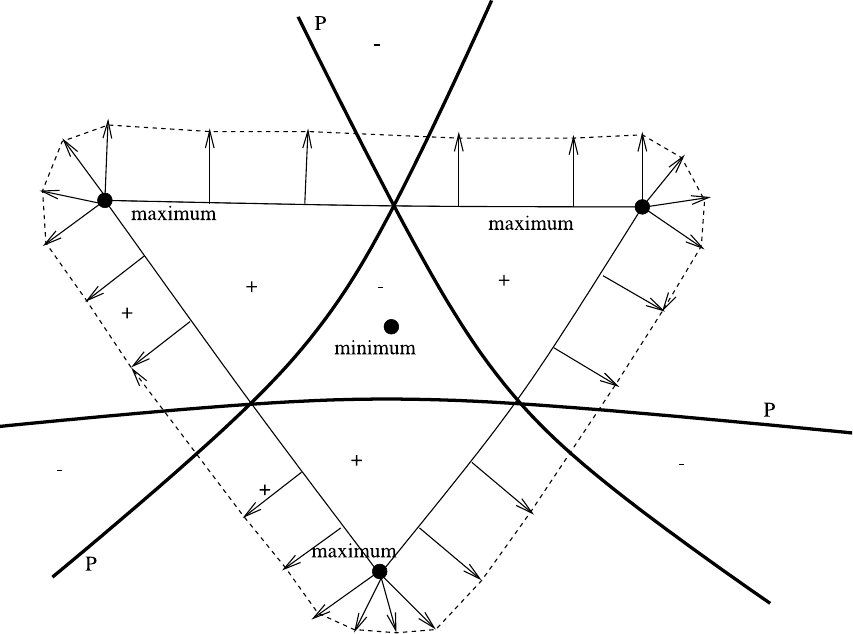}
	\caption{Vanishing cycle $\delta_i$ for a minimum.}
	\label{fig:vanishing_min}
\end{figure}

The involution
$c_{T\mathbb{R}^2}:(x,u)\in T\mathbb{R}^2 \mapsto (x,-u)\in T\mathbb{R}^2$
induces an involution
$c_{S^3}$ on $S^3\subset T\mathbb{R}^2$ that preserves the link $L_P$
of any divide $P$ and that
induces involutions on the fiber surfaces above $\pm 1$.
If the divide $P$ is a divide for a real plane curve singularity $f$, the
involutions induced by $c_{T\mathbb{R}^2}$ or by the complex conjugation $c$ on
the triples $(S^3,L_P,F_P)$ and on $(\partial{B_{\epsilon}},L_f,F)$ 
correspond to
each other by the homeomorphism of pairs of the main theorem of (c.f. 
\cite{Acampo_real_def}).

\begin{theorem} Let $P \subset D$ be a connected divide consisting of the 
	image of a relative generic immersion of a compact one dimensional 
	manifold   
	in the unit disk 
	$D \subset \mathbb{R}^2$. Let $E_-,E_0$ and $E_+$ be the summands in 
	$H_1(F,\mathbb{Z})$ as above. 
	The involution $c_{F*}:H_1(F,\mathbb{Z}) \to H_1(F,\mathbb{Z})$ fixes 
	point-wise
	the summand $E_+$, in particular for $\delta_i \in E_+$ we have
	$$
	c_{F*}(\delta_i)=\delta_i.
	$$
	For $\delta_i \in E_0$ we have
	$$
	c_{F*}(\delta_i)= -\delta_i+
	\sum_{1\leq j \leq \mu_+} \langle N(\delta_j),\delta_i \rangle.
	$$
	For $\delta_i \in E_-$ we have
	$$
	c_{F*}(\delta_i)= \delta_i+
	\sum_{1\leq j \leq \mu_+} \langle N(\delta_j),\delta_i \rangle-
	\sum_{\mu_+ < j \leq \mu_++\mu_0} 
	\langle N(\delta_j), \delta_i \rangle \delta_j.
	$$
	The trace  
	of the involution $c_{F*}:H_1(F,\mathbb{Z}) \to H_1(F,\mathbb{Z})$
	is given by: 
	$$
	{\rm Trace}(c_{F*})=\mu_- - \mu_0 + \mu_+.
	$$
\end{theorem}  

\begin{proof} A vanishing cycle $\delta_i \in E_+$ corresponds 
	to a maximum $M \in D$ of $f$,
	hence as set we have $\delta_i=\{(M,u) \in TD \mid \|M\|^2+\|u\|^2=1\}$.
	The
	involution $(x,u) \mapsto (x,-u)$ induces on $\delta_i$ the antipodal map,
	which is orientation preserving. It follows $c_{F*}(\delta_i)=\delta_i$ in
	homology.  
	A vanishing cycle $\delta_i \in E_0$ corresponds to a saddle point 
	$z \in D$ of $f$. Working with the tear model  of (c.f. 
	\cite{Acampo_quadratic}), see also (the independent) 
	Sect.~\ref{ss:visual_vanishing},
	we see that the involution reverses the orientation and that at
	the endpoints, which are maxima of $f$, we have outward instead of 
	inward vectors. Hence, 
	$$
	c_{F*}(\delta_i)=-\delta_i+\sum m_{i,j} \delta_j,
	$$ 
	where in the sum
	$j$ runs through the maxima of $f$ in the interior of $D$ which are 
	connected by gradient 
	lines of $S_P$ to the saddle point $i$. The coefficient $m_{i,j}$ equals 
	$1$ or $2$ 
	depending on whether the connection by gradient lines is simple or double.
	Finally one gets 
	$c_{F*}(\delta_i)= -\delta_i+
	\sum_{1\leq j \leq \mu_+} \langle N(\delta_j),\delta_i \rangle$. 
	For a vanishing cycle $\delta_i \in E_-$
	one can work with the model of (c.f. \cite{Acampo_quadratic}) and get 
	$c_{F*}(\delta_i)= \delta_i+
	\sum_{1\leq j \leq \mu_+} \langle N(\delta_j),\delta_i \rangle-
	\sum_{\mu_+ < j \leq \mu_+ \mu_0} \langle N(\delta_j),\delta_i \rangle$. 
	Since $N$ is strictly upper-triangular, the matrix of $c_{F*}$ is 
	upper-triangular with $\pm 1$ on the diagonal. We get 
	${\rm Trace}(c_{F*})=\mu_- - \mu_0 + \mu_+$.
\end{proof}

We can present the Seifert form and homological 
monodromy with block-matrices $S,T$ as in \cite{GZ_Index}. On 
$H_1(F,\mathbb{Z}),E_-,E_0,E_+$ 
and 
$H^1(F,\mathbb{Z})$ we work with the basis or dual basis given by the system 
$\delta_i, 1 \leq i \leq \mu$. 
The matrix in block form of the Seifert form is
$$
S=\left(
\begin{matrix}{\rm Id}_{\mu_+}& A & G\cr O& {\rm Id}_{\mu_0}& B\cr
	O & O & {\rm Id}_{\mu_-}\end{matrix}
\right)  
$$
where the block $G$ equals the block matrix product $1/2 (A \circ B)$. 
The matrix coefficients of $A \circ B$ and $G$ have interpretations 
in terms of the divide $P$ or in terms of the Morse function 
$f_P$ on the disk $D$. 
The matrix coefficient 
$(A \circ B)_{(i,j)}, 1 \leq i \leq \mu_+, \mu_+ + \mu_0 < j \leq \mu,$ 
counts the number of  sector adjacencies, that has the 
$+$-region $i$ of $P$  with the $-$-region $j$, while the 
coefficient $G_{(i,j)}$ counts the number of 
common boundary 
segments of the regions $i$ and $j$. The coefficient $(A \circ B)_{(i,j)}$
also counts the number of 
saddle connections via gradient lines of $f_P$ in between
the minimum $j$ and maximum $i$, while the coefficient $G_{(i,j)}$
counts the number of components of regular gradient line connections from
the minimum $j$ to the maximum $i$.
This explains the above factor $1/2$, since a segment of 
$P$ is twice incident 
with a saddle point of $f_P$. 

The matrix of the action of complex conjugation on $H_1(F,\mathbb{Z})$ is
$$
C=\left(\begin{matrix}{\rm Id}_{\mu_+}& A & G \cr 
	O& {\rm -Id}_{\mu_0}& -B\cr
	O & O & {\rm Id}_{\mu_-}\end{matrix}\right)  
$$
and is obtained from the matrix $S$ by multiplying the middle row of blocks by 
$-1$.
It is interesting to compute the matrix of the involution $T\circ c_F$ 
on $H_1(F,\mathbb{Z})$
$$
TC=\left(\begin{matrix}{\rm Id}_{\mu_+}& O & O\cr 
	-{^tA}& {\rm -Id}_{\mu_0}& O \cr
	{^tG} & {^tB} & {\rm Id}_{\mu_-}\end{matrix}\right)  
$$
The matrix $T\circ C$ is the transgradient of the matrix of $C$. 

It turns out that 
the combinatorial property $G=1/2 (A \circ B)$ for divides 
is equivalent to 
$C\circ C={\rm Id}_{\mu}$ or $T\circ C\circ T\circ C={\rm Id}_{\mu}$.

For an isolated  plane curve singularity at $0 \in \mathbb{C}^2$, which is 
given by a real 
equation $\{f=0\}, f \in \mathbb{R}\{x,y\}, f(0)=0,$ we will 
denote by $\delta_{\mathbb{R}}(f)$ the number of 
double points of a divide $P$ 
for the singularity.  Hence $\delta_{\mathbb{R}}(f)$ is 
the maximal number of local
real 
saddle points in some level, that can occur for  
a small real deformation of $f$.
Observe, that one has 
$\delta_{\mathbb{R}}(f) \leq \delta(f)$, where $\delta(f)$ 
is the maximal number of local
critical points in some level, that can occur for  
a small deformation of $f$. We recall the formula
$\mu(f)=2 \delta(f)-r+1$ of Milnor \cite{Milnor}. As example for 
$f=x^4+Kx^2y^2+y^4, -2 \not= K \not= 2,$ one has
$\delta_{\mathbb{R}}(f)=4, \delta(f)=6, r=4$ and 
very surprisingly, Callahan shows 
that
for $-2 < K < 2$ there exists a small real deformation with
$5$ local minima in the same level \cite{Calla}.

\begin{theorem} For an isolated  plane curve singularity at 
	$0 \in \mathbb{C}^2$, which is  given by a real 
	equation $\{f=0\}, f \in \mathbb{R}\{x,y\}, f(0)=0,$ we have
	$$
	\mu(f)=2 \delta_{\mathbb{R}}(f)+{\rm Trace}(c_{F*})
	$$
\end{theorem}

\begin{proof} 
	We have $\delta_{\mathbb{R}}={\rm dim}E_0=\mu_0.$
	Hence 
	$\mu(f)=\mu_-+\mu_0+\mu_+=
	{\rm Trace}(c_{F*})+2\mu_0=2 \delta_{\mathbb{R}}(f)+{\rm Trace}(c_{F*}).$
\end{proof}

Curves $\delta_i$ that correspond to maxima of $f_P$ are invariant by the 
involution $c$. A curve $\delta_i$ 
that corresponds to a saddle point or minimum of $f_P$ is in general not 
invariant by $c$. The union
$\Sigma_P$ of the curves $\delta_i$ is invariant by $c$. We get

\begin{theorem} Let $P$ be a connected divide. The pair $(F_P,\Sigma_P)$ 
	defines  a tri-valent cellular decomposition with $r(P)$ punctured 
	cells  
	of the fiber surface with boundary $F_P$. The involution 
	$c:(x,u) \to (x,-u)$ 
	acts on the map $(F_P,\Sigma_P)$.
\end{theorem}

\begin{proof}
	The union $\Sigma_P$ of the vanishing cycles $\delta_i$
	is a tri-valent graph $\Sigma_P$ that is invariant by $c$. The inclusion 
	$\Sigma_P \subset F_P$ induces an isomorphism 
	$H_1(\Sigma_P,\mathbb{Z}) \to H_1(F_P,\mathbb{Z})$, hence, $\Sigma_P$ is a 
	spine for the 
	surface $F_P$.
\end{proof}

{\bf Question}. In particular if $r(P)=1$ the triple $(F_P,\Sigma_P,c)$ is a 
so-called
{\em maximal unicell map} of genus $g(P)$ with orientation reversing 
involution $c$ with $r(P)=1$ fixed points on the graph of the map. 
The involution $c_F$ has a unique fixed point on $\Sigma_P$, which 
corresponds to the intersection of the 
folding curve $F_P \cap \partial{D}$ of $c_F$ with the graph $\Sigma_P$. 
Maximal here
means that the number of edges of the graph of the map is maximal, i.e.
$3g$.

It would be interesting to compute the generating series  
$$
{\rm divide}(t):=\sum_{g\in \mathbb{N}}d(g)t^g=1+t+2t^2+8t^3+36t^4+ \cdots,
$$ 
where $d(g)$ denotes the 
number of simple, relative and  generic immersions of the  interval $[0,1]$ 
with $g$ double points in the  disk $D$, counted up to homeomorphism in the
source and image. We have taken the coefficients up to the term $t^4$ of 
${\rm divide}(t)$  from the listings of
simple, relative, free or oriented divides by Masaharu Ishikawa \cite{Ishi_phd}.
We ask to compare the numbers $d(g)$ with the numbers $m^+(g)$ of maximal 
maps of genus $g$ with orientation reversing involution having $1$ fixed point
on its graph.

\subsubsection*{Involutions induced by $\pi$-rotations}

Next we study the involutions that appear in the decomposition
of the geometric monodromy of plane curve singularities. We show that they lift 
to $\pi$-rotations
about an axis in the universal cyclic covering of the complement of the link.

Our next goal is to visualize the two involutions $c_F$ and $T \circ c_F$. 
We assume that the monodromy diffeomorphism $T$ is chosen 
in its isotopy class such that $T \circ c_F$
is an involution of the fiber surface $F_P$. In fact, we assume that
the monodromy $T$ was given by a monodromy vector field $X$, which satisfies
$X=X^c$.

\begin{figure}
	\centering \includegraphics*[scale=0.56]{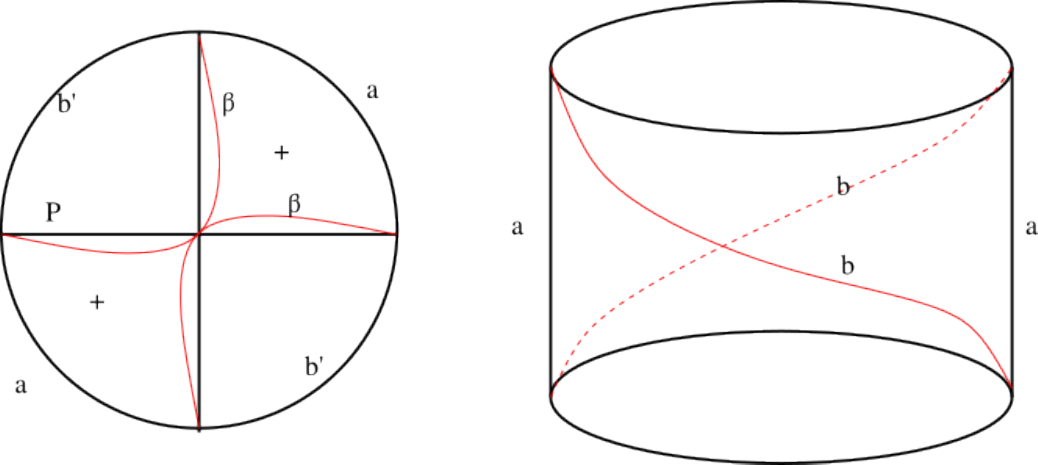}
	\caption{  Dehn twist as composition of two involutions.}
	\label{fig:dehn_inv}
\end{figure}

We assume for simplicity, that the divide $P$ meets the boundary of the disk
$D$.
The involution $c_F$ is an orientation reversing diffeomorphism of $F_1=F_P$,
which fixes point-wise the system of arcs 
$a:=F_P \cap \partial{D}$. The Lefschetz number
of the orientation reversing involution 
$T \circ c_F$ is equal to the Lefschetz number of $c_F$ by Th. $1$. 
Hence the 
involution $T \circ c_F$ also
fixes  point-wise exactly a system of  arcs $b$ on the fiber surface $F_P$ 
with $\chi(a)=\chi(b)$, i.e.
the systems $a$ and $b$ consist of the same number of arcs. 

Let 
$z:Z \to S^3 \setminus K_P$ be the infinite
cyclic covering  of the knot complement. Let $X^Z$ be the lift of the vector
field $X$. Let $T_{{1 \over 2}}:Z \to Z$ be the flow diffeomorphism of 
the vector field $X^Z$ with stopping time ${1 \over 2}$. 
Let $F_1' \subset Z$ be a lift of the fiber surface
$F_1$ and let $F_{-1}'$ be the lift $T_{{1 \over 2}}(F_1')$ 
of the fiber surface $F_{-1}$. The involution $c_F$ can first 
be lifted unambiguously to an 
involution of $F_1'$ and then extended unambiguously 
to an involution
$A$ of $Z$ such that $A$ maps the vector field $X^Z$ to its opposite. 
The involution $A$ is a $\pi$-rotation about a lift of the system of arcs 
$a$ into $F_1'$.
The involution $(x,u) \mapsto (x,-u)$ induces an involution $c_{F_{-1}}$ 
of the fiber surface $F_{-1}$ above $-1$, and,
as we have done for the involution $c_F$,  
the involution $c_{F_{-1}}$ lifts unambiguously to a 
$\pi$-rotation $B:Z \to Z$ about a system of  arcs  
$b'' \subset F_{-1}'$ that
satisfies
$z(b'')=b':=\partial{D} \cap F_{-1}$. 
We have $B \circ A=T_{{1 \over 2}}\circ T_{{1 \over 2}}=T_1$. Since 
$T_1$ is a lift of the monodromy, it follows from $(T \circ c_F) \circ c_F=T$, 
that
the involution $T \circ c_F$ fixes the system of arcs 
$b:=z(T_{{1 \over 2}}^{-1}(b''))$.

\begin{figure}
	\centering \includegraphics*[scale=0.5]{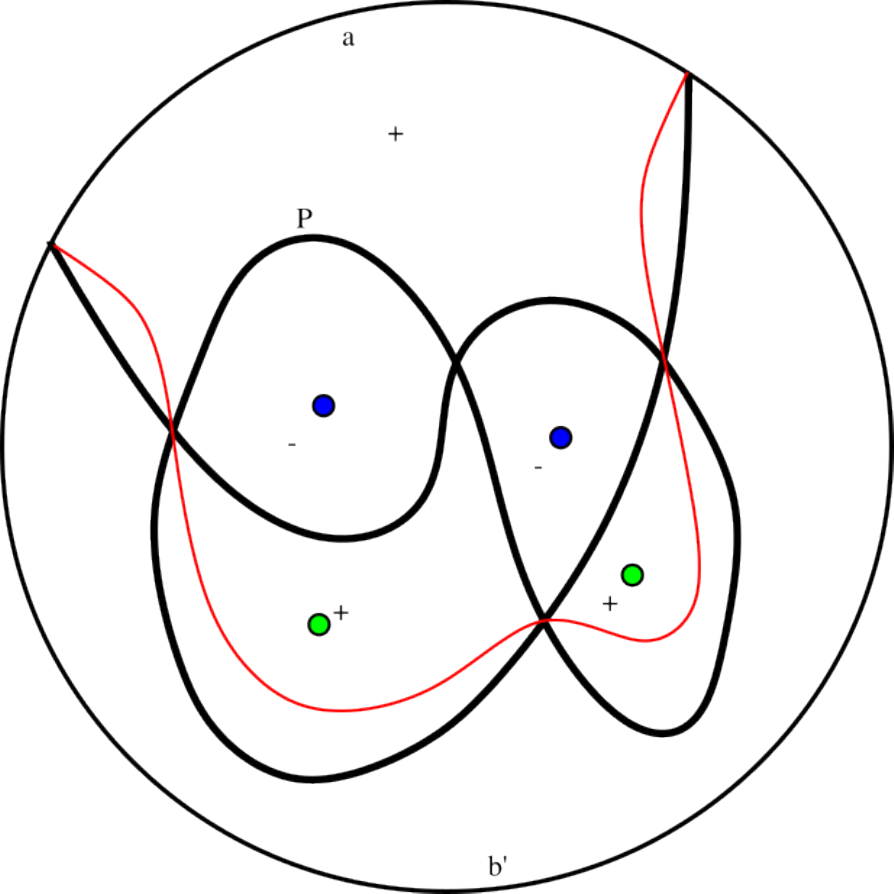}
	\caption{ Divide for the $E_8$ singularity and the curve $\beta$ in red.}
	\label{fig:E_8_beta}
\end{figure}

In Sect.~\ref{ss:monodromy_product} (c.f. \cite{Acampo_generic_imm}) we have 
computed the monodromy
diffeomorphism $T_+:F_1 \to F_{-1}$ for which $T_{{1 \over 2}}$ is a lift to
$Z$ as a product of half Dehn twists for a divide. 
From the above we deduce, that the
involution $T \circ c_F:F_1 \to F_1$ is the composition 
$T_+^{-1} \circ c_{F_{-1}} \circ T_+$. We see that $T \circ c_F$ fixes the 
arc $b=T_+^{-1}(\partial{D} \cap F_{-1})$. We also see that the involution
$c_F \circ T$ of $F_1$ fixes the system of arcs 
$c_F(b)=T_+(\partial{D} \cap F_{-1})$.
Both arc systems $b$ and $c_F(b)$ have equal  projections $\beta$ into $D$.
See  Figures~\ref{fig:dehn_inv} and \ref{fig:E_8_beta} for examples. 

The composition of the orientation reversing
involutions, that fix point-wise the arc systems $a,b$ of 
Fig.~\ref{fig:dehn_inv} on 
the 
cylinder surface,
is indeed a Dehn twist. To see this, think of  the cylinder as 
$[0,\pi] \times S^1$; the two involutions induce on the circle 
$\{\theta\} \times S^1$ reflections about diameters that make an 
angle $\theta$, so the
composition is a rotation of angle $2\theta$ of the circle 
$\{\theta\} \times S^1$.

In Fig.~\ref{fig:E_8_beta} the curve $\beta$ is drawn for a more complicated 
divide 
for the
plane curve singularity $E_8$ with equation $y^3-x^5=0$.

We conclude with a remark we promised that shows that uses the results of this 
section to prove that not all fibered links come from divides.
\begin{remark}\label{rem:eight_knot_divide}
	The figure eight knot (see Fig.~\ref{fig:figure_eight}) is not the knot of 
	a 
	divide. 
	The 
	figure eight knot's 
	complement fibers over the circle having as fiber the punctured torus and 
	as monodromy the isotopy class of the 
	linear diffeomorphism given by a matrix in $SL(2,\mathbb{Z})$ of trace $3$. 
	Such a matrix $M$ is not the product of two  
	unipotent matrices, which are conjugate in $SL(2,\mathbb{Z})$ and the 
	matrices $M$
	and $M^{-1}$ are not conjugated by an integral matrix of determinant $-1$. 
	So according 
	to the remarks  this section,  the 
	figure eight can not be the knot of a divide. Another argument to 
	rule out the figure eight as the knot of a 
	divide goes as follows. The first betti number of the fiber of the 
	figure eight knot is $2$.  But only two 
	connected divides
	have  fibers with betti number $2$ and these two have monodromies with 
	trace equal to $1$ or $2$.   
\end{remark}
\section{Abstract visualization of the Milnor fiber and the monodromy}
\label{s:abstract}

In this section we describe an easy and fast graphical algorithm of visualizing 
the Milnor fiber
together with a distinguished
basis of vanishing cycles directly from the divide (\cite{Acampo_real_def}).

\subsection{Visualization of the vanishing cycles for a divide}
\label{ss:visual_vanishing}
\index{geometric vanishing cycles!visualization}
We start by visualizing the vanishing cycles.

Let $P$ be a connected divide and let 
$\pi_P:S^3 \setminus L(P) \to S^1$ be the 
fibration of the complement of the link $L(P)$ over $S^1$ as in 
Sect.~\ref{s:fibration_thm}
(c.f. \cite{Acampo_real_def,Acampo_generic_imm}).
Recall that the fibration map is given with the help of an auxiliary Morse 
function
$f_P:D \to \mathbb{R}$.
The fiber $F_P$ above $1 \in S^1$ projects to the positive components of
the complement of $P$ in $D$. One has that the closure of 
$$
\{(x,u)\in TD | f_P(x)>0, (df_P)_x(u)=0, \|x\|^2+\|u\|^2=1\}
$$
in $S^3 \setminus L(P)$ is the fiber surface $F_P$.

To each critical point of $f_P:D \to \mathbb{R}$ corresponds a vanishing cycle 
on
the surface $F_P$. In the case, where the divide $P$ is a divide of a 
singularity, the surface $F_P$ is a model for the Milnor fiber and the
system of
vanishing cycles on $F_P$ is a model for  a distinguished system of quadratic
vanishing cycles of the singularity.

Let $M$ be a maximum of $f_P$. The vanishing cycle $\delta_M$ is the 
non-oriented simply closed curve on $F_P$, see Fig.~\ref{fig:van_max},
$$
\delta_M:=\{u\in T_MD | \|M\|^2+\|u\|^2=1\}. 
$$

\begin{figure}
	\centering \includegraphics*{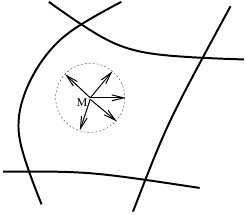}
	\caption{The vanishing cycle $\delta_M$ for a maximum.}
	\label{fig:van_max}
\end{figure}

Let $c$ be a crossing point of $P$. The point $c$ is a saddle type 
singularity of $F_P$. The vanishing cycle is the non-oriented simply closed
curve $\delta_c$ on $F_P$ that results from the following construction.
Let $g_c \subset P_+ \cup \{c\}$ be the singular 
gradient line through $c$, for which
the endpoints are a maximum of $f_P$ or a point in $\partial{D}$. We splice
$g_c$ and get a double {\em tear} $t_c \subset P_+ \cup \{c\}$ as in 
Fig.~\ref{fig:tear_van}. The
tear $t_c$ is a closed curve, that has at $c$ a non-degenerate
tangency with $g_c$ from both sides. Moreover, $t_c$ 
is perpendicular to $g_c$ at the
endpoints of $g_c$, if the endpoint is a maximum of $f_P$ and else $t_c$
has a tangency with $\partial{P}$. The vanishing cycle $\delta_c$
is the closure in $F_P$ of the set 
$$
\{(x,u)\in F_P |  x\in t_c,\,\,(df_P)_x\not=0,\,u\,
\hbox{{\it points to the inside of the tear}}\,\,t_c\}.
$$

\begin{figure}
	\centering \includegraphics*{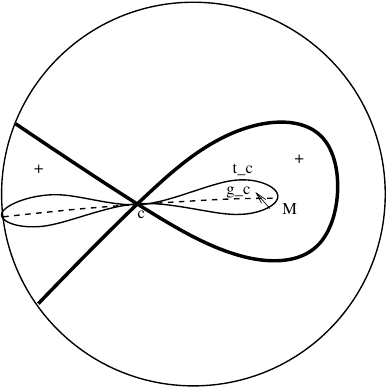}
	\caption{Tear for vanishing cycle $\delta_c$ for a saddle point.}
	\label{fig:tear_van}
\end{figure} 

Let $m$ be a minimum of $f_P$. The following is a description of the
vanishing cycle $\delta_m$ on $F_P$. The projection of $\delta_m$ in 
$D$ is a simply closed curve $t_m$ in $P_+ \cup \{c_1,c_2, \dots ,c_k\}$,
where $\{c_1,c_2, \dots ,c_k\}$ is the list of the double points of $P$
that lie in the closure of the region of $m$, see Fig.~\ref{fig:delta_min}. The 
curve 
$t_m$ and the singular gradient line $g_{c_i},\, 1\leq i\leq k$ coincide in
a neighborhood of $c_i$. Moreover, if the endpoint of $g_{c_i}$ is a 
maximum $M$,
the curve $t_m$ leaves transversally the tear $t_{c_i}$ at $M$ and enters
transversally the next tear $t_{c_{i\pm 1}}$. If
the endpoint of $g_{c_i}$ is on the boundary of $D$, the curve leaves the tear
$t_{c_i}$, becomes tangent to the boundary of $D$ and 
enters the next tear, see Fig.~\ref{fig:delta_min}. The vanishing cycle 
$\delta_m$ 
is the 
non-oriented simply closed curve on $F_P$, that is the closure of
$$\{(x,u)\in F_P | x \in t_m,\,u\,\hbox{{\it points inwards to the
		disk bounded by }}\,\,t_m\}.
$$

\begin{figure}
	\centering \includegraphics*{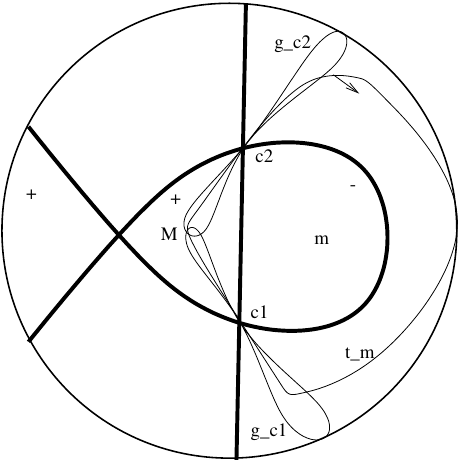}
	\caption{The vanishing cycle $\delta_M$ for a minimum.}
	\label{fig:delta_min}
\end{figure}

\subsection{Visualization of the Milnor fiber associated with a divide}
\label{ss:visual_fiber}
\index{Milnor fiber!visualization}
Now we explain how to visualize the Milnor fiber together with the divide. 
Think of the divide
as  a road network which has $\delta$ junctions, and replace
every junction by a
roundabout,
which leads you to a new road network with $4\delta$ T-junctions.
Realize now
every road section in between two T-junctions by a strip with a half
twist. Do the same for every road section in between a
T-junction and the boundary of the divide. Altogether you will need
$6\delta+r$ strips. Now, applying what we have seen in the previous 
Sect.~\ref{ss:visual_vanishing}, the core line of the four
strips of a roundabout is a black vanishing cycle, the strips
corresponding
to boundary edges and
corners of a positive or negative region have as core line a red or blue
vanishing cycle (see Fig.~\ref{fig:van_cycles_D5} or \ref{fig:van_cycles_comp} 
for color scheme).

\begin{figure}
	\tiny
	\centering
	\resizebox{0.9\textwidth}{!}{\centering 
		\includegraphics*[scale=0.6]{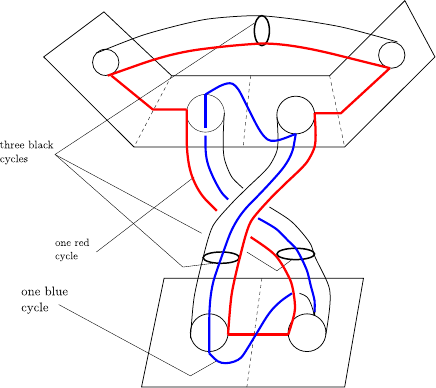}}
	\caption{Milnor fiber with vanishing cycles for $D_5.$}
	\label{fig:van_cycles_D5}
\end{figure}

In Fig.~\ref{fig:van_cycles_comp}  is worked out the singularity with two
Puiseux pairs and $\mu=16$, where we used a divide equivalent (up to an 
admissible isotopy) to that of 
Fig.~\ref{fig:29_23}. We have drawn
for convenience in
Fig.~\ref{fig:van_cycles_comp} only one red, black, or blue cycle.
We have also indicated the position of the arc $\alpha,$ which will
play a role in the next section.

\begin{figure}
	\tiny
	\centering
	\resizebox{0.8\textwidth}{!}{\centering 
		\includegraphics*{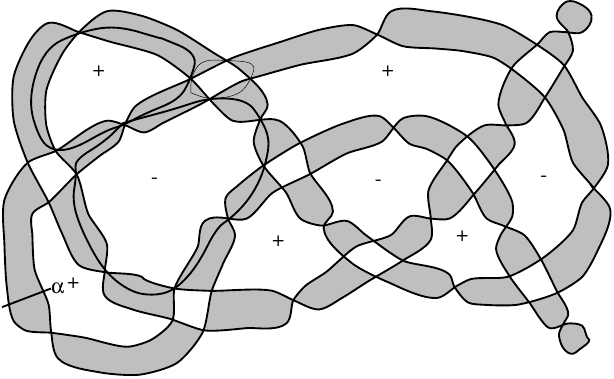}}
	\caption{ Milnor fiber with vanishing cycles for
		$y^4-2y^2x^3+x^6-x^7-4yx^5.$}
	\label{fig:van_cycles_comp}
\end{figure}

\bigskip
\subsection{Reduction cycles and reduction tori}
\label{ss:reduction_system}
We consider a divide of the form $Q:=P_{p,q}*P$, where the divide $P$ is
given by an immersion $\gamma:[-1,1] \to D$. So,
the divide $Q$ is the image of 
$N\gamma \circ T_{p,q}$. If we change the immersion
by a repara\-metrization $\gamma_1:=\gamma \circ \phi$, where $\phi$ is an
oriented diffeomorphism of $[-1,1]$, the divide 
$Q_1:= N\gamma_1 \circ \phi([-1,1])$ is isotopic to 
$Q$ by an transversal isotopy, which does not change 
the type of its knot. By choosing $\phi$ appropriately and $\eta$ small, 
one can achieve that to each double
point of $P$ correspond $p^2$ double points of the divide $Q_1$, 
which look like the intersection points of 
a system of $p$ almost parallel lines with an other system 
of $p$ almost parallel lines, see Fig.~\ref{fig:manhattan}.
\begin{figure}
	\centering \includegraphics*{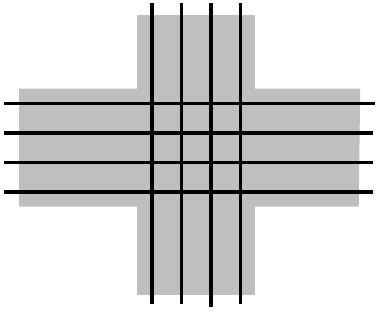}
	\caption{Manhattan: crossing of the box with $4$ by $4$ strands.}
	\label{fig:manhattan}
\end{figure}

We may assume, that the divide $Q$ for each double point of $P$
already has a grid of $p^2$ intersections. 

We will construct \index{reduction curves} reduction curves for the monodromy 
of the knot $L(Q)$
following the ideas of \cite{Acampo_sur_la}. The reduction curves of the 
cabling $P_{p,q}*_{\eta}P$
are the intersection of the fiber $F_{Q}$ over 
$1 \in S^1$ of the fibration on
the complement of the knot $L(Q)$ with the boundary of a regular tubular 
neighborhood $U$ of the closed tubular neighborhood $V$ of the knot $L(P)$
for which $L(Q) \subset \partial{V}$ holds. The intersection 
$(F_{Q} \cap \partial{U}) \subset F_{Q}$ is indeed 
a system of reduction curves
provided that the torus $\partial{U}$  is transversal to the fibration
of the knot $L(Q)$.

Assume that the tubular neighborhood $V$ was constructed with
the field $\Phi_{p,q}$ and a particular value of the parameter $\eta$.
The same field of sectors, but a slightly bigger parameter value
$\eta'$ yields a tubular neighborhood $U$ of $V$ in $S^3$. The construction 
of the fibration will be done as in Sect.~\ref{s:fibration_thm} (c.f. 
\cite{Acampo_generic_imm}). The main choice for the 
construction of the fibration for the knot $L(Q)$ is a carefuly chosen adapted 
Morse function (recall 
Definition~\ref{def:adapted_function})
$f_Q:D \to \mathbb{R}$ with $f_Q^{-1}(0)=Q$ . For  
our purpose here, where we must achieve the above
transversality so we will choose $f_Q$ as follows.
First, after applying a  regular transversal small 
isotopy, we may assume that the divide $P$  
has perpendicular rectilinear crossings. Next, we 
consider an adapted function $f_P:D \to \mathbb{R}$ for 
the divide $P$, that is euclidian near its crossings. Let the fibration 
on the complement of the knot $L(P)$ be
$\pi_{P,\eta}:S^3 \setminus L(P) \to S^1$ 
where we recall that
$\pi_{P}(x,u):=\theta_{P}(x,u)/ |\theta_{P}(x,u)|$
and
$$
\theta_{P,\lambda}(x,u):=f_P(x)+i \lambda^{-1}\, df_P(x)(u)-
{1\over 2}\lambda^{-2}\chi(x)H_{f_P}(x)(u,u).
$$
Here, the function $\chi:D \to \mathbb{R}$ is a bump function at the crossing 
points of $P$ and $\lambda$ is a big real parameter (compare with 
eq.~\ref{eq:theta_func}).
We now choose  a small positive real number $v$, such that 
$\{x \in D | |f_P(x)|\leq v\}$ is a regular tubular neighborhood of $P$, that
meets each component of $\{x\in D | \chi(x)=1\}$. 
Next 
we choose $\eta'>0$ such the corners of $N_{\eta'}\gamma(B)$ are in
$\{|f_P(x)|=v\}$ i.e. $\eta'^2=v$.  We constuct the torus knot $L(Q)$ with
$Q:=P_{p,q}*_{\eta}P$ where $0 < \eta < \eta'$. Since 
$Q \subset \{|f_P(x)| < v\}$ holds, we can construct 
a Morse function
$f_Q:D \to \mathbb{R}$ for the divide $Q$, such 
that on $\{|f_P(x)| \geq v\}$ the function $f_Q$
is constant along the level sets of $f_P$. 

The following 
theorem follows directly from  \cite[Lemme $2$, pg. $153$]{Acampo_sur_la} and 
the above construction.

\begin{theorem} \label{thm:reduction_system} The torus 
	$\partial{\Phi_{\eta',p,q}}\gamma$ 
	is transversal to the fibration $\pi_Q$
	induced by $f_Q$ on the complement of the knot $L(Q)$. The intersection 
	$$
	\partial{\Phi_{\eta',p,q}}\gamma \cap F_{Q}
	$$ 
	is a system of $p$ closed  curves on the fiber $F_{Q}$, which 
	is a reduction of the monodromy of $L(Q)$.
\end{theorem}

With a few examples, we now explain how to inductively depict in the Milnor 
fiber
a distinguished system of vanishing cycles and the
reduction curves for the monodromy of a singularity,
for which a divide of the form $Q=P_{p,q}*P$ is given.

For a $(p,q)$ cabling ($p < q$) the reduction 
system consists of $p$ simply closed curves on the fiber $F_{Q}$. Each 
of them cuts out from $F_{Q}$ a surface diffeomorphic to the fiber 
$F_{P}$ of the divide $P$. The $p$ copies of $F_{P}$ in $F_{Q}$ 
are cyclicly permuted by the monodromy $T_Q$. This description follows from the 
construction of the monodromy of an irreducible plane curve singularity given 
in \cite{Acampo_sur_la}.

One of those copies can be visualized more easily as follows.
Let $\{x \in D | f_P(x) \geq v \}$. For each double point $c$ of $P$
we connect the two components of $\{x \in D | f_P(x) \geq v \}$,
that are incident with the double point $c$, by a special bridge  
which projects diagonally through
the Manhattan part  of the divide $Q$,
that corresponds to $c$. The
projection of the bridge is  a twisted strip $S_c$ in $D$, that 
realizes a boundary connected sum of
the $P_+$-components (see Fig.~\ref{fig:bridge}). The twist points of the strip 
$S_c$ are precisely
the critical points of $f_Q$, that lie on the diagonal. The boundary
of $S_c$ consists of two smooth curves, that intersect
each other transversally and that also intersect the divide $Q$ 
transversally.

\begin{figure}
	\tiny
	\centering \resizebox{0.7\textwidth}{!}{\centering 
		\includegraphics*{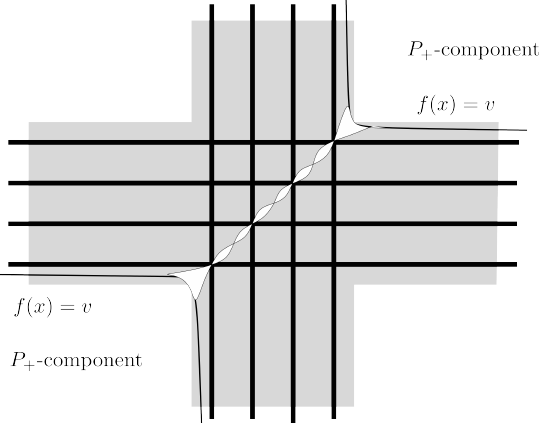}}
	\caption{Bridge through Manhattan.}
	\label{fig:bridge}
\end{figure}

Let $C$ be the union
of the projections $S_c$ of the 
bridges with $\{x \in D | f_P(x) \geq v \}$, see Fig.~\ref{fig:bridge}. 
In Fig.~\ref{fig:bridge_block} we have zoomed
out one block to  show more details.
The copy $F_{P,Q}$ 
of the fiber of the knot $L(P)$ is the closure  in the fiber $F_Q$
of the knot $L(Q)$ of the set
$$
\{(x,u) \in TD | x \in C, (df_Q)_x\not=0, (df_Q)_x(u)=0\}\cap S^3.
$$

The first
reduction curve   is the 
boundary 
\[
R:=\partial{F_{P,Q}}
\] 
of the surface
$F_{P,Q}\subset F_Q$. The reduction system
is the orbit 
\[
\{R, T_Q(R), T_Q^2(R), \dots \}
\]
under the monodromy $T_Q$ of the singularity with divide $Q$.

\begin{figure}
	\centering \includegraphics*{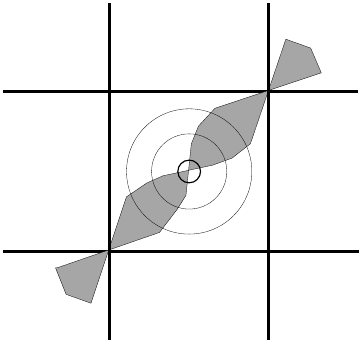}
	\caption{Bridge through a block of Manhattan.}
	\label{fig:bridge_block}
\end{figure}

\begin{example} Our first example is the singularity with two essential Puiseux 
	pairs
	$(x^3-y^2)^2-4x^5y-x^7$, whose link is a two stage iterated torus knot. For 
	a 
	detailed description of the link, the monodromy and the Milnor fiber of 
	this 
	singularity, follow the construction of \cite{Acampo_sur_la}. Its
	divide $Q=P_{2,9}*P_{2,3}$, (see Fig.~\ref{fig:29_23}), has two 
	$P_+$-components, 
	where
	$P=P_{2,3}$ 
	(recall Fig.~\ref{fig:29_23_red} and contents of that section), where the 
	projection of the reduction curve $R$ is
	drawn).  In this case Manhattan consists of one block. 
	The reduction curve $R$ is the pre-image in the fiber 
	$F_Q$ of its projection $\mathrm{proj}(R) \subset D$ under
	the map $(x,u) \mapsto x$ (see Fig.~\ref{fig:29_23_red}). That means $R$ 
	is the closure in $F_Q$ of the set
	$$
	\{(x,u) \in TD | x \in \mathrm{proj}(R), (df_Q)_x\not=0,(df_Q)_x(u)=0, 
	\|x\|^2+\|u\|^2=1\}.
	$$
	The reduction system is $\{R, T_Q(R)\}$.
	
	\begin{figure}
		\centering \includegraphics*{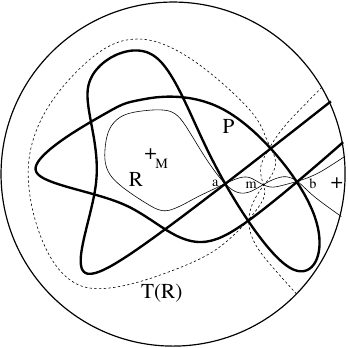}
		\caption{Divide $Q=P_{2,9}*P_{2,3}$ with reduction curves $R$ and 
			$T(R)$ (dotted).}
		\label{fig:29_23_red}
	\end{figure}
	
	The curve $R$ is homologically trivial in $F_Q$. It turns out
	that the power $T_Q^{156}$ of the monodromy is the composition of the
	right Dehn twists, whose core curves are  
	$\{R, T_Q(R)\}$ . The power
	$T_Q^{156}$ is a product of $2496=16 \times 156$ Dehn 
	twists, since $T_Q$ is the
	product of those Dehn twists whose core 
	curves are the system of distinguished
	quadratic vanishing cycles of the real morsification with divide $Q$. 
	It   turns out that the expression as product
	of Dehn twists is far from being as short as possible. In fact, the right 
	Dehn
	twist $\Delta_R$ with core curve $R$ can be 
	written as a product of $36$ right Dehn twists
	that have core curves coming from the morsification with divide $Q$. 
	More precisely, the Dehn twist $\Delta_R$ factors as
	$$
	\Delta_R=(\Delta_M \circ \Delta_b \circ \Delta_m \circ \Delta_a \circ 
	\Delta_m^{-1} \circ \Delta_b^{-1})^6,
	$$
	The factors are right Dehn twists whose core curves are 
	among the quadratic 
	vanishing cycles 
	$\delta_m,\delta_a,\delta_M,\delta_b$ of the 
	divide $Q$ as indicated in Fig.~\ref{fig:bridge_block}, 
	$\delta_M$ is the
	vanishing cycle of a $P_+$-region, $\delta_m$ of the 
	maximum of Manhattan, and $\delta_a,\delta_b$
	of street corners of Manhattan. It follows 
	that $T_Q^{156}$ can also be written 
	as a composition of $72$ Dehn twists with core curves 
	among the vanishing cycles of  the divide $Q$. 
	The 
	composition $\Delta_b \circ \Delta_m \circ \Delta_a \circ \Delta_m^{-1} 
	\circ 
	\Delta_b^{-1}$ is the Dehn 
	twist with core curve $\bar{a}:=\Delta_b(\Delta_m(a))$.
	
	The reduction curve $R$ cuts off from $F_Q$ a piece $F_{P,Q}$ of genus 
	one, which corresponds to a copy of the Milnor fiber of the Brieskorn-Pham 
	singularity $y^2+x^3$ (corresponding to the first Puiseux pair). 
	The Dehn
	twists $\Delta_M$ and $\Delta_{\bar{a}}$ act only on this piece, since
	the curves $\delta_M$ and $\bar{a}$ lie entirely in this piece; in
	this piece, that is a copy of the fiber $F_P$, they generate
	the geometric monodromy group of the accompanying
	singularity $x^3-y^2=0$ with divide $P_{2,3}$.
\end{example}

\begin{example}
	Our second example is the {\em double cusp}, that is, the singularity with 
	two 
	branches $(x^3-y^2)(y^3-x^2)$.
	Its homological monodromy is of infinite order by \cite{Acampo_sur_la}.
	Each branch is a torus knot. Again Manhattan consists of 
	one block. In Fig.~\ref{fig:double_cusp}
	we have drawn the projections of the curves 
	$R,R'$ and $S,S'$, that together are the boundary components of the 
	two diagonals through Manhattan. In this case the curves $R$ and $R'$ are
	isotopic to each other, as are the curves $S$ and $S'$. A complete reduction
	system for the geometric monodromy is the system $\{R,S\}$.
	Each component
	of this system carries a non-trivial homology class. The isotopy classes
	of the curves $R$ and $S$ are permuted by the monodromy $T_P$, hence the 
	system
	$\{R,S\}$ is invariant under the monodromy. 
	
	\begin{figure}
		\centering \includegraphics*{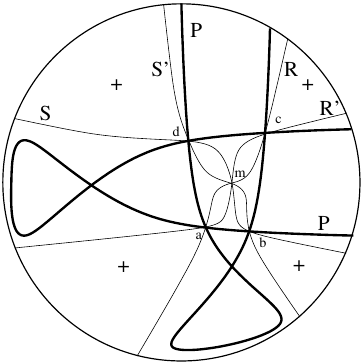}
		\caption{Divide $P$ for $(x^3-y^2)(y^3-x^2)$ with reduction system  
			$R \cup S$.}
		\label{fig:double_cusp}
	\end{figure}
	
	Let $h$ be the 
	action of $T_P$  on
	the homology $H_1(F_P,\mathbb{Z})$ of the the fiber $F_P$. Let
	$\delta_a, \delta_b, \delta_c, \delta_d$ be the 
	vanishing cycles of the double points, that are the
	corners of Manhattan of $P$, let $\delta_m$ be the vanishing cycle of the
	maximum in the center of Manhattan.
	
	If one chooses the orientations appropriately, on has 
	$$
	[R]=[\delta_a]+[\delta_m]+[\delta_c],\, 
	[S]=[\delta_b]+[\delta_m]+[\delta_d],\, 
	h([R])=-[S],\, h([S])=-[R],$$ 
	hence also $h([R]-[S])=[R]-[S]$. Let $[k]$ be any cycle on $F_P$, that
	is carried by a simple oriented curve $k$ and intersects the curves $R$ and 
	$S$
	each transversally in one point. One has $h^{10}([k])=[k]\pm([R]+[S])$,
	which shows that the homological monodromy $h$ is not of finite order. We
	have drawn in Fig.~\ref{fig:double_cusp} the oriented projection of such a 
	cycle 
	$k$, that
	intersects the curves $A$ and $B$. 
	The curve $A$ is halfway in between the curves $R$ and
	$R'$ on the cylinder they cut out. Let $B$ be the curve halfway in between
	$S$ and $S'$. The curves $A$ and $B$ are the reduction curves of 
	Figure $3$
	on page $167$ of \cite{Acampo_sur_la}. The reduction system $A,B$ is much 
	easier to
	draw, see Fig.~\ref{fig:double_cusp_red}, where are drawn the 
	projections 
	in 
	$D$. 
	The projections
	meet transversally at the maximum in Manhattan of $f_P$.
	The curve
	$\delta_m$ intersects 
	transversally in two points each curve $R$ and $S$. One has 
	$h^{10}([\delta_m])=[\delta_m]\pm 2([R]+[S])$.
	
	\begin{figure}
		\centering \includegraphics*{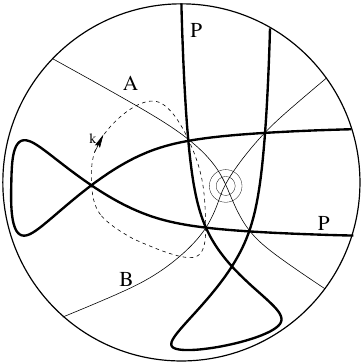}
		\caption{Divide $P$ for $(x^3-y^2)(y^3-x^2)$
			with reduction system  $A \cup B$.}
		\label{fig:double_cusp_red}
	\end{figure}
	
	The power $T_P^{10}$
	of the geometric 
	monodromy, that is a word of length 
	$110$ in the Dehn twists of the divide $P$,
	is equal to the composition of those right Dehn
	twists, whose core curves are  $R$ and $S$. So, the power $T_P^{10}$
	also can be written as  the much shorter word
	$$
	\Delta_c \circ \Delta_m \circ \Delta_a \circ \Delta_m^{-1}\circ 
	\Delta_c^{-1}\circ 
	\Delta_d \circ \Delta_m \circ \Delta_b \circ \Delta_m^{-1} \circ 
	\Delta_d^{-1}.
	$$
	
\end{example}

\begin{remark} The curves $A$ and
	$\Delta_m(\Delta_c(\delta_a))$ are isotopic, where 
	$\Delta_m$ and $\Delta_c$ are the right Dehn
	twists with core curves $\delta_m$ and $\delta_c$. It 
	follows that the reduction system
	$A,B$ consists  of quadratic vanishing cycles
	of the singularity of $\{(x^3-y^2)(y^3-x^2)=0\}$ with two branches. 
	In contrast, a  reduction curve of a 
	singularity with only one branch can not be   a quadratic vanishing cycle,
	since all reduction curves are zero in the homology.  
\end{remark}

\subsection{Geometric monodromy group and reduction system}

Let the polynomial $f_{(a,b)}$ be an equation for an irreducible
plane curve singularity with $n$ 
essential Puiseux pairs $(a_i,b_i)_{1\leq i \leq n}$ with $a_i < b_i$. The 
number of 
simply closed curves contained in a complete reduction
system $R$ for the monodromy of $f$ is 
$$
a_na_{n-1}\cdot\, \dots \,\cdot a_2+a_{n}a_{n-1}\cdot \, \dots \,\cdot a_3+ 
\dots +a_na_{n-1}+a_n.
$$
Indeed, this follows from the construction of the Milnor fiber in 
\cite{Acampo_sur_la}.
Let 
\[
\Gamma_{f,red} < \Gamma_f
\]
be the subgroup of the 
geometric monodromy group (recall Definition~\ref{def:geom_monodromy_group})
of $\Gamma_f$ of $f$ of those elements 
$\gamma \in \Gamma_f$ that up to isotopy fix
each component of $R$. Let $\Gamma_{f,red}^0$ be the subgroup of $\Gamma_f$
which is generated by the Dehn twist whose core curves are
quadratic vanishing cycles and do not intersect any component of $R$.
Obviously, one has $\Gamma_{f,red}^0 \subset \Gamma_{f,red}$, but 
we do not know if this inclusion is strict. A component of $F \setminus R$ is
called a top-component if its closure in $F$ meets only one
component of $R$. Let $\Gamma_{f,top}$ be 
the subgroup of $\Gamma_{f}$ of those monodromy transformations,
which induce the identity in each component of $F \setminus R$ that 
is not a top-component. 
Let $\Gamma_{f,top}^0$ be the intersection 
$\Gamma_{f,top}\cap \Gamma_{f,red}^0$. We have

\begin{theorem}\label{thm:geometric_mon_prod} Let $f=f_{(a,b)}$ be an 
	irreducible 
	singularity with
	$n \geq 2$ essential Puiseux pairs $(a_i,b_i)_{1\leq i \leq n}$. 
	Let $g=f_{(a',b')}$ be a singularity with 
	the $n-1$ essential Puiseux pairs
	$(a',b')=(a_i,b_i)_{1\leq i \leq n-1}$. The group 
	$\Gamma_f$ contains the product
	of $a_n$ copies of the group $\Gamma_g$.
\end{theorem}

\begin{theorem} \label{thm:isom_geom_group} Let $f_{(a,b)}$ be an irreducible 
	singularity with
	$n \geq 2$ essential Pui\-seux pairs. The group $\Gamma_{f,top}^0$ is 
	isomorphic
	to the product of $a_na_{n-1}\cdot\, \dots \,\cdot a_2$ copies of the
	geometric monodromy group of the
	singularity $y^{a_1}-x^{b_1}=0$.
\end{theorem}

\begin{proof}[Proof of Theorem~\ref{thm:geometric_mon_prod}] Let $P$ be the 
	divide 
	$P_{a_{n-1},b'_{n-1}}* \dots *P_{a_2,b'_2}*P_{a_1,b_1}$ for the singularity
	of $g$ and let $Q=P_{a_n,b'_n}*P$ be the divide for the singularity of $f$.
	A copy $F_{P,Q}$ of the fiber $F_P$ is constructed as a subset of the fiber
	$F_Q$. Remember, that $F_P$ is obtained by connecting with strips
	the sets $\{(x,u)\in TD | f_P(x)>0, (df_P)_x(u)=0\}$, where $f_P:D \to 
	\mathbb{R}$
	is a Morse function for the divide $P$. For each double point of $P$
	there are two connecting strips. To each $+$-component of $P$ corresponds a
	$P_+$-component of $Q$ with the same topology  
	and to each double point of $P$ corresponds a Manhattan
	grid of $Q$, in which we have drawn diagonally the projection of the strips
	that connect $\{(x,u)\in TD | x \in Q_{P,+}, (df_Q)_x(u)=0\}$. Here,
	$Q_{P,+}$ denotes the union of the $P_+$-components  of the 
	complement of the divide $Q$. From the divide $P$ is deduced a distinguished
	basis of quadratic 
	vanishing cycles for the singularity of $f$. Let 
	$B_P$ be the union of the curves of this basis.  This basis can be drawn
	on the fiber $F_P$, see Sect.~\ref{ss:visual_vanishing} for an algorithmic 
	method 
	to 
	visualize distinguished basis of vanishing cycles. 
	
	In order to prove the theorem, we will 
	construct inside $F_{P,Q}$ a system
	of simply closed curves with union $B_{P,Q}$, each of them being a 
	quadratic vanishing cycle for
	the singularity $g$, such that the pairs $(F_P,B_P)$ and 
	$(F_{P,Q},B_{P,Q})$ are diffeomorphic. This finishes the proof,
	since the Dehn twist, whose cores are the quadratic vanishing cycles
	of $B_{P,Q}$, generate a copy of $\Gamma_g$ in $\Gamma_f$. By acting with 
	the 
	geometric monodromy $T$ of the singularity $f$, one obtains $a_n$ commuting
	copies of $\Gamma_g$ in $\Gamma_f$.
	
	To each $+$-region of $P$ corresponds one $P_+$-region of $Q$. The maximum
	of $f_P$, say at $M$ in the region, is also a maximum of $f_Q$. The 
	quadratic
	vanishing cycle $\delta_M:=\{(M,u)\in TD | \|M\|^2+\|u\|^2=1\}$ of $F_Q$ 
	lies in $F_P$
	and also in $F_{P,Q}$.  For each double point $c$ of $P$ the quadratic
	vanishing cycle $\delta_c \subset F_P$ projects in $D$ to a tear splicing
	the
	gradient line $g_c$ of $f_P$ through $c$. The endpoints of $g_c$ 
	are maxima of $f_P$ or points on $\partial{D}$. The function $f_Q$ 
	has exactly one gradient line $g_{Q,c}$ that has the same endpoints
	as $g_c$ and coincides with $g_c$ 
	in a neighborhood of the common endpoints. The gradient line $g_{Q,c}$
	runs along a diagonal through the Manhattan grid corresponding to $c$. Let
	$g_{Q,c}$ be the simply closed curve on $F_Q$, that projects to a tear
	$t_{Q,c}$ equal to $g_{Q,c}$, except  above a neighborhood of its endpoints
	where $t_{Q,c}$ equals $t_c$. 
	We remark that $\delta_{Q,c}$ is a cycle in $F_{P,Q}$. Let
	$c_1, \dots ,c_p$ be the $p:=a_n$ double points of $Q$ that occur along the 
	$g_{Q,c}$ and let $M_2, \dots , M_p$ along $g_{Q,c}$ be 
	the maxima.
	Let $\delta_{Q,c_1}$ be the quadratic vanishing cycle of the singularity $f$
	that corresponds to $c_1$. One verifies that the cycles
	$\delta_{Q,c}$ and 
	$$
	\Delta_{c_p}\circ \Delta_{M_p} \circ \dots \circ \Delta_{c_2}\circ 
	\Delta_{M_2}(\delta_{Q,c_1})
	$$
	are isotopic. Here $\Delta_{c_i}$ or $\Delta_{M_i}$ stands for 
	the right Dehn twist of $F_Q$ whose
	core curve is the quadratic vanishing cycle 
	$\delta_{c_i}$ or $\delta_{M_i}$
	of the singularity $f$. 
	Hence $\delta_{Q,c} \subset F_{P,Q}$ is a quadratic
	vanishing cycle for the singularity $f$. So far, we have constructed for
	each maximum and for each saddle point of $f_P$ a simply closed curve on
	$F_{P,Q}$ that is a quadratic vanishing cycle of the singularity $f$. These
	cycles intersect on $F_{P,Q}$ as do the corresponding quadratic vanishing 
	cycles of the singularity $g$ on $F_P$. 
	
	\begin{figure}
		\centering \includegraphics*{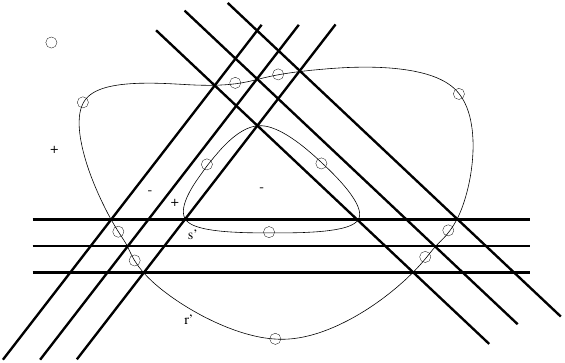}
		\caption{Vanishing cycle $s$ on $F_Q$ from a minimum of 
			$f_P$ and cycle $r$ on $F_{P,Q}$.}
		\label{fig:min_neg}
	\end{figure}
	
	We now wish to construct for each minimum of $f_P$ a vanishing 
	cycle on $F_{P,Q}$. We have to handle two cases: $p$ odd, see 
	Fig.~\ref{fig:min_neg}, and
	$p$ even, see Fig.~\ref{fig:min_pos}.
	
	If $p$ is odd, a minimum $m$ of $f_P$ will 
	also be a minimum of $f_Q$. Let  $\delta_{Q,m}$ 
	be the vanishing cycle on $F_Q$ corresponding to $m$, see 
	Fig.~\ref{fig:min_neg}. 
	The 
	projection of $\delta_{Q,m}$ into $D$ is a smooth simply closed 
	curve $s'$ transversal to $Q$, that 
	surrounds the $-$ region of $m$ through
	the its neighboring $+$ regions of $Q$. One needs to take care 
	that in each neighboring
	$+$ component the projection runs through the maximum of $f_Q$ in that 
	region.
	The points of $s$ correspond to pairs $(x,u)$ with $x\in s'$ and $u$ 
	pointing
	inwards to $m$. Let $r$ be a simply closed cycle on $F_{P,Q}$ 
	that projects into
	$D$ upon the curve $r'$, which now surrounds the $-$ region of $m$ 
	through
	the $P_+$-components of $Q$, see Fig.~\ref{fig:min_neg}. In the 
	Manhattan grids $r'$ ist just
	a diagonal, again $r'$ runs through the maxima of the regions or 
	touches $\partial{D}$. On $r$ we 
	only allow pairs $(x,u)$ where $u$ points inwards to $m$. It is clear that 
	the cycle $r$ on $F_{P,Q}$ intersects the cycles of the previous
	construction as the vanishing cycle to the minimum of $f_P$
	intersects the vanishing cycles of the critical points of $f_P$.
	It remains however to check that the cycle $r$ is a quadratic
	vanishing cycle of the singularity of $g$. By 
	applying to $r \subset F_Q$ 
	the Dehn twist corresponding to the
	critical points of $f_Q$ that are in between the curves $r'$ and $s'$,
	one can transform the isotopy class of the curve $\delta_{Q,m}$ 
	to the class of the curve
	$r$. This proves that $r$ is indeed a quadratic vanishing cycle of the
	singularity of $g$.
	
	\begin{figure}
		\centering \includegraphics*{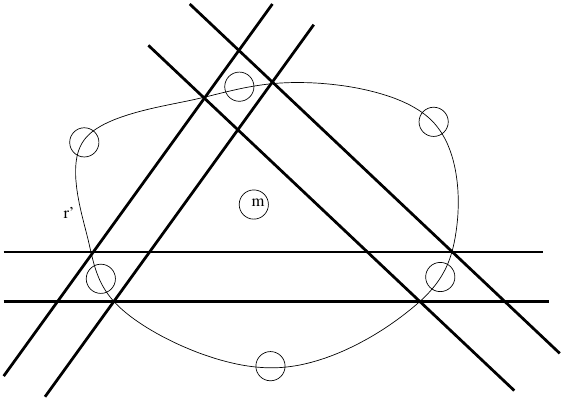}
		\caption{Vanishing cycle $\delta_{Q,m}$ on $F_Q$ from a minimum of
			$f_P$ and cycle $r$ on $F_{P,Q}$.}
		\label{fig:min_pos}
	\end{figure} 
	
	If $p$ is even, a minimum $m$ of $f_P$ will
	be a maximum of $f_Q$. Let  $\delta_{Q,M}$
	be the vanishing cycle on $F_Q$ corresponding to maximum $M:=m$, 
	see Fig.~\ref{fig:min_pos}. Its 
	projection into $D$ is the point $M:=m$. Let $r$ be a simply 
	closed cycle on $F_{P,Q}$
	that projects into
	$D$ upon the curve $r'$ which now surrounds the $-$ region of $M:=m$
	through
	the $P_+$-regions of $Q$, see Fig.~\ref{fig:min_pos}. In the Manhattan 
	grids 
	$r'$ 
	ist just
	a diagonal, again $r'$ runs through the maxima of the regions. On $r$ we
	only allow pairs $(x,u)$ where $u$ points inwards to $m$. It is clear that
	the cycle $r$ on $F_{P,Q}$ intersects the cycles of the previous
	construction as the vanishing cycle to the minimum of $f_P$
	intersects the vanishing cycles of the critical points of $f_P$.
	By
	applying to $\delta_{Q,M} \subset F_Q$
	the Dehn twist corresponding to the
	critical points of $f_Q$ that are in between the curve $r'$ and 
	the point $M:=m$,
	one can transform the isotopy class of the curve $\delta_{Q,M}$ to the 
	class of 
	the curve
	$r$, and proves that $r$ is indeed a quadratic vanishing cycle of the
	singularity of $g$. As explained, this terminates the proof.
\end{proof}

\begin{proof}[Proof of Theorem~\ref{thm:isom_geom_group}]
	The proof of Theorem~\ref{thm:geometric_mon_prod} constructs a copy 
	$\Gamma_{P,Q}$ of the
	mono\-dromy group $\Gamma_f$ of the singulatity $f$ as subgroup 
	in the monodromy group $\Gamma_g$ of the
	singularity $g$. This copy acts with support in a copy $F_{P,Q}$ 
	of the fiber $F_P$. The the first $a_n-1$ iterates of the 
	monodromy $T_Q$ of the singularity $g$  constructs $a_n$ copies of $F_P$
	in $F_Q$. By conjugation with $T_Q$ one gets $a_n$ copies from 
	$\Gamma_{P,Q}$.
	We end the proof by repeating this argument. One 
	gets $a_na_{n-1} \cdots a_2$ commuting copies 
	of the geometric monodromy group of the 
	singularity $x^{b_1}-y^{a_1}$ in $\Gamma_g$.
\end{proof}

\section{Some other questions and properties}
\label{s:other_questions}

We conclude by stating some other open questions and properties related to 
divides associated with plane curve singularities.

\subsection{Connected sum and characterization of plane curve divides}
The connected sum of two divides $(D_1,P_1)$ and $(D_2,P_2)$ is done by 
making a boundary connected sum of $D_1$ and $D_2$ such that a boundary 
point of $P_1$ matches with a boundary point of $P_2.$ For 
divides 
with one branch we have the
formula:

$$
L(P_1\#P_2)=L(P_1)\#L(P_2)
$$

A  relative immersion $i:I \to D$ of a copy of $[0,1]$ in $D,$ such 
that at selftangencies the
velocities are with opposite orientation, defines an embedded and 
oriented arc $I'$ in $S^3$ by putting:

$$
I':=\{(x,u) \in S^3 | x \in i(I), (di^{-1})_x(u) \geq 0\}
$$

Let $j:I \to D$ be a relative immersion with only transversal crossings 
and opposite selftangencies, such that
the endpoints of $i$ and $j$ are tangent with opposite orientations and that 
all tangencies of $i$ and $j$ are generic and have opposite orientations. 
The union $I' \cup J'$
is the oriented knot of the pair $(i,j).$ A divide $P$ 
defines pairs $(i_P,j_P)$ of relative immersions 
with opposite orientations by taking both orientations. Those pairs 
$(i_P,j_P)$ have a special $2$-fold symmetry. For instance 
the complex conjugation 
realizes this $2$-fold 
symmetry for a  divide, which arises as a real deformation of a real
plane curve singularity. It is in\-teresting to observe that this symmetry 
acts on $F_1$ with as
fixed point set the intersection $D\cap F_1$, which is a collection of $r$
disjointly embedded arcs in $F_1.$ The quotient of $\bar{F_1}$ by the symmetry
is an orbifold surface with exactly $2r$ boundary 
${\pi \over 2}$-singularities. 
Any link of singularity of a plane curve can be obtained as the link of a 
divide by Sect.~\ref{s:divides_plane_curves} (c.f. \cite{Acampo_groupe_I, 
	Acampo_real_def,GZ_inter_two}). It is an 
interesting problem to characterize links of 
singularities among links of divides.

\subsection{Symplectic properties}
The link of a divide is transversal to the standard contact structure in the
$3$-sphere. This can be seen explicitly by the following computation, where
we use the multiplication of quaternions. Let $P$ be a divide in the unit
disk. We assume that the part of $P$, which lies in the collar of
$\partial{D}$  
with inner radius $1 \over \sqrt{2}$,
consists of radial line segments. We think of 
the branches of $P$ as parametrized
curves $\gamma(t)=(a(t),b(t)),-A \leq t \leq A,$ 
where the parameter speed is adjusted such that
$a^2+b^2+\dot{a}^2+\dot{b}^2=1.$ To the branch $\gamma$ correspond two arcs
$\Gamma^+$ and $\Gamma^-$ on the sphere of quaternions of unit length:

$$
\Gamma^+(t):=a(t)-\dot{a}(t)i+b(t)j+\dot{b}(t)k
$$

$$
\Gamma^-(t):=a(-t)+\dot{a}(-t)i+b(t)j-\dot{b}(-t)k
$$

The left invariant speed of $\Gamma^+$ at time $t$ is

$$
v(t):=\Gamma^+(t)^{-1}{d \over dt}\Gamma^+(t)
$$

We have

$$
v=a\dot{a}+\dot{a}\ddot{a}+b\dot{b}+\dot{b}\ddot{b}+
[-a\ddot{a}+\dot{a}^2-b\ddot{b}+\dot{b}^2]i+v_jj+v_kk
$$
The coefficient 
$v_0:=a\dot{a}+\dot{a}\ddot{a}+b\dot{b}+\dot{b}\ddot{b}$ vanishes, 
since ${d\over dt}\Gamma(t)$ is perpendicular to $\Gamma(t),$ and hence we 
can
rewrite the coefficient $v_i$ of $i$ in $v$ as

$$
v_i=-a\ddot{a}+\dot{a}^2-b\ddot{b}+\dot{b}^2=<(a+\dot{a},b+\dot{b})\mid
(\dot{a},\dot{b})>
$$

Outside of the collar neighborhood of $\partial{D}$ 
we have $v_i > 0,$ since $a^2+b^2 < 1/2 <
\dot{a}^2+\dot{b}^2.$ In the collar we also have $v_i > 0$ by a direct
computation. Since the left
invariant contact structure on the unit sphere in the skew field 
of the quaternions 
is given by the span of the tangent vectors 
$j$ and $k$ at the point $1$, we conclude that $\Gamma^{+}$ with its
orientation is in the positive sense transversal to the left invariant 
contact structure $S^3$. 

For the link of a divide we now will construct 
a polynomial, hence symplectic, spanning surface in the
$4$-ball. For $\lambda \in \mathbb{R}, \lambda > 0$ put 

$$
B_{\lambda}:=\{p+ui \in \mathbb{C}^2 | p,u \in \mathbb{R}^2,\|p\|^2 + 
\lambda^{-2} \|u\| 
\leq 1\}
$$

We have $B_{\lambda} \cap \mathbb{R}^2 = D$ and $B_{\lambda}$ is a strictly
holomorphically convex domain with smooth boundary in $\mathbb{C}^2$. The map 
$(p,u) \mapsto ((p,u/{\lambda})$ identifies $\partial{B_{\lambda}}$ with
the unit $3$-sphere of $\mathbb{C}^2$.

\begin{theorem}
	Let $P$ be a connected divide in the disc $D$ with $\delta$ double points 
	and
	$r$ branches.
	There exist $\lambda > 0, \eta  > 0$
	and
	there exists a polynomial function $F:B_{\lambda} \to \mathbb{C}$ with the 
	following
	properties:
	
	\begin{enumerate}[label=(\alph*)]
		\item the function $F$ is real, i.e. 
		$F(\overline{p+ui})=\overline{F(p+ui)},$
		
		\item the set $P_0:=\{p \in D | F(p)=0 \}$ is a 
		divide, which is $C^1$ close to
		the divide P, and hence the divides $P$ and $P_0$ are combinatorially
		equivalent,
		
		\item the function $F$ has only non degenerate 
		singularities, which are all real,
		
		\item for all 
		$t \in \mathbb{C},\ |t| \leq \eta$ the intersection 
		$K_t:=\{(p+iu) \in B_{\lambda} |  F(p+iu)=t \} \cap 
		\partial{B_{\lambda}}$ 
		is transversal and by a small isotopy
		equivalent to the link $L(P)$,
		
		\item for the link $K_{\eta}$ the surface 
		$\{(p+iu) \in B_{\lambda} | F(p+iu)=\eta\}$ 
		is  a connected  smooth
		symplectic spanning surface of genus $\delta-r+1$ in the $4$-ball.
	\end{enumerate}
	
\end{theorem}

\begin{proof}
	Let the divide $P$ be given by smooth
	parametrized curves
	$\gamma_l:[-1,1] \to \mathbb{R}^2, 1\leq l \leq r$. Using the 
	Weierstrass Approximation Theorem, we can
	construct polynomial approximations $\gamma_{l,0}:[-1,1] \to \mathbb{R}^2$ 
	being 
	$C^2$ close to $\gamma_l$ and henceforth give a divide $P_0$ 
	with the combinatorics of the divide $P.$ We may choose $\gamma_{l,0}$ such 
	that
	$\gamma_{l,0}(s) \notin D, |s| > 1.$ 
	Let 
	$F:\mathbb{C}^2 \to \mathbb{C}$ be a real polynomial map such $F=0$ is a 
	regular
	equation for the union of the images of $\gamma_{l,0}.$
	Let $S_{\lambda}^3$ be the sphere 
	$S_{\lambda}^3:=\{p+iu \in \mathbb{C}^2 | \|p\|^2 + 
	\lambda^{-2}\|u\|^2=1\}.$ For a
	sufficiently small $\lambda > 0 $ we have that
	the $0$-level of $F$ on $S_{\lambda}^3$ is a model for the link
	$L(P).$ For $t \in \mathbb{C}, t \not=0,$ and $t$ sufficiently small, say 
	$|t| \leq
	\eta$, the surface
	$X_t:=\{p+iu \in B_{\lambda}^4 | F(p+iu)=t \}$ 
	is connected and smooth of genus $\delta-r+1$, and 
	has a polynomial equation, hence is a symplectic surface in the $4$-ball
	$B_{\lambda}$ 
	equipped with the standard symplectic structure of $\mathbb{C}^2.$
	The intersection $K_{\eta}:=X_{\eta} \cap \partial{B_{\lambda}}$ is also a 
	model for the link $L(P)$  and has hence a symplectic filling with the
	required properties.
\end{proof}

\begin{remark} Unfortunately, it is not the case that the 
	restriction of 
	$F$ to $B_{\lambda}$ is a 
	fibration with only
	quadratic singularities, such that for some $\eta > 1$ the fibers
	$f_0^{-1}(t), t \in \mathbb{C}, |t| < \eta,$ are transversal to the 
	boundary of
	$B_{\lambda}.$ So, we do not know, as it is the case for divides coming from
	plane curve singularities, if it is possible to fill in with a
	Picard-Lefschetz fibration, which is compatible with the contact and
	symplectic structure.
\end{remark}

\subsection{Presentation of the geometric monodromy group} We like to 
state the problem of presenting
the geometric
monodromy group of plane curve  singularities with generators
and relations. It 
would be particulary nice to express
the presentation in terms of a divide of the singularity. The same
problem can also be stated for the homological monodromy group of
plane curve singularities, but we think that the problem for the
geometric monodromy group is more tractable, since all reduction curves
can be taken into account 
Theorems~\ref{thm:geometric_mon_prod} and \ref{thm:isom_geom_group}. 
However, an important missing 
piece in this program is a presentation
with generators and relations of the geometric monodromy group of the
singularities $y^p-x^q=0$ for $3\leq p \leq q,\, 7 \leq p+q$. The fundamental
group of the complement of the discriminant in the unfolding of
the singularity $y^2-x^q=0$ is the braid group $B_{q-1}$. Bernard Perron and 
Jean-Pierre Vannier \cite{Perron_Vann} have proved for the singularities 
$y^2-x^q=0$ that
the geometric monodromy group is a faithful image of the braid group $B_{q-1}$
and that a similar result holds for the $D$ singularities $x(y^2-x^q)=0$. 
The fundamental
group of the complement of the discriminant in the unfolding of
the singularity $y^3-x^6=0$ is the Artin $AE_6$ group 
of the Dynkin diagramm $E_6$.
Bronislaw Wajnryb \cite{Waj} has proved that the geometric 
monodromy representation of $E_6$ into the mapping class group 
of the Milnor fiber of the singularity $y^3-x^6=0$ is not faithful. Recently, 
Nick Salter and the second author \cite[Corollary C]{Sal} have proven that the 
geometric monodromy representation of a plane curve singularity that is not of 
type $A_n$ or $D_n$ is never faithful if the genus of the Milnor fiber is at 
least $7$.

\bibliographystyle{alpha}
\bibliography{bibliography}

\printindex
\end{document}